\documentclass{article}
\usepackage[a4paper,margin=25mm]{geometry}
\usepackage{amsmath,amssymb}
\usepackage{tikz,tikz-cd}
\usetikzlibrary{decorations.markings}
\usepackage{adjustbox}
\usepackage{graphicx}
\usepackage{hyperref}

\DeclareFontFamily{U}{min}{}
\DeclareFontShape{U}{min}{m}{n}{<-> udmj30}{}
\newcommand\yo{\!\text{\usefont{U}{min}{m}{n}\symbol{'207}}\!}

\makeatletter
\def\blfootnote{\xdef\@thefnmark{}\@footnotetext}
\makeatother

\tikzset{
  blackdot/.pic={
	\fill (0,0) circle[radius=0.065];
  }
}

\tikzset{
  whitedot/.pic={
	\draw (0,0) circle[radius=0.065];
  }
}

\usepackage{amsthm}

\theoremstyle{plain}
\newtheorem{lemma}{Lemma}[subsection]
\newtheorem{prop}[lemma]{Proposition}
\newtheorem{theorem}[lemma]{Theorem}
\newtheorem{cor}[lemma]{Corollary}

\newtheorem{taller}[lemma]{$\!\!$}
\newenvironment{blanko}[1]%
{\begin{taller}{\normalfont\bfseries  #1}\normalfont}%
{\end{taller}}

\theoremstyle{definition}
\newtheorem{remark}[lemma]{Remark}

\newtheorem{example}[lemma]{Example}
\newtheorem{definition}[lemma]{Definition}

\newtheorem{observation}[lemma]{Observation}

\newtheorem{aplemma}{Fact}[section]
\newtheorem{apdefinition}[aplemma]{Definition}

\newenvironment{dashlist}%
{%
	\begin{list}%
	{---}%
	{%
	\setlength{\itemsep}{1pt}\setlength{\parsep}{0pt}\setlength{\topsep}{0pt}%
	\setlength{\itemindent}{16pt}\setlength{\labelwidth}{4pt}%
	\setlength{\labelsep}{4pt}\setlength{\leftmargin}{7pt}}%
}%
{\end{list}}

\DeclareRobustCommand\upperstar{%
  \mathchoice%
    {\kern0pt\raise0.55ex\hbox{$\displaystyle *$}\kern0.8pt}
    {\kern0pt\raise0.58ex\hbox{$\textstyle *$}\kern0.8pt}
    {\kern0pt\raise0.45ex\hbox{$\scriptstyle *$}\kern0.4pt}
    {\kern0pt\raise0.4ex\hbox{$\scriptscriptstyle *$}\kern0.2pt}
}%
\DeclareRobustCommand\lowerstar{%
  \mathchoice%
    {\kern0pt\raise-0.65ex\hbox{$\displaystyle *$}\kern0.8pt}
    {\kern0pt\raise-0.68ex\hbox{$\textstyle *$}\kern0.8pt}
    {\kern0pt\raise-0.54ex\hbox{$\scriptstyle *$}\kern0.4pt}
    {\kern0pt\raise-0.5ex\hbox{$\scriptscriptstyle *$}\kern0.2pt}
}%

\DeclareRobustCommand\upperfivestar{%
  \mathchoice%
    {\kern0pt\raise0.55ex\hbox{$\displaystyle \star$}\kern0.8pt}
    {\kern0pt\raise0.58ex\hbox{$\textstyle \star$}\kern0.8pt}
    {\kern0pt\raise0.45ex\hbox{$\scriptstyle \star$}\kern0.4pt}
    {\kern0pt\raise0.4ex\hbox{$\scriptscriptstyle \star$}\kern0.2pt}
}%

\newcommand{\isleftadjointto}{\dashv}
\newcommand{\op}{^{\text{{\rm{op}}}}}

\newcommand{\psimpcat}{\simplexcategory^{\operatorname{pt}}}

\newcommand{\augmap}{\zeta}

\newcommand{\uu}{\mathsf{u}}
\newcommand{\ii}{\mathsf{i}}
\newcommand{\kk}{\mathsf{k}}
\newcommand{\bb}{\mathsf{b}}
\newcommand{\qq}{\mathsf{q}}
\newcommand{\pp}{\mathsf{p}}
\newcommand{\rr}{\mathsf{r}}
\newcommand{\hh}{\mathsf{h}}
\newcommand{\w}{\mathsf{w}}
\newcommand{\jj}{\mathsf{j}}

\newcommand{\spaces}{\mathcal{S}}

\usepackage{eucal}

\providecommand{\simplexcategorytext}{%
\begin{tikzpicture}
	\begin{scope}[scale=0.82]
	  \draw[line width = 0.7pt] (0.0, 0.0) -- (0.142, 0.284) -- (0.284, 0.0) -- 
	  (0.0, 0.0) -- (0.142, 0.284);
	  \draw[line width = 0.7pt] (0.057, 0.0) -- +(0.113, 0.227);
	\end{scope}
\end{tikzpicture}%
}
\providecommand{\simplexcategoryscript}{%
\begin{tikzpicture}
	\begin{scope}[scale=0.57]
	  \draw[line width = 0.5pt] (0.0, 0.0) -- (0.142, 0.284) -- (0.284, 0.0) -- 
	  (0.0, 0.0) -- (0.142, 0.284);
	  \draw[line width = 0.5pt] (0.057, 0.0) -- +(0.113, 0.227);
	\end{scope}
\end{tikzpicture}%
}
\providecommand{\simplexcategoryscriptscript}{%
\begin{tikzpicture}
	\begin{scope}[scale=0.42]
	  \draw[line width = 0.7pt] (0.0, 0.0) -- (0.142, 0.284) -- (0.284, 0.0) -- 
	  (0.0, 0.0) -- (0.142, 0.284);
	  \draw[line width = 0.7pt] (0.057, 0.0) -- +(0.113, 0.227);
	\end{scope}
\end{tikzpicture}%
}

\DeclareRobustCommand\simplexcategory{%
  \mathchoice%
    {\simplexcategorytext}
    {\simplexcategorytext}
    {\simplexcategoryscript}
    {\simplexcategoryscriptscript}
}%

\def\rad{0.20}
\newsavebox{\subbotrecttext}
\sbox{\subbotrecttext}{\begin{tikzpicture}[line width = 0.36pt, line cap=round]
	\draw (0,0) --+ (\rad,0);
	\draw (0,-0.3*\rad) --+ (\rad,0);
	\draw (0.5*\rad,0) --+ (0,0.8*\rad);
  \end{tikzpicture}}
\newsavebox{\suptoprecttext}
\sbox{\suptoprecttext}{\begin{tikzpicture}[line width = 0.36pt, line cap=round]
	\draw (0,0) --+ (\rad,0);
	\draw (0,0.3*\rad) --+ (\rad,0);
	\draw (0.5*\rad,0) --+ (0,-0.8*\rad);
  \end{tikzpicture}}

\def\rad{0.15}
\newsavebox{\subbotrectscript}
\sbox{\subbotrectscript}{\begin{tikzpicture}[line width = 0.32pt, line cap=round]
	\draw (0,0) --+ (\rad,0);
	\draw (0,-0.3*\rad) --+ (\rad,0);
	\draw (0.5*\rad,0) --+ (0,0.8*\rad);
  \end{tikzpicture}}
\newsavebox{\suptoprectscript}
\sbox{\suptoprectscript}{\begin{tikzpicture}[line width = 0.32pt, line cap=round]
	\draw (0,0) --+ (\rad,0);
	\draw (0,0.3*\rad) --+ (\rad,0);
	\draw (0.5*\rad,0) --+ (0,-0.8*\rad);
  \end{tikzpicture}}

\def\rad{0.12}
\newsavebox{\subbotrectscriptscript}
\sbox{\subbotrectscriptscript}{\begin{tikzpicture}[line width = 0.26pt, line cap=round]
	\draw (0,0) --+ (\rad,0);
	\draw (0,-0.3*\rad) --+ (\rad,0);
	\draw (0.5*\rad,0) --+ (0,0.8*\rad);
  \end{tikzpicture}}
\newsavebox{\suptoprectscriptscript}
\sbox{\suptoprectscriptscript}{\begin{tikzpicture}[line width = 0.26pt, line cap=round]
	\draw (0,0) --+ (\rad,0);
	\draw (0,0.3*\rad) --+ (\rad,0);
	\draw (0.5*\rad,0) --+ (0,-0.8*\rad);
  \end{tikzpicture}}
  
\DeclareRobustCommand\subbot{%
  {\mathchoice%
    {\usebox{\subbotrecttext}}
    {\usebox{\subbotrecttext}}
    {\usebox{\subbotrectscript}}
    {\usebox{\subbotrectscriptscript}}
}}%
\DeclareRobustCommand\suptop{%
  {\mathchoice%
    {\usebox{\suptoprecttext}}
    {\usebox{\suptoprecttext}}
    {\usebox{\suptoprectscript}}
    {\usebox{\suptoprectscriptscript}}
}}%

\newcommand{\ssplit}{s_\subbot}
\newcommand{\scosplit}{s^\subbot}
\newcommand{\tsplit}{t_\suptop}

\newcommand{\Tot}{\operatorname{Tot}}
\newcommand{\decbot}{\operatorname{Dec}_\bot{}\kern-2pt}
\newcommand{\dectop}{\operatorname{Dec}_\top{}\kern-2pt}
\newcommand{\Decbot}[1]{\operatorname{Dec}_\bot{}\kern-2pt{#1}}
\newcommand{\Dectop}[1]{\operatorname{Dec}_\top{}\kern-2pt{#1}}
\newcommand{\decbotb}{{\operatorname{D}}\widetilde{\operatorname{ec}}_\bot{}\kern-1pt}

\newcommand{\Dec}{\operatorname{Dec}}

\newcommand{\Id}{\operatorname{id}}
\newcommand{\id}{\operatorname{id}}
\newcommand{\Map}{\operatorname{Map}}
\newcommand{\Fun}{\operatorname{Fun}}
\newcommand{\name}[1]{\ulcorner #1\urcorner}

\providecommand{\kat}[1]{\text{\textbf{\textsl{#1\/}}}}
\newcommand{\PrSh}{\kat{Pr}}
\newcommand{\sS}{s\mathcal{S}}
\renewcommand{\Pr}{\kat{Pr}} 
\newcommand{\twoSeg}{\PrSh^{\operatorname{2-Seg}}(\simplexcategory)}

\newcommand{\Abcs}{\kat{ABC}}
\newcommand{\BOORS}{\Pr^{\operatorname{BOORS}}(\Sigma)}
\newcommand{\DD}{\mathcal{D}}
\newcommand{\sd}{\operatorname{sd}}
\newcommand{\isopil}{\stackrel{\raisebox{0.1ex}[0ex][0ex]{\(\sim\)}}%
			{\raisebox{-0.15ex}[0.28ex]{\(\rightarrow\)}}}

\newcommand{\drpullback}
  {\arrow[phantom]{dr}[very near start,description]{\lrcorner}}
\newcommand{\dlpullback}
  {\arrow[phantom]{dl}[very near start,description]{\llcorner}}

\newcommand{\ulpullback}
  {\arrow[phantom]{ul}[very near start,description]{\ulcorner}}
\newcommand{\urpullback}
  {\arrow[phantom]{ur}[very near start,description]{\urcorner}}

\newcommand{\comma}{\raisebox{1pt}{$\downarrow$}}

\begin{document}

\author{Joachim Kock and Thomas Jan Mikhail}

\title{

Abacus bicomodule configurations and the 
Bergner--Osorno--Ozornova--Rovelli--Scheimbauer equivalence
}

\date{}

\maketitle

\begin{abstract}
  A theorem of Bergner, Osorno, Ozornova, Rovelli, and Scheimbauer states an
  equivalence between \mbox{$2$-Segal} spaces and certain augmented stable double
  Segal spaces. In this paper we establish more general equivalences, involving simplicial maps of $2$-Segal spaces and abacus
  bicomodule configurations, extending results of Carlier. The BOORS equivalence is recovered from the special case of the
  identity map. One main ingredient is an analysis of the relationship between the
  BOORS and Carlier notions of augmentation, hitherto considered unrelated.\blfootnote{This work was funded by grant FI-DGR 2020 of
  AGAUR (Catalonia) and by grant 10.46540/3103-00099B from the Independent Research
  Fund Denmark. It was also supported by research grant PID2020-116481GB-I00
  (AEI/FEDER, UE) of Spain and grant 2021-SGR-1015 of Catalonia, and through the
  Severo Ochoa and Mar\'ia de Maeztu Program for Centers and Units of Excellence in
  R\&D grant number CEX2020-001084-M as well as the Danish National Research
  Foundation through the Copenhagen Centre for Geometry and Topology (DNRF151).}
\end{abstract}

\tableofcontents

\newpage

\section{Introduction}

\subsection{Background}

\begin{blanko}{$2$-Segal spaces and $S_\bullet$-constructions.}
  One of the main motivations for Dyckerhoff and Kapranov to introduce the notion
  of $2$-Segal space~\cite{Dyckerhoff-Kapranov:1212.3563} was Waldhausen's
  $S_\bullet$-construction~\cite{Waldhausen} and Hall algebras. They showed~\cite[Theorem~7.3.3]{Dyckerhoff-Kapranov:1212.3563} that for any proto-exact
  $\infty$-category $\mathcal{A}$ (such as for example a stable
  $\infty$-category), the Waldhausen $S_\bullet$-construction
  $S_\bullet(\mathcal{A})$ is a $2$-Segal space (see
  also~\cite{Galvez-Kock-Tonks:1512.07573}), and that the Hall-algebra
  construction on $\mathcal{A}$ factors through $S_\bullet(\mathcal{A})$.
  
  Bergner, Osorno, Ozornova, Rovelli, and
  Scheimbauer~\cite{Bergner-Osorno-Ozornova-Rovelli-Scheimbauer:1609.02853},~\cite{Bergner-Osorno-Ozornova-Rovelli-Scheimbauer:1809.10924} (BOORS) made the
  surprising discovery that {\em every} $2$-Segal space arises as an
  $S_\bullet$-construction, if just the $S_\bullet$-construction is extended to
  more general inputs, certain bisimplicial spaces. They identified 
  certain augmented stable double Segal
  spaces as the appropriate input to produce (all) $2$-Segal spaces, and showed
  that this generalized $S_\bullet$-construction is part of an equivalence. The
  BOORS theorem is quite remarkable, as it gives a completely new perspective on
  $2$-Segal spaces, and at the same time gives new insight on the
  $S_\bullet$-construction (see~\cite{Bergner-Osorno-Ozornova-Rovelli-Scheimbauer:1901.03606}), by staging it in
  a setting where it has an inverse. It is striking that this inverse is another
  well-appreciated construction: the inverse takes a $2$-Segal space to its total
  decalage, suitably augmented. Waldhausen's original $S_\bullet$-construction
  dates back to 1983, while the total decalage construction is credited to
  Illusie~\cite{Illusie2} (1972). For an overview of how various Waldhausen
  constructions relate in this perspective, see~\cite{Bergner-Osorno-Ozornova-Rovelli-Scheimbauer:1901.03606}; for an
  introduction to the BOORS equivalence, see~\cite{Rovelli:2412.17400}.
\end{blanko}

\begin{blanko}{Decomposition spaces and bicomodules.}
  Independently of the work of Dyckerhoff and Kapranov, the notion of
  decomposition space was introduced by G\'alvez,
  Kock, and Tonks~\cite{Galvez-Kock-Tonks:1512.07573} with the purpose
  of providing the most general setting for the incidence coalgebra
  construction, originally introduced by Rota~\cite{Rota:Moebius} for posets. An
  important motivation for this generalization was M\"obius
  inversion~\cite{Galvez-Kock-Tonks:1512.07577}. It was quickly realized (first
  by Anel) that decomposition spaces are essentially the same thing as $2$-Segal
  spaces (although it took some years before the last bit of the equivalence
  fell into place, namely the statement that unitality of a $2$-Segal space is
  automatic~\cite{Feller-Garner-Kock-Proulx-Weber:1905.09580}).
  
  A first goal of the theory of decomposition spaces was to upgrade the
  classical theory of incidence algebras and M\"obius inversion from posets to
  decomposition spaces
  \cite{Galvez-Kock-Tonks:1612.09225}. Carlier~\cite{Carlier:1801.07504} took an
  important step in this direction with a far-reaching generalization of Rota's
  formula for the M\"obius functions of two posets related by a Galois
  connection~\cite{Rota:Moebius}. He generalized Rota's formula to adjunctions
  of $\infty$-categories, and went further to the situation of certain
  correspondences between decomposition spaces. To this end he developed a
  theory of {\em bicomodule configurations}, namely suitably augmented stable
  double Segal spaces designed to have incidence bicomodules, just like
  decomposition spaces have incidence coalgebras. 
  In this setting, Carlier established a M\"obius
  inversion principle for certain {\em M\"obius} bicomodule configurations,
  which feature vertical top splittings and horizontal bottom splittings. Having
  horizontal splittings is equivalent to having so-called {\em abacus} maps,
  which are additional operators $B_{i+1,j}\to B_{i,j+1}$ on a bisimplicial
  space $B$.
  
  A key ingredient in the theory of Carlier is a construction that to any
  correspondence of decomposition spaces defines a bicomodule configuration, and
  to any functor between $\infty$-categories defines a bicomodule configuration
  with abacus maps. It is via this construction that the classical set-up of
  Galois connections and $\infty$-adjunctions is brought into the abstract
  framework of stable double Segal spaces and bicomodule configurations,
  so that Rota's original formula can be seen as a special case of Carlier's general
  M\"obius inversion principle for bicomodules.
\end{blanko}

\begin{blanko}{Stability and augmentations.} \label{bla:Stb+Aug}
  The notion of stable double Segal space is central to both the BOORS equivalence
  and Carlier's theory of bicomodules, but with very different motivations. A
  double Segal space is called stable if certain vertical-horizontal bisimplicial
  identities are pullback squares. For BOORS, the purpose was to capture certain
  bipullback features of stable or proto-exact $\infty$-categories as used in the
  classical $S_\bullet$-construction; for Carlier the purpose was to encode the
  bicomodule condition of a left and a right coaction.
  (Carlier~\cite{Carlier:1801.07504} learned about the stability condition from
  BOORS~\cite{Bergner-Osorno-Ozornova-Rovelli-Scheimbauer:1609.02853}, but
  reformulated it in a way suitable for $\infty$-categories.)
  
  However, the notions of augmentation used by BOORS and Carlier are 
  very different. For Carlier, an augmentation of a bisimplicial space $B$
  consists of an augmentation (in the usual sense of simplicial spaces) of each
  row and each column, assembling into two extra new simplicial spaces: one
  forming an augmentation column $B_{\bullet,-1}$ and the other forming an
  augmentation row $B_{-1,\bullet}$; these play the role of the coalgebras that
  the bicomodule is over. A bicomodule configuration in the sense of Carlier thus
  has the shape
  \begin{equation}\label{bla:Delta/[1]}
    \begin{tikzcd} 
         & B_{-1,0} \ar[r, shorten <=5pt, 
	shorten >=4pt] 
         & B_{-1,1} \ar[l, shift left=1.5]
                    \ar[l, shift right=1.5]
                    \ar[r, shorten <=5pt, 
	shorten >=4pt, shift left=1.5] 
                    \ar[r, shorten <=5pt, 
	shorten >=4pt, shift right=1.5]
         & B_{-1,2} \ar[l]
                    \ar[l, shift left=3] 
                    \ar[l, shift right=3] 
                    \ar[r, shorten <=5pt, 
	shorten >=4pt, phantom, "\dots"]
                    & \phantom{} \\
       B_{0,-1} \ar[d, shorten <=5pt, shorten >=2pt]
         & B_{0,0} \ar[d, shorten <=5pt, shorten >=2pt]
                   \ar[u]
                   \ar[l] 
                   \ar[r, shorten <=5pt, 
	shorten >=4pt] 
         & B_{0,1} \ar[d, shorten <=5pt, shorten >=2pt]
                   \ar[u]
                   \ar[l, shift left=1.5]
                   \ar[l, shift right=1.5]
                   \ar[r, shorten <=5pt, 
	shorten >=4pt, shift left=1.5] 
                   \ar[r, shorten <=5pt, 
	shorten >=4pt, shift right=1.5]
         & B_{0,2} \ar[u]
                   \ar[d, shorten <=5pt, shorten >=2pt]
                   \ar[l]
                   \ar[l, shift left=3] 
                   \ar[l, shift right=3]
                   \ar[r, shorten <=5pt, 
	shorten >=4pt, phantom, "\dots"]
                    & \phantom{}\\
       B_{1,-1} \ar[u, shift left=1.5]
                \ar[u, shift right=1.5]
                   \ar[d, shorten <=5pt, shorten >=2pt, phantom, "\vdots"]
         & B_{1,0} \ar[u, shift left=1.5]
                   \ar[u, shift right=1.5]
                   \ar[l] 
                   \ar[r, shorten <=5pt, 
	shorten >=4pt] 
                   \ar[d, shorten <=5pt, shorten >=2pt, phantom, "\vdots"]
         & B_{1,1} \ar[u, shift left=1.5]
                   \ar[u, shift right=1.5]
                   \ar[l, shift left=1.5]
                   \ar[l, shift right=1.5]
                   \ar[r, shorten <=5pt, 
	shorten >=4pt, shift left=1.5] 
                   \ar[r, shorten <=5pt, 
	shorten >=4pt, shift right=1.5] 
                   \ar[d, shorten <=5pt, shorten >=2pt, phantom, "\vdots"]
         & B_{1,2}   
                   \ar[u, shift left=1.5]
                   \ar[u, shift right=1.5]
                   \ar[l]
                   \ar[l, shift left=3] 
                   \ar[l, shift right=3]
                   \ar[r, shorten <=5pt, 
	shorten >=4pt, phantom, "\dots"]
                   \ar[d, shorten <=5pt, shorten >=2pt, phantom, "\vdots"]
                    & \phantom{}\\
        \phantom{}
          & \phantom{}\
          &\phantom{}
          & \phantom{} &
    \end{tikzcd}
  \end{equation}
  Technically, this is the shape of a presheaf on the category
  $\simplexcategory_{/[1]}$.

  In contrast, what in the BOORS papers~\cite{Bergner-Osorno-Ozornova-Rovelli-Scheimbauer:1609.02853,Bergner-Osorno-Ozornova-Rovelli-Scheimbauer:1809.10924}
  is called an augmentation of a bisimplicial space
  $B$ 
  is the addition
  of a single extra space $B_{-1}$ with a morphism
  {\em to} $B_{0,0}$, subject to some pullback conditions. This space parametrizes the
  objects that play the role of the zero object in a proto-exact
  $\infty$-category.
  A BOORS-augmented bisimplicial space thus has the shape
  \begin{center}
    \begin{tikzcd} 
      B_{-1} \ar[rd]   &  
         & 
         &  \\
         & B_{0,0}
		 \ar[d, shorten <=5pt, 
     shorten >=3pt]
                   \ar[r, shorten <=5pt, 
                   shorten >=4pt] 
         & B_{0,1}
		 \ar[d, shorten <=5pt, 
     shorten >=3pt]
                   \ar[l, shift left=1.5]
                   \ar[l, shift right=1.5]
                   \ar[r, shorten <=5pt, 
                   shorten >=4pt, shift left=1.5] 
                   \ar[r, shorten <=5pt, 
                   shorten >=4pt, shift right=1.5]
         & B_{0,2}
                   \ar[d, shorten <=5pt, 
                   shorten >=3pt]
                   \ar[l]
                   \ar[l, shift left=3] 
                   \ar[l, shift right=3]
                   \ar[r, phantom, "\dots"]
                    & \phantom{}\\
         & B_{1,0}
         \ar[u, shift left=1.5]
                   \ar[u, shift right=1.5]
                   \ar[r, shorten <=5pt, 
                   shorten >=4pt] 
                   \ar[d, phantom, "\vdots"]
         & B_{1,1} 
         \ar[u, shift left=1.5]
                   \ar[u, shift right=1.5]
                   \ar[l, shift left=1.5]
                   \ar[l, shift right=1.5]
                   \ar[r, shorten <=5pt, 
                   shorten >=4pt, shift left=1.5] 
                   \ar[r, shorten <=5pt, 
                   shorten >=4pt, shift right=1.5] 
                   \ar[d, phantom, "\vdots"]
         & B_{1,2}   
                   \ar[u, shift left=1.5]
                   \ar[u, shift right=1.5]
                   \ar[l]
                   \ar[l, shift left=3] 
                   \ar[l, shift right=3]
                   \ar[r, phantom, "\dots"]
                   \ar[d, phantom, "\vdots"]
                    & \phantom{}\\
        \phantom{}
          & \phantom{}\
          &\phantom{}
          & \phantom{} &
    \end{tikzcd}
  \end{center}
  --- it is a presheaf on a certain category $\Sigma$.
  
  In the present work (outside of this introduction), in order to avoid confusion regarding the word ``augmentation'', we will instead refer to the map
  $B_{-1} \to B_{0,0}$ as a {\em pointing}, and call the pullback conditions the
  {\em pointing axioms}. For all such larger diagram shapes extending 
  bisimplicial spaces, we call the 
  nonnegatively-indexed part the {\em bulk}.
\end{blanko}

The fact that stable double Segal spaces arise in two different settings with
disparate motivations could be ascribed to a coincidence --- after all it
is a very natural notion.
The discrepancy between the two notions of
  augmentation abets the supposition that
  these settings might be unrelated.

\subsection{Contributions of this paper}

  The achievement of the present paper is to connect the works of 
  BOORS and Carlier.
  We have two main contributions (as well as the theory building up to these
  results):
  \begin{itemize}
  \item We upgrade Carlier's key construction to an equivalence, by
  identifying the conditions allowing one to ``go back''. The
  equivalence we establish goes between certain simplicial maps and certain
  abacus bicomodule configurations.
  \item We exploit Carlier's abacus maps to explain the relationship between the two notions
  of augmentation. This allows the BOORS equivalence to be derived from more
  general equivalences involving simplicial maps and bicomodule configurations.
  \end{itemize}
We proceed to a more detailed and technical overview of the main results.

\begin{blanko}{Simplicial infrastructure (\S\ref{sec:from -1 to splittings}).}
  The paper is primarily about bisimplicial spaces. However, we do need some
  groundwork which is simplicial rather than bisimplicial, culminating with a
  proposition (\ref{prop:RigCoalgIsRigPoint}) comparing the BOORS augmentation axioms
  with certain coalgebras for the lower-decalage comonad. This is a
  generalization of the observation of
  Garner--Kock--Weber~\cite{Garner-Kock-Weber:1812.01750} that Dec-coalgebra
  structure on a $1$-category expresses a local-initial-objects structure (or
  local-terminal-objects structure, 
  which is important
  in the theory of
  operadic categories of Batanin and Markl~\cite{Batanin-Markl:1404.3886}.)
\end{blanko}

\begin{blanko}{Abacus maps (\S\ref{sec:bisimplicial}).} \label{bla:Abacus}
  A prominent role is played throughout by the so-called {\em abacus maps}: 
  these are a
  family of diagonal maps $B_{i+1,j} \to B_{i,j+1}$ in a row-and-column-augmented
  bisimplicial space $B$ satisfying a number of relations. Such an
  object thus looks like this:
  \begin{center}
    \begin{tikzcd} 
         & B_{-1,0} \ar[r, shorten <=5pt, shorten >=4pt] 
         & B_{-1,1} \ar[l, shift left=1.5]
                    \ar[l, shift right=1.5]
                    \ar[r, shorten <=5pt, shorten >=4pt, shift left=1.5] 
                    \ar[r, shorten <=5pt, shorten >=4pt, shift right=1.5]
         & B_{-1,2} \ar[l]
                    \ar[l, shift left=3] 
                    \ar[l, shift right=3] 
                    \ar[r, shorten <=5pt, shorten >=4pt, phantom, "\dots"]
                    & \phantom{} \\
       B_{0,-1} \ar[d, shorten <=5pt, shorten >=3pt]
                \ar[ur]
         & B_{0,0} \ar[d, shorten <=5pt, shorten >=3pt]
                   \ar[u]
                   \ar[l] 
                   \ar[r, shorten <=5pt, shorten >=4pt] 
                   \ar[ur]
         & B_{0,1} \ar[d, shorten <=5pt, shorten >=3pt]
                   \ar[u]
                   \ar[l, shift left=1.5]
                   \ar[l, shift right=1.5]
                   \ar[r, shorten <=5pt, shorten >=4pt, shift left=1.5] 
                   \ar[r, shorten <=5pt, shorten >=4pt, shift right=1.5]
                   \ar[ur, shorten <=6pt, shorten >=3pt] 
         & B_{0,2} \ar[u]
                   \ar[d, shorten <=5pt, shorten >=3pt]
                   \ar[l]
                   \ar[l, shift left=3] 
                   \ar[l, shift right=3]
                   \ar[r, shorten <=5pt, shorten >=4pt, phantom, "\dots"]
                    & \phantom{}\\
       B_{1,-1} \ar[u, shift left=1.5]
                \ar[u, shift right=1.5]
                   \ar[ur]
                   \ar[d, phantom, "\vdots"]
         & B_{1,0} \ar[u, shift left=1.5]
                   \ar[u, shift right=1.5]
                   \ar[l] 
                   \ar[r, shorten <=5pt, shorten >=4pt] 
                   \ar[ur]
                   \ar[d, phantom, "\vdots"]
         & B_{1,1} \ar[u, shift left=1.5]
                   \ar[u, shift right=1.5]
                   \ar[l, shift left=1.5]
                   \ar[l, shift right=1.5]
                   \ar[r, shorten <=5pt, shorten >=4pt, shift left=1.5] 
                   \ar[r, shorten <=5pt, shorten >=4pt, shift right=1.5] 
                   \ar[ur, shorten <=6pt, shorten >=3pt]
                   \ar[d, phantom, "\vdots"]
         & B_{1,2}   
                   \ar[u, shift left=1.5]
                   \ar[u, shift right=1.5]
                   \ar[l]
                   \ar[l, shift left=3] 
                   \ar[l, shift right=3]
                   \ar[r, shorten <=5pt, shorten >=4pt, phantom, "\dots"]
                   \ar[d, phantom, "\vdots"]
                    & \phantom{}\\
        \phantom{}
          & \phantom{}\
          &\phantom{}
          & \phantom{}&
    \end{tikzcd}
\end{center}
  It is a presheaf on a suitable index category $\mathcal{D}$, which we study 
  in some detail. 
  As part of the description of $\DD$, the data of abacus maps is shown to
  be equivalent to having horizontal splittings (that is, $\decbot$-coalgebra
  structure) on all 
  bulk rows. Abacus maps appear already in Carlier's
  first paper~\cite{Carlier:1801.07504} and were named and axiomatized in his
  second paper~\cite{Carlier:1812.09915}.
  An {\em abacus bicomodule configuration} is a presheaf on $\DD$ that is stable double Segal and satisfies further conditions on the augmentations (cf.~\ref{bimod-conf}).
  
\end{blanko}

\begin{blanko}{Simplicial maps vs.~abacus bicomodule configurations (\S\ref{sec:carlier-stuff}).}
  \label{bla:RelSetting}
  Presheaves on $\DD$ arise naturally from general simplicial maps (that is,
  presheaves on $\simplexcategory\times[1]$), by way of right Kan extension
  $$
  \begin{tikzcd}
  \PrSh(\simplexcategory\times[1])  \ar[r, shift left, "\qq\lowerstar"]
  &
  \PrSh(\mathcal{D})
  \ar[l, shift left, "\qq\upperstar"]
  \end{tikzcd}
  $$
  along the functor $\qq :\simplexcategory\times[1] \to \DD$ that includes the two copies of $\simplexcategory$ into the two augmentations of $\DD$.
  In Theorem~\ref{thm:star-equiv} we identify
  a certain condition ($\star$) on $\DD$-presheaves, 
  required for the adjunction to restrict to an equivalence
  $$
  \PrSh(\simplexcategory\times[1]) \simeq \PrSh\upperfivestar(\mathcal{D})  .
  $$
  This has the flavor of the BOORS equivalence, but in a relative setting, and
  with more elaborate categories in place of $\simplexcategory$ and $\Sigma$.
  From here we gradually impose more conditions on the two sides of the
  equivalence to arrive at objects of special interest. We introduce the notion
  of relatively upper $2$-Segal simplicial maps between $2$-Segal spaces,
  identified as precisely those simplicial maps that correspond to Carlier's
  bicomodule configurations (with abacus maps satisfying ($\star$)). In
  particular, the adjunction $ \qq\upperstar \isleftadjointto \qq\lowerstar$
  restricts to an equivalence (Theorem~\ref{thm:ABC*=Up2Seg})
  $$
    \PrSh^{\operatorname{up\, 2-Seg}} (\simplexcategory\times [1])
	\simeq \Abcs\upperfivestar ,
  $$
  where $ \Abcs\upperfivestar $ stands for the full subcategory of $
  \PrSh\upperfivestar(\DD) $ spanned by abacus bicomodule configurations. By
  establishing an equivalence, this improves upon a result of Carlier~\cite{Carlier:1812.09915}, who showed how to construct an abacus bicomodule
  configuration from any functor of $\infty$-categories. (Note that between
  $\infty$-categories, any functor is relatively upper $2$-Segal.) Finally,
  restricting further, to the full subcategory of $\PrSh(\simplexcategory)$
  spanned by the $2$-Segal spaces (interpreted as equivalences of $2$-Segal
  spaces) and to the full subcategory of bicomodule configurations
  with invertible abacus maps, we arrive at an equivalence (Theorem~\ref{thm:ABCInv=2Seg})
  $$
    \twoSeg \simeq \Abcs^\simeq ,
  $$
  which we interpret as a BOORS-type equivalence 
  but with 
  $\DD$-presheaves instead of $\Sigma$-presheaves.
\end{blanko}

\begin{blanko}{Relating the two notions of augmentation (\S\ref{sec:Boors-comp}).}
  \label{bla:Unific}
  The results above prompt a closer analysis of the relationship between the two
  notions of augmentation, leading finally to a derivation of the BOORS
  equivalence from equivalences involving bicomodules.
  The comparison of the two notions of augmentation is mediated by a functor $\jj
  : \Sigma \to \DD$ sending $[-1]$ to $[0,-1]$. 
  We show that the pullback functor 
  $\jj\upperstar$
  restricts 
  to an equivalence (Theorem~\ref{thm:j-starEquivalence}) 
  $$
   {\Abcs^\simeq} \isopil \BOORS
   $$ 
  between bicomodule configurations
  with invertible abacus maps and $\Sigma$-presheaves satisfying the BOORS axioms. 
  
  The equivalences established can be organized into the diagram
  \begin{equation}\label{overview-diagram}
    \begin{tikzcd}
      {\PrSh(\simplexcategory\times [1])} & {\PrSh(\DD)} & {\PrSh(\Sigma)} \\
      {\PrSh(\simplexcategory\times[1])} & {\PrSh\upperfivestar(\DD)} \\
      {\PrSh^{\operatorname{up\, 2-Seg}}(\simplexcategory\times [1])} & {\Abcs\upperfivestar} \\
      \twoSeg & {\Abcs^\simeq} & \BOORS
      \arrow["{\qq\lowerstar}", shift left, from=1-1, to=1-2]
      \arrow["{\qq\upperstar}", shift left, from=1-2, to=1-1]
      \arrow["{\jj\upperstar}", from=1-2, to=1-3]
      \arrow[equals, from=2-1, to=1-1]
      \arrow["\simeq", "\ref{thm:star-equiv}"', from=2-1, to=2-2]
      \arrow[hook, from=2-2, to=1-2]
      \arrow[hook, from=3-1, to=2-1]
      \arrow["{{\simeq}}", "{\ref{thm:ABC*=Up2Seg}}"', from=3-1, to=3-2]
      \arrow[hook, from=3-2, to=2-2]
      \arrow[hook, from=4-1, to=3-1]
      \arrow["{{\simeq}}", "{\ref{thm:ABCInv=2Seg}}"', from=4-1, to=4-2]
      \arrow[hook, from=4-2, to=3-2]
      \arrow["{{\simeq}}", "{\ref{thm:j-starEquivalence}}"', from=4-2, to=4-3]
      \arrow[hook, from=4-3, to=1-3]
    \end{tikzcd}
  \end{equation}
  where all the vertical inclusions are full. The composite of the bottom two
  horizontal functors turns out to be precisely the original BOORS equivalence
  (Corollary~\ref{cor:BOORS}).
  
\end{blanko}

  After the main body of the paper, we include a short little cheat sheet with
  standard facts about $2$-Segal spaces and various classes of simplicial maps.

\bigskip

\noindent {\bf Acknowledgments.} We thank Steve Lack and Julie Bergner for fruitful discussions and helpful feedback.
Subsections~\ref{subsec:split} and~\ref{subsec:dec} owe a lot to forthcoming work
of Batanin--Kock--Weber~\cite{Batanin-Kock-Weber:operadic}.

\section{Splittings, decalage, and rigid \texorpdfstring{$\decbot$}{decbot}-coalgebras}
\label{sec:from -1 to splittings}

In this section we set up some simplicial machinery concerning decalage
and subtle variations on the simplex category involving extra bottom
sections. The full scope of these results will be used in Section~5, but
the set-up, notation, and some of the results are also used in earlier
sections. The main goal is to establish an equivalence
(Proposition~\ref{prop:RigCoalgIsRigPoint}) between the notion of
local-initial-object structures expressed in terms of pointings (a
condition imposed on the zeroth row in a BOORS-augmented stable double
Segal space) and certain bottom-split simplicial spaces (as occur in the
rows of a Carlier-augmented stable double Segal space with abacus maps).
This builds on insight from
Garner–Kock–Weber~\cite{Garner-Kock-Weber:1812.01750} in a different
context.

\subsection{Bottom-split simplicial spaces}
\label{subsec:split}

As always, let $\simplexcategory$ be the category of non-empty finite standard
ordinals and monotone maps. We work with simplicial spaces, meaning 
$\simplexcategory$-presheaves 
with values in the $\infty$-category of spaces $\spaces$. They form the 
$\infty$-category $\sS := \Pr(\simplexcategory) := 
\Fun(\simplexcategory\op,\spaces)$. For a space $C$ we denote by $\bar C$ the 
simplicial space with constant value $C$.

We shall need also several other base 
categories closely related to $\simplexcategory$.
Let $\simplexcategory^\bb$ be the category of
non-empty finite standard ordinals with a bottom element and bottom-preserving
monotone maps. Let $\simplexcategory_\bb$ be the full subcategory of 
$\simplexcategory^\bb$ consisting of those objects where the bottom element is not alone.
The forgetful functor $\uu :\simplexcategory^\bb \to \simplexcategory$ has a left 
adjoint which freely adds a bottom element. The resulting monad on 
$\simplexcategory$ is denoted $\bb$. The forgetful functor $\uu$ is in fact monadic,
so that $\simplexcategory^\bb$ becomes the category of Eilenberg--Moore algebras 
for $\bb$, and $\simplexcategory_\bb$ the Kleisli category (the full subcategory 
of free algebras), hence the notation for these two categories.
The left adjoint free functor thus factors as
$$
\simplexcategory \stackrel{\kk}{\longrightarrow} 
\simplexcategory_\bb \stackrel{\ii }{\longrightarrow}
\simplexcategory^\bb  ,
$$
where $\kk$ is identity-on-objects and $\ii $ is full.
We let these functors dictate the naming conventions for the objects in the 
three categories:
$$
[n] \mapsto [n] \mapsto [n]  .
$$
The category $\simplexcategory^\bb$ thus has an extra object denoted $[-1]$
corresponding to the linear-order-with-bottom-element consisting of only the 
bottom 
element.

We now take presheaves (of spaces), to get functors
$$
\PrSh(\simplexcategory) \stackrel{\kk\upperstar}{\longleftarrow} 
\PrSh(\simplexcategory_\bb) \stackrel{\ii \upperstar}{\longleftarrow}
\PrSh(\simplexcategory^\bb) .
$$

Objects in $\PrSh(\simplexcategory)$ are of course simplicial spaces;
we picture them by drawing the first few face and degeneracy maps
\begin{center}
  \begin{tikzcd}[column sep={7em,between origins}]
    \phantom{X_{-1}} & {X_0} & {X_1} & {X_2} & {}
    \arrow["{s_0}"{description, inner sep = .5pt}, shorten <=5pt, shorten >=5pt, from=1-2, to=1-3]
    \arrow["{d_1}"', shift right=2.5, from=1-3, to=1-2]
    \arrow["{d_0}", shift left=2.5, from=1-3, to=1-2]
    \arrow["{s_1}"{description, inner sep = .5pt, pos=0.4}, shift left=2.5, shorten <=5pt, 
	shorten >=5pt, from=1-3, to=1-4]
    \arrow["{s_0}"{description, inner sep = .5pt, pos=0.4}, shift right=2.5, shorten <=5pt, 
	shorten >=5pt, from=1-3, to=1-4]
    \arrow["{d_1}"{description, inner sep = .5pt, pos=0.4}, from=1-4, to=1-3]
    \arrow["{d_2}"', shift right=5, from=1-4, to=1-3]
    \arrow["{d_0}", shift left=5, from=1-4, to=1-3]
    \arrow["\cdots"{description}, phantom, from=1-5, to=1-4]
  \end{tikzcd}
\end{center}

Objects in $\PrSh(\simplexcategory_\bb)$ are called {\em bottom-split simplicial
spaces}; the diagram of their generating maps starts like this:
\begin{center}
  \begin{tikzcd}[column sep={7em,between origins}]
    \phantom{X_{-1}} & {X_0} & {X_1} & {X_2} & {}
    \arrow["{s_0}"{description, inner sep = .5pt}, shorten <=5pt, shorten >=5pt, from=1-2, to=1-3]
    \arrow["{d_1}"', shift right=2.5, from=1-3, to=1-2]
    \arrow["{d_0}", shift left=2.5, from=1-3, to=1-2]
    \arrow["{\ssplit}"', bend right=25, shift right=2, from=1-2, to=1-3, shorten <=5pt, shorten >=5pt]
    \arrow["{s_1}"{description, inner sep = .5pt, pos=0.4}, shift left=2.5, shorten <=5pt, 
	shorten >=5pt, from=1-3, to=1-4]
    \arrow["{s_0}"{description, inner sep = .5pt, pos=0.4}, shift right=2.5, shorten <=5pt, 
	shorten >=5pt, from=1-3, to=1-4]
    \arrow["{d_1}"{description, inner sep = .5pt, pos=0.4}, from=1-4, to=1-3]
    \arrow["{d_2}"', shift right=5, from=1-4, to=1-3]
    \arrow["{d_0}", shift left=5, from=1-4, to=1-3]
    \arrow["\cdots"{description}, phantom, from=1-5, to=1-4]
    \arrow["{\ssplit}"', bend right=25, shift right=5, from=1-3, to=1-4, shorten <=5pt, shorten >=5pt]
  \end{tikzcd}
\end{center}
The extra bottom sections satisfy the simplicial identities corresponding to a 
section to $d_\bot$ ``below'', and we will always use the symbol $\ssplit$ for 
these ``sub-bottom'' degeneracy maps. 

Objects in $\PrSh(\simplexcategory^\bb)$ are called {\em bottom-split augmented
simplicial spaces}; they are drawn as
\begin{center}
  \begin{tikzcd}[column sep={7em,between origins}]
    {X_{-1}} & {X_0} & {X_1} & {X_2} & {}
    \arrow["{d_0}"', from=1-2, to=1-1]
    \arrow["{\ssplit}"', bend right=25, shift left=1, shorten <=5pt, shorten >=5pt, from=1-1, to=1-2]
    \arrow["{s_0}"{description, inner sep = .5pt}, shorten <=5pt, shorten >=5pt, from=1-2, to=1-3]
    \arrow["{d_1}"', shift right=2.5, from=1-3, to=1-2]
    \arrow["{d_0}", shift left=2.5, from=1-3, to=1-2]
    \arrow["{\ssplit}"', bend right=25, shift right=2, from=1-2, to=1-3, shorten <=5pt, shorten >=5pt]
    \arrow["{s_1}"{description, inner sep = .5pt, pos=0.4}, shift left=2.5, shorten <=5pt, shorten >=5pt, from=1-3, to=1-4]
    \arrow["{s_0}"{description, inner sep = .5pt, pos=0.4}, shift right=2.5, shorten <=5pt, shorten >=5pt, from=1-3, to=1-4]
    \arrow["{d_1}"{description, inner sep = .5pt, pos=0.4}, from=1-4, to=1-3]
    \arrow["{d_2}"', shift right=5, from=1-4, to=1-3]
    \arrow["{d_0}", shift left=5, from=1-4, to=1-3]
    \arrow["\cdots"{description}, phantom, from=1-4, to=1-5]
    \arrow["{\ssplit}"', bend right=25, shift right=5, from=1-3, to=1-4, shorten <=5pt, shorten >=5pt]
  \end{tikzcd}
\end{center}

Let $\simplexcategory_+$ denote the simplex category  
augmented with the empty ordinal, denoted $[-1]$. 
An {\em augmented simplicial space} is a presheaf on 
$\simplexcategory_+$, where the index convention is this:
\begin{center}
  \begin{tikzcd}[column sep={7em,between origins}]
    {X_{-1}} & {X_0} & {X_1} & {X_2} & {}
    \arrow["{d_0}"', from=1-2, to=1-1]
    \arrow["{s_0}"{description, inner sep = .5pt}, shorten <=5pt, shorten >=5pt, from=1-2, to=1-3]
    \arrow["{d_1}"', shift right=2.5, from=1-3, to=1-2]
    \arrow["{d_0}", shift left=2.5, from=1-3, to=1-2]
    \arrow["{s_1}"{description, inner sep = .5pt, pos=0.4}, shift left=2.5, shorten <=5pt, shorten >=5pt, from=1-3, to=1-4]
    \arrow["{s_0}"{description, inner sep = .5pt, pos=0.4}, shift right=2.5, shorten <=5pt, shorten >=5pt, from=1-3, to=1-4]
    \arrow["{d_1}"{description, inner sep = .5pt, pos=0.4}, from=1-4, to=1-3]
    \arrow["{d_2}"', shift right=5, from=1-4, to=1-3]
    \arrow["{d_0}", shift left=5, from=1-4, to=1-3]
    \arrow["\cdots"{description}, phantom, from=1-4, to=1-5]
  \end{tikzcd}
\end{center}
It is often practical to interpret this data as a simplicial map
$$
\bar X_{-1} \leftarrow X
$$
from the underlying simplicial space to the constant simplicial space on $X_{-1}$.

For any simplicial space $X$, we can take the geometric realization (colimit)
  $X_{-1}$ forming altogether an augmented simplicial space with augmentation
  map $X_{-1} \leftarrow X_0$. 
In the situation where $X$ is split by extra bottom degeneracy maps (that is, $X$ 
is the underlying simplicial space of a $\simplexcategory_\bb$-presheaf), the
augmentation map $X_{-1} \stackrel{d_0}\leftarrow X_0 $ acquires a section $
\ssplit : X_{-1} \to X_0 $, so as to form altogether a 
$\simplexcategory^\bb$-presheaf. In fact, this shape of diagram is an absolute
colimit: every $\simplexcategory^\bb$-presheaf $A$ exhibits $A_{-1}$ as the
geometric realization (colimit) of $\ii\upperstar(A)$, or equivalently, as the
geometric realization (colimit) of the underlying simplicial space $\kk \upperstar
\ii \upperstar (A)$ (see~\cite[6.1.3.16]{Lurie:HTT}). This is one interpretation of the fact that $\ii \upperstar$ is an equivalence.

\subsection{Decalage}
\label{subsec:dec}

The adjunction 
\[
\begin{tikzcd}[column sep={3em,between origins}]
& \Pr(\simplexcategory^\bb) \ar[ld, "\ii \upperstar"'] \\
\Pr(\simplexcategory_\bb) \ar[rd, "\kk \upperstar"'] \ar[r, phantom, "\isleftadjointto" description] & {}\\
& \Pr(\simplexcategory) \ar[uu, "\uu \upperstar"']
\end{tikzcd}
\]
defines a comonad on $\PrSh(\simplexcategory)$ which is the lower decalage
comonad 
$$
\decbot := \bb\upperstar = (\ii \kk )\upperstar \uu \upperstar : \PrSh(\simplexcategory) \to \PrSh(\simplexcategory).
$$
We write 
$$
\varepsilon: \decbot(X) \to X
$$
for its counit; it is given in each
degree by bottom face maps. We write $\delta: \decbot(X) \to \decbot\decbot (X)$
for the comultiplication; it is given in each degree by bottom degeneracy
maps. We shall also use the canonical augmentation of $\decbot(X)$, namely the
simplicial map
$$
\bar X_0 \stackrel\augmap\longleftarrow \decbot(X)
$$
given in each degree by composites of top face maps; in particular, in degree 
zero it is the map $X_0 \stackrel{d_1}\leftarrow X_1$.

The
functor $\kk\upperstar\circ \ii \upperstar$ is comonadic, so as to identify
$\PrSh(\simplexcategory^\bb)$ with the category of $\decbot$-coalgebras. In
explicit terms, a $ \simplexcategory^\bb $-presheaf $A$ has an underlying simplicial space
$X:=\kk \upperstar \ii  \upperstar (A)$ possessing extra bottom degeneracies; these
assemble into a simplicial map $\gamma: X \to \decbot(X) $ which is the
structure map of a $\decbot$-coalgebra: the split-simplicial identities
satisfied by the extra bottom degeneracies correspond precisely to the coalgebra
axioms.

It is also the case that $\uu \upperstar$ is monadic, so that the induced monad
$\decbotb:= \uu \upperstar (\ii \kk )\upperstar $ on
$\PrSh(\simplexcategory^\bb)\simeq \decbot$-$\kat{Coalg}$ has the category
$\PrSh(\simplexcategory)$ as category of algebras
\cite{Garner-Kock-Weber:1812.01750}. We shall not need that fact, but we shall
reference the monad $\decbotb$ a few times. We write $\eta_A : A \to \decbotb(A)$
for its unit.

\begin{observation}
  \label{obs:observation}
  If $X\in
  \Pr(\simplexcategory)$ has a $\decbot$-coalgebra structure, so as to constitute
  a $\simplexcategory^\bb$-presheaf $A$ (that is, $X = (\ii \kk )\upperstar (A)$),
  then the structure map $\gamma: X \to \decbot(X)$ underlies a 
  $\decbot$-coalgebra map, namely $\gamma= (\ii \kk )\upperstar 
  (\eta_A)$. (This identification is a general fact about comonads.)
  We say 
  loosely that $\gamma$ is itself a $\decbot$-coalgebra map.
\end{observation}

For categories (and in fact for $\infty$-categories), $\decbot$-coalgebras 
can be interpreted as specifying local initial objects, meaning an initial object 
in each connected component (cf.~Garner--Kock--Weber~\cite{Garner-Kock-Weber:1812.01750}
for the dual case of local terminal objects). The degeneracy map
$\ssplit: A_{-1} \to A_0$ picks out an object in each connected component, and 
the coalgebra axioms imply that these objects are in fact initial in 
each component.

\subsection{Maps between bottom-split simplicial spaces}

Since our main interest is in simplicial objects, we use the underlying
simplicial objects of $\simplexcategory^\bb$-presheaves to dictate how we extend
the various notions of simplicial map to bottom-split simplicial objects: a
morphism of bottom-split (possibly augmented) simplicial spaces $F: A \to B$ is
called {\em left fibration}, {\em right fibration}, {\em or culf} if the map of
underlying simplicial spaces
$$
(\kk \upperstar\circ \ii \upperstar) (A) \to (\kk \upperstar\circ \ii \upperstar) (B)
$$
is a left fibration, a right fibration, or culf, respectively. For the property
``cartesian'' there is a potential conflict of meanings: it could mean cartesian
on the simplicial part in line with the convention, or it could mean truly
cartesian (on all arrows). The following lemma resolves this conflict, as it
implies that the two notions agree:

\def\q{q}
\begin{lemma}\label{lem:lfib=>cart}
  A left fibration $F: A \to B$ of bottom-split simplicial spaces is automatically
  cartesian. If the bottom-split simplicial spaces are augmented,
  then $F$ is cartesian also on the augmentation map, here temporarily denoted 
  $A_{-1} \stackrel q \leftarrow A_0$. In particular
  \begin{equation}\label{eq:aug-pbk}
  \begin{tikzcd}
  A_{-1} \ar[d] & \ar[l, "\q"'] A_0 \ar[d]  \\
  B_{-1} & \ar[l, "\q"] B_0
  \end{tikzcd}
  \end{equation}
  is a pullback.
\end{lemma}

\begin{proof}
  We can assume that $F:A\to B$ is augmented (since it can be augmented 
  uniquely by taking geometric realization). We first prove that 
  then the square \eqref{eq:aug-pbk} is a pullback.
  This is a retract argument: the diagram
  \[
  \begin{tikzcd}
      {A_0} \ar[rrr, "\ssplit"{pos=0.6}] \ar[dd] \ar[dr, "\q"] & & & 
      {A_1} \ar[rrr, "d_0"{pos=0.6}] \ar[dd] \ar[dr, "d_1"] & & & 
      {A_0} \ar[dd] \ar[dr, "\q"] &
	  \\ &
      {A_{-1}} \ar[rrr, "\ssplit"{pos=0.4}] \ar[dd]  & & & 
      {A_0} \ar[rrr, "\q"{pos=0.4}] \ar[dd]  & & & 
      {A_{-1}} \ar[dd] 
	  \\ 
      {B_0} \ar[rrr, "\ssplit"{pos=0.6}] \ar[dr, "\q"] & & & 
      {B_1} \ar[rrr, "d_0"{pos=0.6}] \ar[dr, "d_1"] & & & 
      {B_0} \ar[dr, "\q"] &
	  \\ &
      {B_{-1}} \ar[rrr, "\ssplit"{pos=0.4}] & & &  
      {B_0} \ar[rrr, "\q"{pos=0.4}] & & &  
      {B_{-1}}
  \end{tikzcd}
  \]
  is easily seen to commute, and it exhibits the square~\eqref{eq:aug-pbk} as a
  retract of the middle vertical square in the diagram. This middle square is
  a pullback, since the simplicial map is assumed to be a left fibration.
  Since limits are stable
  under retracts,  it follows that also the square~\eqref{eq:aug-pbk} is a
  pullback, as asserted. Now since $F$ is cartesian on both $\q$ and on $d_1$,
  it is cartesian on $q \circ d_1 = q \circ d_0$ and therefore it must
  also be cartesian on $d_0$; this argument can be repeated for all
  face maps. So now $F$ is cartesian on all face maps. But then it is also
  cartesian on all degeneracy maps, since these are sections.
\end{proof}

\begin{remark}
  The above lemma is an abstraction of the fact for $1$-categories that if a left 
  fibration preserves local initial objects, then it is also a right fibration.
\end{remark}

\def\q{d_0}

\begin{lemma}\label{lem:rfib-on-s-1}
  Let $F: A \to B$ be a right fibration of bottom-split simplicial spaces.
  Then $F$ is also cartesian on all splitting maps, including
  the augmentation splitting.
\end{lemma}

\begin{proof}
  (Although the proof is similar to the previous proof, there are subtle 
  differences, and in particular the 
  square~\eqref{eq:aug-pbk} is not always a pullback.)
  That $F$ is a right fibration means that it is cartesian on bottom face maps. Since the
  splittings 
  $$
  A_i \stackrel{\ssplit}\longrightarrow A_{i+1}
  $$
  are sections to bottom face maps for all $i\geq 0$, it follows that $F$ is also 
  cartesian on all these splittings. It remains to show that $F$ is cartesian on 
  the augmentation splitting, meaning that the square
    \begin{equation}\label{eq:augsplit-pbk}
  \begin{tikzcd}
  A_{-1} \ar[d] \ar[r, "\ssplit"] &  A_0 \ar[d]  \\
  B_{-1} \ar[r, "\ssplit"'] &  B_0
  \end{tikzcd}
  \end{equation}
  is a pullback.
  This is a retract argument: the diagram
  \[
  \begin{tikzcd}
      {A_{-1}} \ar[rrr, "\ssplit"] \ar[dd] \ar[dr, "\ssplit"'] & & & 
      {A_0} \ar[rrr, "\q"] \ar[dd] \ar[dr, "\ssplit"] & & & 
      {A_{-1}} \ar[dd] \ar[dr, "\ssplit"] &
	  \\ &
      {A_0} \ar[rrr, pos=0.33, "s_0"] \ar[dd]  & & & 
      {A_1} \ar[rrr, pos=0.33, "d_1"] \ar[dd]  & & & 
      {A_0} \ar[dd] 
	  \\ 
      {B_{-1}} \ar[rrr, pos=0.66, "\ssplit"'] \ar[dr, "\ssplit"'] & & & 
      {B_0} \ar[rrr, pos=0.66, "\q"'] \ar[dr, "\ssplit"'] & & & 
      {B_{-1}} \ar[dr, "\ssplit"] &
	  \\ &
      {B_0} \ar[rrr, "s_0"'] & & &  
      {B_1} \ar[rrr, "d_1"'] & & &  
      {B_0}
  \end{tikzcd}
  \]
  is easily seen to commute, and it exhibits the square~\eqref{eq:augsplit-pbk}
  as a retract of the middle vertical square in the diagram. But the middle
  square is a pullback since $p$ is a right fibration and by the first argument
  of the proof. Since limits are stable under retracts, it follows that also the
  square~\eqref{eq:augsplit-pbk} is a pullback, as asserted.
\end{proof}

It is a standard fact for $1$-categories that if
$\mathcal{C}'\to\mathcal{C}$ is a right fibration of categories, and if
$\mathcal{C}$ has local initial objects, then also $\mathcal{C}'$ has
local initial objects, and they are just the preimages of the local
initial objects of $\mathcal{C}$. The next lemma generalizes this to
state that $\decbot$-coalgebra structure can be pulled back along right
fibrations. We will need this result again in
Section~\ref{sec:Boors-comp}, where we need it in a slightly more
general version, so we state it in that generality: Let $ \mathcal{E} $
be an $\infty$-category with pullbacks. Then we may speak of right
fibrations in the $\infty$-category $ \Fun(\simplexcategory\op, \mathcal
E) $, and we can also speak of the lower decalage comonad and its
coalgebras.

\begin{lemma} \label{lem:RFibLiftDecCoal}
  Let $ \mathcal E $ be a category with pullbacks and let $ F : C' \to C $ be a
  right fibration in $ \Fun(\simplexcategory\op, \mathcal E)
  $. Then a $ \decbot $-coalgebra structure on $ C $ induces a $ \decbot
  $-coalgebra structure on $ C' $ (which is in fact unique with the property that
  $ F $ becomes a morphism of $ \decbot$-coalgebras).
\end{lemma}

\begin{proof}
  The $\decbot$-coalgebra structure on $C$ is given by the bottom splitting of the
  bar resolution of the comonad $\decbot$, as pictured as the
  (horizontal) solid part of the diagram
  \[
  \begin{tikzcd}[row sep={6em,between origins}, column sep={8em,between origins}]
  C \ar[r, bend right=28, shorten <=7pt, shorten >=6pt, "\gamma"', end anchor=west, start anchor=east, shift left=0.3]& \decbot C 
  \ar[r, bend right=28, shorten <=7pt, shorten >=6pt, "\decbot(\gamma)"', end anchor=west, start anchor=east, shift right=1.6] 
  \ar[l, shift right, "\varepsilon"'] & \Dec_\bot^2 C 
  \ar[l, shift right=2, "\varepsilon_{\decbot}"'] \ar[l, "\decbot\varepsilon"{description,inner sep = .5pt}] \ar[r, "\cdots"{description}, phantom]
  & {} 
  \\
  C' \ar[u, "F"] \ar[r, dotted, bend right=28, shorten <=7pt, shorten >=6pt, 
  "\gamma'"', end anchor=west, start anchor=east, shift left=0.1] & \decbot C' \ar[u, dotted, "\decbot F"] \ar[l, dotted, shift right, 
  "\varepsilon"'] \ar[r, dotted, bend right=28, shorten <=6pt, shorten >=4pt, "\decbot(\gamma')"', end anchor=west, start anchor=east, shift right=1.6] 
  & \Dec_\bot^2 C' 
  \ar[u, dotted] 
  \ar[l, dotted, shift right=2, "\varepsilon_{\decbot}"'] 
  \ar[l, dotted, "\decbot\varepsilon"{description,inner sep = .5pt}] \ar[r, "\cdots"{description}, phantom] & {}
  \end{tikzcd}
  \]
  In the solid diagram, all the $\Dec_\bot^k(\varepsilon)$ are instances 
  of $d_0$, and all the $\varepsilon_{\Dec_\bot^k}$ are instances of $d_k$ (which are active).
  Since $F$ is a right fibration, it forms pullbacks with all of them, which 
  means that when pulling back the whole diagram, a new $\decbot$-coalgebra 
  structure results on $C'$.
\end{proof}

\subsection{Rigid \texorpdfstring{$\decbot$}{decbot}-coalgebras}
\label{subsec:rigid}

Among all $\decbot$-coalgebras, of particular importance for this article are
the rigid $\decbot$-coalgebras:

\begin{definition}\label{rigid}
  A $\decbot$-coalgebra $ \gamma : X \to \decbot(X) $ on a simplicial space $X$
  is said to be {\em rigid} if $\gamma$ is cartesian.
\end{definition}

\begin{remark}\label{rmk:rigid-unit}
  Equivalently, {\em a $\decbot$-coalgebra $A \in \Pr(\simplexcategory^\bb)$ (with underlying structure map $\gamma$) is rigid
  if and only if the $A$-component $\eta_A: A \to \decbotb(A)$ of the unit of the 
  monad $\decbotb$ is cartesian}. 
  Indeed, we have $\gamma \simeq (\ii\kk)\upperstar(\eta_A)$ by
  Observation~\ref{obs:observation}, saying that $\gamma$ is the restriction of 
  $\eta_A$ to the underlying simplicial space. So if $\eta_A$ is cartesian, then 
  also $\gamma$ is cartesian.
  Conversely, if $\gamma$ is cartesian and hence a left
  fibration, then $\eta_A$ is a left fibration (by our convention for left
  fibrations in $\Pr(\simplexcategory^\bb)$), but now Lemma~\ref{lem:lfib=>cart}
  guarantees that $\eta_A$ is actually cartesian.
\end{remark}

The notion of rigidity is motivated by the notion of local initial objects, 
cf.~the next subsection. For $1$-categories (or $\infty$-categories),
$\decbot$-coalgebra structure encodes a choice
of local initial objects~\cite{Garner-Kock-Weber:1812.01750}, but
for more
general simplicial spaces, the further property of rigidity
must be imposed separately in order to get a reasonable
notion of local initial objects, as we shall see.
The reason the rigidity condition does not turn up in category 
theory~\cite{Garner-Kock-Weber:1812.01750} is that it is automatic for Segal 
spaces:

\begin{lemma}\label{lem:1-Segal-is-rigid}
  For a Segal space $X$, any $\decbot$-coalgebra structure is rigid.
\end{lemma}

\begin{proof}
  Let $\gamma : X \to \decbot(X)$ be a $ \decbot$-coalgebra. The counit axiom for
  coalgebras gives the equivalence $\varepsilon \circ \gamma \simeq \Id$. Now $X$ is
  Segal if and only if $\varepsilon$ is a left fibration (\ref{cheat:counits}), so in
  this case also $\gamma$ is a left fibration. On the other hand, by
  Observation~\ref{obs:observation}, $\gamma$ is actually a morphism of
  bottom-split simplicial spaces, and we know from Lemma~\ref{lem:lfib=>cart} that
  a left fibration of bottom-split simplicial spaces is in fact cartesian.
\end{proof}

\begin{example}
  Let $Y$ be lower $2$-Segal, and consider the $1$-Segal space
  $X:= \Decbot{Y}$. Then $X$ has a canonical
  extra bottom degeneracy map $\ssplit: X \to \Decbot{X}$, and this is cartesian.
\end{example}

For the next lemma, recall that $\decbot$-coalgebra structure can be pulled
back along cartesian maps (see Lemma~\ref{lem:RFibLiftDecCoal}). It turns out
the rigidity accounts for all the higher coherences in the notion of
$\decbot$-coalgebra:

\begin{lemma} \label{lem:CartCounit=>Coalg}
  Let $ \gamma : X \to \decbot X $ be a simplicial map satisfying the
  equation $ \varepsilon \circ \gamma \simeq 1 $. If $ \gamma $ is
  cartesian, then $\gamma$ is obtained by pulling back the canonical
  $\decbot$-coalgebra structure on $\decbot(X)$ along $\gamma$. In particular,
  $\gamma$ is a rigid $\decbot$-coalgebra structure.
\end{lemma}

\begin{proof}
  First of all, since $\gamma$ is cartesian, it follows from
  Lemma~\ref{lem:RFibLiftDecCoal} that
  $X$ inherits a $\decbot$-coalgebra structure $\ssplit : X \to \decbot X $ from
  the canonical one on $\decbot X$, given by the comultiplication of the comonad
  $ \delta : \decbot X \to \Dec_\bot^2 X$. According to the same
  lemma, the top square in the diagram
  \begin{center}
    \begin{tikzcd}
      X & {\decbot X} \\
      {\decbot X} & {\Dec^2_\bot{}\kern-2pt X} \\
      & {\decbot X}
      \arrow["{s_{\subbot}}", from=1-1, to=1-2]
      \arrow["\gamma"', from=1-1, to=2-1]
      \arrow["{\decbot (\gamma)}"', from=1-2, to=2-2]
      \arrow[bend left=50, Rightarrow, no head, from=1-2, to=3-2]
      \arrow["{\delta}"', from=2-1, to=2-2]
      \arrow[bend right=25, Rightarrow, no head, from=2-1, to=3-2]
      \arrow["{\decbot(\varepsilon)}"', from=2-2, to=3-2]
    \end{tikzcd}
  \end{center}
  commutes (and is in fact a pullback). By the commutativity of the whole
  diagram we see that $\gamma$ is in fact equivalent to $\ssplit$, so since
  $\ssplit$ is a coalgebra structure, also $\gamma$ is a coalgebra structure.
\end{proof}

\subsection{Local initial objects}

The property characterizing an initial object $x$ in a Segal space $X$ is that the
projection $X_{x/}\to X$ is an equivalence. The coslice under $x$ is formally
defined as the pullback in the diagram below, and the projection required to be an
equivalence is the vertical composite in the diagram
\[
\begin{tikzcd}
\bar 1 \ar[d, "\bar x"'] & X_{x/} \dlpullback \ar[l] \ar[d]  \\
\bar X_0 & \decbot(X) \ar[l, "\augmap"] \ar[d, "\varepsilon"'] \\
& X  .
      \arrow[bend left=50, Rightarrow, no head, from=1-2, to=3-2]
\end{tikzcd}
\]
If it exists, an initial object is essentially unique.
(The resulting map $\gamma: X \to \decbot(X)$ is the corresponding
$\decbot$-coalgebra structure.)
For {\em local} initial objects, there is instead an initial object in each connected 
component, and the situation is rather
\begin{equation}\label{C->X}
\begin{tikzcd}
\bar C \ar[d, "\bar a"'] & X_{a/} \dlpullback \ar[l] \ar[d]  \\
\bar X_0 & \decbot(X) \ar[l, "\augmap"] \ar[d, "\varepsilon"'] \\
& X  ,
      \arrow[bend left=50, Rightarrow, no head, from=1-2, to=3-2]
\end{tikzcd}
\end{equation}
for some pointing $a:C\to X_0$.
Again, the choice of a pointing with this property is 
essentially unique, and by absoluteness, $C$ must be the geometric realization 
of $X$. This definition is appropriate also for general simplicial spaces:

\begin{definition} \label{def:InitPoint}
   A {\em local-initial-objects structure} on a simplicial space $X$ consists of
   a pointing $a: C \to X_0$ for which the map
   $$
   \bar C \times_{\bar{X}_0} \decbot(X) \stackrel{\text{pr}}{\longrightarrow} \decbot (X) \stackrel{\varepsilon}{\longrightarrow} X
   $$
   is an equivalence (as in diagram~\eqref{C->X}).

   Dually, a pointing $a: C \to X_0$ is said to be a {\em local-terminal-objects structure} on $X$ if
  $\bar C \times_{\bar{X}_0} \dectop(X) \stackrel{\text{pr}}{\longrightarrow}
	\dectop (X) \stackrel{\varepsilon}{\longrightarrow} X$
  is an equivalence.  
\end{definition}

Once we pass to the bisimplicial setting, this condition will appear in the pointing 
axiom~\ref{def:PointingAxiom} of 
BOORS~\cite{Bergner-Osorno-Ozornova-Rovelli-Scheimbauer:1609.02853}, 
\cite{Bergner-Osorno-Ozornova-Rovelli-Scheimbauer:1901.03606}.

We now head towards proving that local-initial-objects structures 
and rigid $\decbot$-coalgebras are in fact equivalent notions. The underlying combinatorial shape of local-initial-objects structures is captured
by the category $\psimpcat$, given by adjoining a terminal object to $\simplexcategory$. In terms of generators and relations, $
\psimpcat $ has, in addition to the generators and relations of $
\simplexcategory$, an object $ [-1] $ and a new morphism $ [-1] \leftarrow [0] $. We call
the objects in $ \PrSh(\psimpcat) $ {\em pointed simplicial spaces}. A pointed
simplicial space $X$ thus consists of a simplicial space, also denoted by $X$, and
a map $a : C \to X_0 $ called the {\em pointing}. A convenient way of organizing
the data of such a pointed simplicial space is by a simplicial map $ \bar C \to X
$.

The comparison of rigid $\decbot$-coalgebras and local-initial-objects structures
spaces is based on the inclusion of categories
\begin{center}
  \begin{tikzcd}[column sep={5em,between origins}]
	\psimpcat && {[-1]} & {[0]} & {[1]} & {[2]} & \cdots \\
	{\simplexcategory^\bb} && {[-1]} & {[0]} & {[1]} & {[2]} & \cdots
	\arrow["\hh"', from=1-1, to=2-1]
	\arrow[shift left=1, bend left=20, shorten <=5pt, shorten >=5pt, end anchor=east, start anchor=west, from=1-4, to=1-3]
	\arrow[shift right=2, from=1-4, to=1-5]
	\arrow[shift left=2, from=1-4, to=1-5]
	\arrow[shorten <=5pt, shorten >=5pt, from=1-5, to=1-4]
	\arrow[from=1-5, to=1-6]
	\arrow[shift left=4, from=1-5, to=1-6]
	\arrow[shift right=4, from=1-5, to=1-6]
	\arrow[shift left=2, shorten <=5pt, shorten >=5pt, from=1-6, to=1-5]
	\arrow[shift right=2, shorten <=5pt, shorten >=5pt, from=1-6, to=1-5]
	\arrow[from=2-3, to=2-4]
	\arrow[shift left=1, bend left=20, shorten <=5pt, shorten >=5pt, end anchor=east, start anchor=west, from=2-4, to=2-3]
	\arrow[shift left=2, from=2-4, to=2-5]
	\arrow[shift right=2, from=2-4, to=2-5]
	\arrow[shorten <=5pt, shorten >=5pt, from=2-5, to=2-4]
	\arrow[shift left=2, bend left=20, shorten <=5pt, shorten >=5pt, from=2-5, to=2-4]
	\arrow[from=2-5, to=2-6]
	\arrow[shift left=4, from=2-5, to=2-6]
	\arrow[shift right=4, from=2-5, to=2-6]
	\arrow[shift left=2, shorten <=5pt, shorten >=5pt, from=2-6, to=2-5]
	\arrow[shift right=2, shorten <=5pt, shorten >=5pt, from=2-6, to=2-5]
	\arrow[shift left=4, bend left=20, shorten <=5pt, shorten >=5pt, from=2-6, to=2-5]
\end{tikzcd}
\end{center}
and the induced adjunction
\begin{center}
    \begin{tikzcd}[column sep=large]
        {\text{Pr}(\simplexcategory^\bb)} & {\text{Pr}(\psimpcat)}.
        \arrow["{\hh \upperstar }", "{\bot}"', shift left=2, from=1-1, to=1-2]
        \arrow["{\hh \lowerstar }", shift left=2, from=1-2, to=1-1]
    \end{tikzcd}
\end{center}
We first describe $\hh\lowerstar$.
Given a simplicial space $X$, we can consider the bottom-split augmented
simplicial space $\uu\upperstar(X)$; recall that $\uu: 
\simplexcategory^\bb \to \simplexcategory$ is the forgetful functor, so 
that pulling back along it amounts to removing the bottom face
maps, but not the bottom degeneracy maps, and shifting down everything so that the original degree 0
becomes the new degree $-1$. There is a canonical morphism of bottom-split augmented
simplicial space $\bar X_0 \stackrel{\iota}{\longleftarrow} \uu\upperstar(X)$, which in degree $-1$ is given by the identity. The
right Kan extension 
of a pointed simplicial space $ \bar C
\to X $, denoted by $ \hh\lowerstar(X)$ for short, can now be described as the pullback of the entire
$\simplexcategory^\bb$-presheaf $\uu\upperstar (X)$ along the pointing $ a : 
C \to X_0 $:

\begin{center}
  \begin{tikzcd}
    {\bar C} & {\dlpullback \hh\lowerstar(X)} \\
    {\bar X_0} & {\uu\upperstar(X)} ,
    \arrow["{\bar a}"', from=1-1, to=2-1]
    \arrow[from=1-2, to=1-1]
    \arrow[from=1-2, to=2-2]
    \arrow["\iota", from=2-2, to=2-1]
  \end{tikzcd}
\end{center}
where $\bar a$ is a map of constant bottom-split simplicial spaces. Expanding
these bottom-split augmented simplicial spaces, the construction can be depicted
in more detail as in the diagram
\begin{center}
  \begin{tikzcd}
    {C} & {\dlpullback \hh\lowerstar(X)_0} & {\dlpullback \hh\lowerstar(X)_1} & \cdots \\
    {X_0} & {X_1} & {X_2} & \cdots
    \arrow[bend right=20, end anchor=west, start anchor=east, shift right=1, shorten <=5pt, shorten >=5pt, from=1-1, to=1-2]
    \arrow["a"', from=1-1, to=2-1]
    \arrow[from=1-2, to=1-1]
    \arrow[bend right=20, end anchor=west, start anchor=east, shift right=2, shorten <=5pt, shorten >=3pt, from=1-2, to=1-3]
    \arrow[from=1-2, to=2-2]
    \arrow[shift right, from=1-3, to=1-2]
    \arrow[shift left, from=1-3, to=1-2]
    \arrow[from=1-3, to=2-3]
    \arrow["{s_0}"', end anchor=west, start anchor=east, bend right=20, shift right=1, shorten <=5pt, shorten >=5pt, from=2-1, to=2-2]
    \arrow["{{{d_1}}}"', from=2-2, to=2-1]
    \arrow["{s_0}"', end anchor=west, start anchor=east, shift right=2, bend right=20, shorten <=5pt, shorten >=5pt, from=2-2, to=2-3]
    \arrow["{{{d_2}}}"', shift right, from=2-3, to=2-2]
    \arrow["{{{d_1}}}"{description, inner sep=.5pt}, shift left=1, from=2-3, to=2-2]
  \end{tikzcd}
\end{center}
where $\hh\lowerstar(X)_{-1} \simeq C $.

The counit $ \hh\upperstar \hh\lowerstar \to 1 $ of the adjunction evaluated at a
pointed simplicial space $ X $ is given by 
$\id_C$ on the augmentation and by
$$
  \bar C \times_{\bar{X}_0} \decbot(X) \stackrel{\text{pr}}{\longrightarrow}
  \decbot (X) \stackrel{\varepsilon}{\longrightarrow} X
$$
on the underlying simplicial space.

With the adjunction $\hh\upperstar \isleftadjointto\hh\lowerstar$ at hand we can now
reformulate the definition of local-initial-objects structure as in the following
lemma, which is a tautology.

\begin{lemma}\label{lem:h-counit}
  A pointed simplicial space $\bar C \to X$ is a local-initial-objects structure on $X$
  if and only if the counit of the adjunction $\hh\upperstar \isleftadjointto\hh\lowerstar$ evaluated at $X$ is invertible.
\end{lemma}

\begin{lemma}
  Let $ A \in \PrSh(\simplexcategory^\bb) $ be a rigid $ \decbot $-coalgebra. Then
  $ \hh\upperstar (A) $ is a local-initial-objects structure on the underlying simplicial space of $A$.
\end{lemma}

\begin{proof}
  Let $X := (\ii\kk)^*(A)$ denote the underlying simplicial space of $A$ and let $
  \gamma : X \to \decbot X $ denote the corresponding $\decbot$-coalgebra
  structure map. We must show that the composite
  $$
    \bar A_{-1} \times_{\bar{A}_0} \decbot(X) \stackrel{\text{pr}}{\longrightarrow} \decbot (X) \stackrel{\varepsilon}{\longrightarrow} X
  $$
  is invertible. The fact that $A$ is rigid means that the pullback projection $
  \bar A_{-1} \times_{\bar{A}_0} \decbot(A')
  \stackrel{\text{pr}}{\longrightarrow} \decbot (A') $ is computed by $ \gamma :
  X \to \decbot X $, as can be seen in the diagram
  \begin{center}
    \begin{tikzcd}
      {A_{-1}} & {\dlpullback A_0} & {\dlpullback A_1} & \cdots \\
      {A_0} & {A_1} & {A_2} & \cdots
      \arrow["{\gamma_{-1}}"', from=1-1, to=2-1]
      \arrow[from=1-2, to=1-1]
      \arrow["{\gamma_0}"', from=1-2, to=2-2]
      \arrow["{d_0}", shift left, from=1-3, to=1-2]
      \arrow["{d_1}"', shift right, from=1-3, to=1-2]
      \arrow["{\gamma_1}", from=1-3, to=2-3]
      \arrow[from=2-2, to=2-1]
      \arrow["{d_1}", shift left, from=2-3, to=2-2]
      \arrow["{d_2}"', shift right, from=2-3, to=2-2]
    \end{tikzcd}
  \end{center}
  As a result, the above composite reduces to
  $$
    X \stackrel{\gamma}{\longrightarrow} \decbot X 
	\stackrel{\varepsilon}{\longrightarrow} X ,
  $$
  which is equivalent to the identity (by the $\decbot$-coalgebra 
  counit law) and thereby invertible.
\end{proof}

\begin{lemma}
  Let $a : C \to X_0$ be a local-initial-objects structure on a simplicial space
  $X$. Then $ \hh\lowerstar (X) $ is a rigid $ \decbot $-coalgebra.
\end{lemma}

\begin{proof}
  The assumption that $\bar C \to X$ is a local-initial-objects structure says
  that we have a pullback square
  \begin{equation}
    \begin{tikzcd}
    \bar C \ar[d, "\bar a"']  & X \ar[d, "\gamma"] \ar[l] \dlpullback  \\
    \bar X_0 & \decbot X \ar[l, "\augmap"]
    \end{tikzcd}
  \end{equation}
  where $\gamma$ satisfies the unit law $\varepsilon \circ \gamma \simeq 1$. Since $
  \bar a $ is cartesian so is $\gamma$. But, by definition, the pullback is
  precisely the underlying simplicial space of $\hh\lowerstar(X)$, and the
  $\decbot$-coalgebra structure on $\hh\lowerstar(X)$ is precisely that induced by
  pulling back the canonical $\decbot$-coalgebra structure on $\decbot(X)$ along $\gamma$.
  By Lemma~\ref{lem:CartCounit=>Coalg}, the $\decbot$-coalgebra structure map of
  $\hh\lowerstar(X)$ is identified with $\gamma$, rendering $\hh\lowerstar(X)$
  into a rigid $\decbot$-coalgebra.
\end{proof}

Thanks to these two lemmas we may restrict the functors $\hh\upperstar
\isleftadjointto\hh\lowerstar$ to an adjunction between the full
subcategory of rigid $\decbot$-coalgebras and the full subcategory of
local-initial-objects structures. Lemma~\ref{lem:h-counit} characterizes
the fixpoint locus of the counit. We now analyze the unit.

The unit $ 1 \to \hh\lowerstar \hh\upperstar $ evaluated on $ A \in
\PrSh(\simplexcategory^\bb) $ is given by
$$
  A \longrightarrow \bar A_{-1} \times_{\bar A_0} \decbotb A ,
$$
where $ \bar A_{-1} $ and $ \bar A_0 $ are interpreted as $
\simplexcategory^\bb$-presheaves. This map is induced by the universal property 
of the pullback, as detailed in the diagram

\begin{equation} \label{diag:hUnit}
  \begin{tikzcd}
    && A \\
    {\bar A_{-1}} & {\dlpullback \hh\lowerstar\hh\upperstar(A)} \\
    {\bar A_0} & {\decbotb(A) .}
    \arrow["\iota"', bend right=20, from=1-3, to=2-1]
    \arrow[dashed, from=1-3, to=2-2]
    \arrow[bend left=30, from=1-3, to=3-2]
    \arrow["{\bar s_\subbot}"', from=2-1, to=3-1]
    \arrow[from=2-2, to=2-1]
    \arrow[from=2-2, to=3-2]
    \arrow["\iota", from=3-2, to=3-1]
  \end{tikzcd}
\end{equation}
Here the arrows labeled $\iota$ are the 
canonical maps of augmented (bottom-split) simplicial spaces given by the identity in 
degree $-1$, and where the map $A \to \decbotb(A)$ is the unit for the $\decbotb$ 
monad; it is given by $\ssplit$ in each degree, so 
the outer square commutes.

\begin{lemma}\label{lem:h-unit}
  The unit of the adjunction $\hh\upperstar \isleftadjointto\hh\lowerstar$ evaluated on a
  $\decbot$-coalgebra $ A $ is an equivalence if and only if $ A $ is
  rigid.
\end{lemma}

\begin{proof}
  The unit of $\hh\upperstar \isleftadjointto\hh\lowerstar$ is the dashed arrow in 
  diagram~\eqref{diag:hUnit}, so it is an equivalence if and only if the 
  outer square is a pullback, which is to say that 
  the map $A \to \decbotb(A)$ is cartesian. But according to Remark~\ref{rmk:rigid-unit} 
  this
  is precisely to say that $A$ is rigid.
\end{proof}

The last four lemmas (\ref{lem:h-counit}--\ref{lem:h-unit})
together establish the following result.

\begin{prop}\label{prop:RigCoalgIsRigPoint}
  The adjunction $\hh\upperstar \isleftadjointto\hh\lowerstar$ restricts to an equivalence
  of full subcategories
  \begin{center}
    \begin{tikzcd}
      {\{\text{rigid } \decbot\text{-coalgebras}\}} & 
	  {\{\text{local-initial-objects structures}\} .}
      \arrow["{\hh\upperstar }", shift left, from=1-1, to=1-2]
      \arrow["{\hh\lowerstar }", shift left, from=1-2, to=1-1]
    \end{tikzcd}
  \end{center}
\end{prop}

\begin{example}
  Let $ X $ be a lower 2-Segal simplicial space. Then by definition $\decbot(X)$
  is Segal~(\ref{upper-lower-2Segal}). Now, $\decbot(X)$ supports a canonical $\decbot$-coalgebra structure
  given by the comonad comultiplication. Since $\decbot(X)$ is $1$-Segal, it
  follows from Lemma~\ref{lem:1-Segal-is-rigid} that $\decbot(X)$ is a rigid
  $\decbot$-coalgebra. Its restriction to $\psimpcat$ is a local-initial-objects
  structure on $X$ by Proposition~\ref{prop:RigCoalgIsRigPoint}.
\end{example}

\begin{remark}
  According to Proposition~\ref{prop:RigCoalgIsRigPoint}, the data of local initial
  objects can be understood from two different viewpoints. While the notion in
  terms of a pointing is more economical, the viewpoint of rigid 
  $\decbot$-coalgebras has other advantages by relating to various functorialities 
  of interest, as exploited in this article.
\end{remark}

\section{Augmented bisimplicial spaces with abacus maps}

\label{sec:bisimplicial}

From now on we will be considering bisimplicial spaces, i.e.~presheaves on $
\simplexcategory\times\simplexcategory$. For a bisimplicial space 
$B : (\simplexcategory\times\simplexcategory)\op\to\spaces$, 
we use matrix convention for rows and columns, but starting of course the 
indexations from $0$. So $B_{\bullet,0}$ is referred to as the zeroth column
and $B_{0,\bullet}$ as the zeroth row. (Note that the 
opposite row-column convention is common in algebraic topology (see for 
example \cite{Bousfield-Friedlander:10.1007/BFb0068711}). We have 
chosen to follow the convention used by BOORS and Carlier.)
Following Carlier~\cite{Carlier:1801.07504}
we use the convention of using the letters
\begin{quotation}
$d$ and $s$ for face and degeneracy maps in the horizontal 
direction (second index)

\smallskip

$e$ and $t$ for face and 
degeneracy maps in the vertical direction (first index).

\end{quotation}

\subsection{Stability} \label{subsec:stab}

A bisimplicial space is called {\em upper stable} if all $e_\bot$ and $d_{\bot}$
form pullbacks with each other; it is called {\em lower stable} if all $e_\top$
and $d_\top$ form pullbacks with each other; it is called {\em stable} if it is
both upper stable and lower stable. Upper stable can also be formulated as
saying that all $e_\bot$ (or equivalently all $d_\bot$) are right fibrations;
lower stable can be formulated as saying that all $e_\top$ (or equivalently all
$d_\top$) are left fibrations.
  
\begin{remark}
  The stability condition is due to
  BOORS~\cite{Bergner-Osorno-Ozornova-Rovelli-Scheimbauer:1609.02853}, see also
  Carlier~\cite{Carlier:1801.07504}. The separate notions of lower and upper
  stable are required in a few places below. It may be a bit confusing that lower
  stable relates to top face maps while upper stable relates to bottom face maps,
  but the terminology matches the usage for upper and lower $2$-Segal conditions
  under the chosen convention for total decalage (cf.~\ref{bla:Tot}), as 
  expressed by this straightforward lemma:
\end{remark}

\begin{lemma}
  If a simplicial space $X$ is $2$-Segal (respectively upper $2$-Segal, 
  respectively lower $2$-Segal), then $\Tot(X)$ is stable (respectively upper 
  stable, respectively lower stable).
\end{lemma}

\begin{lemma}\label{lem:stable}
  Let $B$ be a bisimplicial space.
  \begin{enumerate}
	\item
  If $B$ is upper stable, then
  
  ---
  all horizontal active maps are right fibrations (considered
as simplicial maps between columns),

--- all vertical active maps are right 
fibrations (considered
as simplicial maps between rows).

\item If $B$ is lower stable, then

---
all horizontal active maps are left fibrations (considered
as simplicial maps between columns),

--- all vertical active maps are left 
fibrations (considered
as simplicial maps between rows).
  
\item If $B$ is stable, then

--- every vertical active map is cartesian
  (considered as a simplicial map between rows),
  
  --- every horizontal active map 
  is cartesian (considered as a simplicial map between columns).
  \end{enumerate}
\end{lemma}

\begin{proof}
  This follows from standard pullback-prism-lemma arguments (in analogy with the
  so-called bonus pullbacks holding for decomposition
  spaces~\cite[Lemma 3.10]{Galvez-Kock-Tonks:1512.07573}).
\end{proof}

\begin{definition}
  A bisimplicial space is called {\em double Segal} (resp.~{\em double 
  $2$-Segal}) if every row and every 
  column is a $1$-Segal space (resp.~a 
  $2$-Segal space).
\end{definition}

\begin{lemma}[Carlier~{\cite[Lemma~2.3.3]{Carlier:1801.07504}}]
  If a bisimplicial space is double Segal, then for it to be stable
  it is enough for the two squares
  \begin{center}
    \begin{tikzcd}
      {B_{0,0}} & {B_{0,1}} \\
      {B_{1,0}} & {\ulpullback B_{1,1}}
      \arrow["{d_0}"', from=1-2, to=1-1]
      \arrow["{e_0}", from=2-1, to=1-1]
      \arrow["{e_0}"', from=2-2, to=1-2]
      \arrow["{d_0}", from=2-2, to=2-1]
    \end{tikzcd}
    \qquad\qquad
    \begin{tikzcd}
      {B_{0,0}} & {B_{0,1}} \\
      {B_{1,0}} & {\ulpullback B_{1,1}}
      \arrow["{d_1}"', from=1-2, to=1-1]
      \arrow["{e_1}", from=2-1, to=1-1]
      \arrow["{e_1}"', from=2-2, to=1-2]
      \arrow["{d_1}", from=2-2, to=2-1]
    \end{tikzcd}
  \end{center}
  to be pullbacks.
  (To check upper stability, it is enough to check the square on the left, and 
  to check lower stability, it is enough to check the square on the right.)
\end{lemma}

\subsection{Augmentations}

Let $B$ be a bisimplicial space.
Then by taking the geometric realization of each row we obtain an 
augmentation column $X = B_{\bullet,-1}$. The simplicial operators 
inside $X$ are induced by the vertical simplicial operators of 
$B$ and the universal property of the geometric 
realization. Similarly, by taking geometric realization of each column, we 
obtain an augmentation row $Y = B_{-1,\bullet}$, where  
the simplicial operators inside $Y$ are induced by the horizonal 
simplicial operators of $B$.

The total shape of such a diagram (see picture~\eqref{bla:Delta/[1]} in
the introduction) is that of a presheaf on
$$
(\simplexcategory_+ \times \simplexcategory_+) \setminus ([-1],[-1]) \simeq 
\simplexcategory_{/[1]} 
$$
(see Carlier~\cite{Carlier:1801.07504}). 
Under this isomorphism, an object $[i,j]$ on the left corresponds to
the map  $(s^\bot)^{\circ i} (s^\top)^{\circ j} : [i{+}1{+}j]\to [1]$ in $\simplexcategory$
on the right.
Here one of $i$ and $j$, but not both, 
could be equal to $-1$, corresponding to an empty fibre.
The left-hand category motivates our notation (which agrees with 
that of $\simplexcategory\times\simplexcategory$), whereas the
$\simplexcategory_{/[1]}$ viewpoint is more elegant, and becomes 
relevant from Subsection~\ref{ssec:CoCartCorr}.

\begin{prop}\label{prop:colim-aug}
  If a stable double $2$-Segal space $B$ is endowed with
  colimit augmentation column $X$ and colimit augmentation row
  $Y$, then $X$ and $Y$
  are again $2$-Segal, and the augmentation maps are culf.
\end{prop}

\begin{proof}
  Let $a:Y_k \to Y_n$ be an active map of $Y$. It is induced by the
  corresponding simplicial map $a_{\bullet}:B_{\bullet,k} \to B_{\bullet,n}$
  between columns. Since $B$ is stable, the
  simplicial map $a_\bullet$ is cartesian by Lemma~\ref{lem:stable}. 
  By descent (see
  Lurie~\cite[Theorem 6.1.3.9, Proposition 6.1.3.10]{Lurie:HTT}), this implies
  that also the augmentation square
  \[
  \begin{tikzcd}
  Y_k \ar[r, "a"] & Y_n  \\
  B_{0,k} \ar[u] \ar[r, "a_0"'] & B_{0,n} \ar[u]
  \end{tikzcd}
  \]
  is a pullback. Since this argument works for every active map, this means
  that the simplicial map $B_{0,\bullet} \to Y$ is culf. 
  
  It is a general fact that the augmentation map from a simplicial space to its
  geometric realization is an effective epimorphism. To see this, note first that
  the induced map from the space of all simplices $ \bigsqcup_{[n] \in
  \simplexcategory} B_{\bullet, n} \to Y_n $ is an effective epimorphism
  (by~\cite[6.2.3.13]{Lurie:HTT}), but this map factors through the augmentation
  map $B_{0,n} \to Y_n$ which must therefore be an effective epimorphism too
  (see~\cite[6.2.3.12]{Lurie:HTT}). Since the augmentation maps are effective
  epimorphisms, and since the zeroth row is $2$-Segal, it follows
  from~\ref{cheat:culf+effepi} that also $Y$ is $2$-Segal.
\end{proof}

\begin{remark} \label{rem:colim-aug}
  Note that both lower stable and upper stable are required in the proposition,
  because this is what makes the active maps cartesian (\ref{lem:stable}) so that we can apply
  descent. In the applications of Proposition~\ref{prop:colim-aug}, the bisimplicial
  space $B$ will actually be $1$-Segal in each column and each
  row, but the augmentation row and column will still only be $2$-Segal. A key
  example of this situation is when $B = \Tot(X)$ for $X$ a
  $2$-Segal space.
\end{remark}

\begin{definition}\label{bimod-conf} (Carlier~\cite{Carlier:1801.07504})
  A {\em bicomodule configuration} is a presheaf
  $ B \in \PrSh(\simplexcategory_{/[1]}) $ whose underlying bisimplicial space
  is stable, double Segal, 
  where the augmentation row and column are $2$-Segal and where the
  augmentation maps are culf.
\end{definition}

\begin{remark}
  Bicomodule configurations are so named because they induce
  bicomodules in the $\infty$-category of spaces and spans
  \cite{Carlier:1801.07504}, \cite{Godicke:2407.13357}. The role of the
  augmentations, which we often denote $X:= B_{\bullet,-1}$ and $Y:=
  B_{-1,\bullet}$, is to encode the coalgebras that $B$ is a bicomodule
  over: being a left $X$-comodule and a right $Y$-comodule requires the
  $1$-Segal and culf conditions imposed
  (cf.~Walde~\cite{Walde:1611.08241} and Young~\cite{Young:1611.09234}),
  whereas the bicomodule condition is expressed by stability,
  cf.~Carlier~\cite{Carlier:1801.07504}.

  We should warn that the general morphisms of
  $\simplexcategory_{/[1]}$-presheaves maps are not the ``correct''
  notion of morphism of bicomodule configurations: the correct morphisms
  corresponding to $X$-$Y$-bicomodule homo\-morphisms are certain spans of
  $\simplexcategory_{/[1]}$-presheaves where the left leg is a right
  fibration on columns and a left fibration on rows, and where the
  right leg is a left fibration on columns and a right fibration on
  rows. The most general result in this direction is G\"odicke's
  theorem~\cite{Godicke:2407.13357}, which establishes an equivalence of
  $\infty$-categories between bicomodule configurations and bicomodules
  in spans, having first identified the correct morphisms.
\end{remark}

\subsection{The abacus category \texorpdfstring{$\DD$}{D}}

\label{sub:DD}

In this subsection we describe the category $\DD$
that captures the combinatorics of bicomodule configurations with abacus maps. 
The category $\DD$ was first studied by Carlier~\cite{Carlier:1801.07504}, as far as 
we know (he denotes it $\overline{\Delta}_{/[1]}$). Carlier~\cite[\S3.2]{Carlier:1801.07504} defined it in terms of certain mapping cylinders and 
cocartesian fibrations. 
We give a slightly more elementary definition, more closely reflecting 
Carlier's suggested combinatorial interpretation in terms of black and white 
beads (\cite[3.2.1]{Carlier:1801.07504}). The elementary combinatorial/graphical 
interpretation is useful in practice to quickly verify identities in $\DD$.

\begin{definition}
  We define $\DD$ formally as a full subcategory of the arrow category of
  $\simplexcategory_+$: the objects are the maps $(d^\top)^{\circ (1+j)} : [i]
  \to [i+1+j] $ for all $ i,j \in \simplexcategory_+ $ except for the case $i =
  j = -1$. The object $(d^\top)^{\circ(1+j)} : [i] \to [i+1+j]$ will be denoted
  by $[i,j]$.
  
  The morphisms $[i,j] \to [i',j']$ in $\DD$ are thus commutative diagrams
    \begin{center}
	\begin{tikzcd}
	  {[i]} & {[i']} \\
	  {[i{+}1{+}j]} & {[i'{+}1{+}j']}  ,
	  \arrow[from=1-1, to=1-2]
	  \arrow["{(d^\top)^{\circ(1+j)}}"', from=1-1, to=2-1]
	  \arrow["{(d^\top)^{\circ (1+j')}}", from=1-2, to=2-2]
	  \arrow[from=2-1, to=2-2]
	\end{tikzcd}
  \end{center}
  where the horizontal arrows are monotone maps.
\end{definition}

\begin{definition}
  The category $\DD$ contains the category $\simplexcategory_{/[1]} \simeq
  (\simplexcategory_+\times \simplexcategory_+) \setminus \{(-1,-1)\}$, and they
  have the same objects.
  The {\em vertical coface maps} in $\DD$ are given by
  \begin{center}
    $e^k : [i,j] \to [i{+}1,j]$ \qquad \qquad
    \begin{tikzcd}
      {[i]} & {[i+1]} \\
      {[i{+}1{+}j]} & {[(i{+}1){+}1{+}j]}  .
      \arrow["{d^k}", from=1-1, to=1-2]
      \arrow["{(d^\top)^{\circ(1+j)}}"', from=1-1, to=2-1]
      \arrow["{(d^\top)^{\circ ( 1+j)}}", from=1-2, to=2-2]
      \arrow["{d^k}"', from=2-1, to=2-2]
    \end{tikzcd}
    \qquad\qquad $0\leq k\leq i+1$
  \end{center}
  The {\em horizontal coface maps} in $\DD$ are given by
  \begin{center}
    $d^k : [i,j] \to [i,j{+}1]$ \qquad\qquad
    \begin{tikzcd}
      {[i]} & {[i]} \\
      {[i{+}1{+}j]} & {[i{+}1{+}(j{+}1)]}  .
      \arrow[Rightarrow, no head, from=1-1, to=1-2]
      \arrow["{(d^\top)^{\circ(1+j)}}"', from=1-1, to=2-1]
      \arrow["{(d^\top)^{\circ (1+j+1)}}", from=1-2, to=2-2]
      \arrow["{d^{k+i+1}}"', from=2-1, to=2-2]
    \end{tikzcd}
    \qquad\qquad
    $0 \leq k \leq j+1$
  \end{center}
  The vertical and horizontal codegeneracy maps can be described similarly.
  Following the global convention we use the letter $s$ for the horizontal
  codegeneracy maps and the letter $t$ for the vertical codegeneracy maps.
\end{definition}

A distinctive feature of $\DD$ are the abacus maps, which we now 
describe.

\begin{definition}
  The {\em abacus maps} are the maps in $\DD$ given by
  \begin{center}
	$f: [i,j+1] \to [i+1,j]$ \qquad 
	\qquad
	\begin{tikzcd}
	  {[i]} & {[i+1]} \\
	  {[i{+}1{+}(j{+}1)]} & {[(i{+}1){+}1{+}j]}  .
	  \arrow["{d^\top}", from=1-1, to=1-2]
	  \arrow["{(d^\top)^{\circ(1+j+1)}}"', from=1-1, to=2-1]
	  \arrow["{(d^\top)^{\circ (j+1)}}", from=1-2, to=2-2]
	  \arrow[Rightarrow, no head, from=2-1, to=2-2]
	\end{tikzcd}
  \end{center}
\end{definition}

There is a useful graphical interpretation of $\DD$ (mentioned in passing by
Carlier~\cite[Remark~3.2.1]{Carlier:1801.07504}) in terms of black and white
beads: the objects in $\DD$ are columns of beads, one for each element in the
codomain $[i+1+j]$, first black then white: the object $[i,j]$ has $i+1$ black
beads followed by $j+1$ white beads, such as
\begin{center}
  \begin{tikzpicture}
	\begin{scope}[shift={(0,0)},decoration={ markings, mark=at position 0.55 with {\arrow{>}}}]
	  \draw (0,0.0) pic {blackdot};
	  \draw (0,0.3) pic {blackdot};
	  \draw (0,0.6) pic {whitedot};
	  \draw (0,0.9) pic {whitedot};
	  \draw (0,1.2) pic {whitedot};
	\end{scope}
  \end{tikzpicture}
\end{center}
representing the object $[1,2]$. (We read from the bottom to the top.)
The black beads are the elements in the image of the
map (that is, the elements in $[i]$), whereas the white beads are the elements in
$[i+1+j]$ that are not in the image. 
The graphical interpretation of 
a map in $\DD$ is that it is monotone and 
maps black beads to black beads, but may map white beads to 
black beads. The maps that preserve
colors are the maps that belong to the subcategory $\simplexcategory_{/[1]}$.
The maps that are the identity on white beads are the vertical maps;
the maps that are the identity on black beads are the horizontal maps, such as 
for example

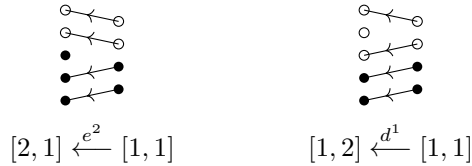
\vspace*{4pt}
\begin{center}
  \begin{tikzpicture}
	\begin{scope}[shift={(0,0)},decoration={ markings, mark=at position 0.6 with {\arrow{>}}}]
	  \draw[postaction={decorate}] (0,0) pic {blackdot} -- (-0.7, -0.15) pic {blackdot}; 
	  \draw[postaction={decorate}] (0,0.3) pic {blackdot} -- (-0.7, 0.15) pic {blackdot}; 
	  \draw (-0.7, 0.45) pic {blackdot};
	  \draw[postaction={decorate}] (0,0.6) pic {whitedot} -- (-0.7, 0.75) pic {whitedot}; 
	  \draw[postaction={decorate}] (0,0.9) pic {whitedot} -- (-0.7, 1.05) pic {whitedot}; 
	\end{scope}
	\node at (-0.35, -0.7) {$[2,1] \stackrel{e^2}\longleftarrow [1,1]$};
  \end{tikzpicture}
  \qquad \qquad
  \begin{tikzpicture}
	\begin{scope}[shift={(0,0)},decoration={ markings, mark=at position 0.6 with {\arrow{>}}}]
	  \draw[postaction={decorate}] (0,0) pic {blackdot} -- (-0.7, -0.15) pic {blackdot}; 
	  \draw (-0.7, 0.75) pic {whitedot};
	  \draw[postaction={decorate}] (0,0.3) pic {blackdot} -- (-0.7, 0.15) pic {blackdot}; 
	  \draw[postaction={decorate}] (0,0.6) pic {whitedot} -- (-0.7, 0.45) pic {whitedot}; 
	  \draw[postaction={decorate}] (0,0.9) pic {whitedot} -- (-0.7, 1.05) pic {whitedot}; 
	\end{scope}
	\node at (-0.35, -0.7) {$[1,2] \stackrel{d^1}\longleftarrow [1,1]$};
  \end{tikzpicture}
\end{center}

The abacus maps $f$ are those for which the underlying map is the identity
and whose only effect is to turn the bottom white bead into a black
bead (loosely speaking, swiping it from the white part to the black part,
hence Carlier's terminology ``abacus map''), such as
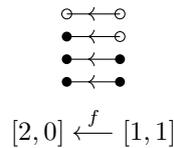
\vspace*{4pt}
\begin{center}
  \begin{tikzpicture}
	\begin{scope}[shift={(0,0)},decoration={ markings, mark=at position 0.6 with {\arrow{>}}}]
	  \draw[postaction={decorate}] (0,0) pic {blackdot} --+ (-0.7, 0) pic {blackdot}; 
	  \draw[postaction={decorate}] (0,0.3) pic {blackdot} --+ (-0.7, 0) pic {blackdot}; 
	  \draw[postaction={decorate}] (0,0.6) pic {whitedot} --+ (-0.7, 0) pic {blackdot}; 
	  \draw[postaction={decorate}] (0,0.9) pic {whitedot} --+ (-0.7, 0) pic {whitedot}; 
	\end{scope}
		\node at (-0.35, -0.6) {$[2,0] \stackrel{f}\longleftarrow [1,1]$};
  \end{tikzpicture}
\end{center}

The abacus maps together with the two commuting sets of (augmented) cosimplicial
operators on the black and white beads form a generating set of
morphisms for $ \DD $. To see this, let $ g $ be a morphism in $ \DD $. If $ g $
maps a white bead to a black bead, then $ g $ factors uniquely through an abacus
map as $ g = g' \circ f $, where if we forget the colors, $ g $ and $ g' $ are equal as
monotone maps. We can repeat this process and extract from this a factorization
$ g = g_{\text{simp}} \circ g_{\text{ab}} $ where $ g_{\text{ab}} $ is a composite of
abacus maps and $ g_{\text{simp}} $ is a composite of (augmented) bisimplicial
operators. This shows that the (augmented) bisimplicial operators, together with
the abacus maps, generate $ \DD $. To put any morphism in $ \DD $ in this
canonical form it suffices to know how this factorization applies to the case
where an (augmented) cosimplicial operator is followed by an abacus map.
These identities, together with the usual identities for bisimplicial operators
of $\simplexcategory_{/[1]}$ therefore give a full description of the category $
\DD $:

\begin{lemma} \label{lem:DDAbacus}
  The category $\DD$ is presented in terms of generators and relations as having the generators and relations of $\simplexcategory_{/[1]}$ and in addition to that, abacus maps $f$, subject to the relations
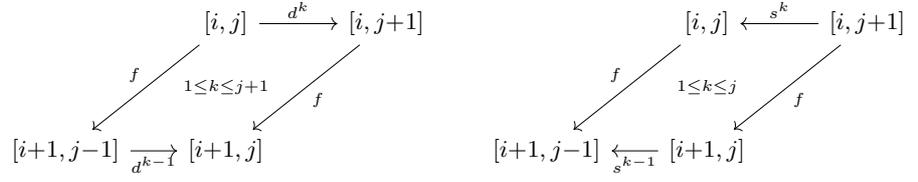
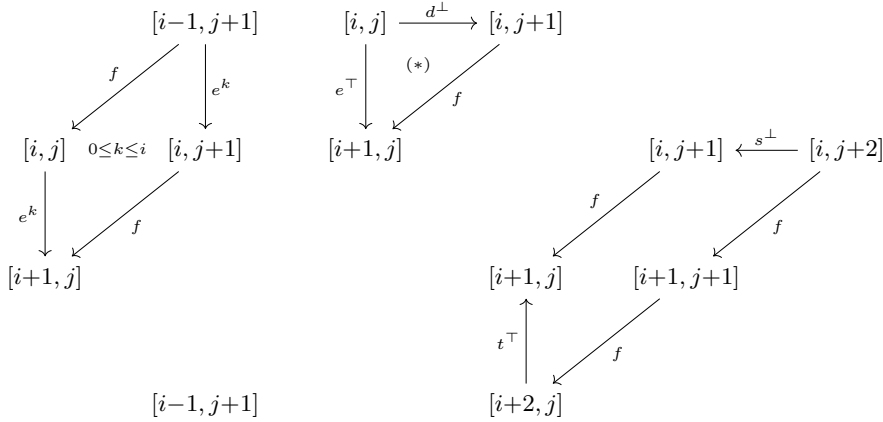
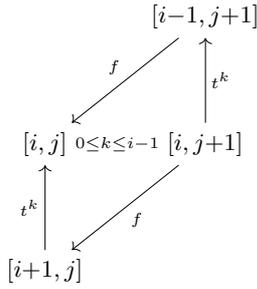
\def\dhh{48pt}
\def\dww{60pt}
\newsavebox{\relall}
\sbox{\relall}{\small
    \begin{tikzcd}[column sep={\dww,between origins}, row sep={\dhh,between origins}, ampersand replacement=\&]
      \&\&\& {[i,j]} \& {[i,j{+}1]} \&\& {[i,j]} \& {[i,j{+}1]} \\
      \&\& {[i{+}1,j{-}1]} \& {[i{+}1,j]} \&\& {[i{+}1,j{-}1]} \& {[i{+}1,j]} \\
      \& {[i{-}1,j{+}1]} \& {[i,j]} \& {[i,j{+}1]} \\
      {[i,j]} \& {[i,j{+}1]} \& {[i{+}1,j]} \&\& {[i,j{+}1]} \& {[i,j{+}2]} \\
      {[i{+}1,j]} \&\&\& {[i{+}1,j]} \& {[i{+}1,j{+}1]} \\
      \& {[i{-}1,j{+}1]} \&\& {[i{+}2,j]} \\
      {[i,j]} \& {[i,j{+}1]} \\
      {[i{+}1,j]}
      \arrow["{d^k}", from=1-4, to=1-5]
      \arrow["f"', from=1-4, to=2-3]
      \arrow["{1\le k \le j+1}"{description}, draw=none, from=1-4, to=2-4]
      \arrow["f", from=1-5, to=2-4]
      \arrow["f"', from=1-7, to=2-6]
      \arrow["{1 \le k \le j}"{description}, draw=none, from=1-7, to=2-7]
      \arrow["{s^k}"', from=1-8, to=1-7]
      \arrow["f", from=1-8, to=2-7]
      \arrow["{d^{k-1}}"', from=2-3, to=2-4]
      \arrow["{s^{k-1}}", from=2-7, to=2-6]
      \arrow["f"', from=3-2, to=4-1]
      \arrow["{e^k}", from=3-2, to=4-2]
      \arrow["{d^\bot}", from=3-3, to=3-4]
      \arrow["{e^\top}"', from=3-3, to=4-3]
      \arrow[""{name=0, anchor=center, inner sep=0}, "f", from=3-4, to=4-3]
      \arrow["{0 \le k \le i}"{description}, draw=none, from=4-1, to=4-2]
      \arrow["{e^k}"', from=4-1, to=5-1]
      \arrow["f", from=4-2, to=5-1]
      \arrow["f"', from=4-5, to=5-4]
      \arrow["{s^\bot}"', from=4-6, to=4-5]
      \arrow["f", from=4-6, to=5-5]
      \arrow["f", from=5-5, to=6-4]
      \arrow["f"', from=6-2, to=7-1]
      \arrow["{t^\top}", from=6-4, to=5-4]
      \arrow["{0 \le k \le i-1}"{description}, draw=none, from=7-1, to=7-2]
      \arrow["{t^k}"', from=7-2, to=6-2]
      \arrow["f", from=7-2, to=8-1]
      \arrow["{t^k}", from=8-1, to=7-1]
      \arrow["{(\ast)}"{description}, draw=none, from=3-3, to=0]
    \end{tikzcd}
}
\[
\begin{tikzpicture}
  \small
  \def\dhh{48pt}
  \def\dww{60pt}
  \node at (-7, 5.85) {$i, j \geq -1$:};
  \node at (5.5,-5.95) {$(\ast)$ not both $i$ and $j$ equal to $-1$};

  \begin{scope}[shift={(0,0)}]
    \node at (0,0) {\usebox{\relall}};
    \end{scope}
\end{tikzpicture}
\]
\end{lemma}
  
The commutation relations of the abacus map and the cosimplicial operators,
displayed above can be read off the columns of beads. As an example, the equation
\begin{equation}\label{edf}
  e^\top = f \circ d^\bot
\end{equation}
is pictorially represented by
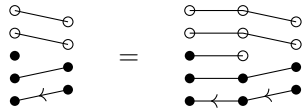
\begin{center}
  \begin{tikzpicture}
	
	\begin{scope}[shift={(0,0)},decoration={ markings, mark=at position 0.55 with {\arrow{>}}}]
	  \draw[postaction={decorate}] (0,0) pic {blackdot} --+ (-0.7, -0.15) pic {blackdot}; 
	  \draw (0,0.3) pic {blackdot} --+ (-0.7, -0.15) pic {blackdot};
	  \draw (0,0.6) pic {whitedot} --+ (-0.7, 0.15) pic {whitedot};
	  \draw (0,0.9) pic {whitedot} --+ (-0.7, 0.15) pic {whitedot};
	  \draw (-0.7, 0.45) pic {blackdot};
	\end{scope}
	
	\node at (0.8,0.4) {$=$};	
	
	\begin{scope}[shift={(3,0)},decoration={ markings, mark=at position 0.6 with {\arrow{>}}}]
	  \draw[postaction={decorate}] (0,0) pic {blackdot} --+ (-0.7, -0.15) pic {blackdot};
	  \draw[postaction={decorate}] (-0.7,-0.15) --+ (-0.7, 0) pic {blackdot}; 
	  \draw (0,0.3) pic {blackdot} --+ (-0.7, -0.15) pic {blackdot} --+ (-1.4, -0.15) pic {blackdot};
	  \draw (0,0.6) pic {whitedot} --+ (-0.7, 0.15) pic {whitedot} --+ (-1.4, 0.15) pic {whitedot};
	  \draw (0,0.9) pic {whitedot} --+ (-0.7, 0.15) pic {whitedot} --+ (-1.4, 0.15) pic {whitedot};
	  \draw (-0.7, 0.45) pic {whitedot}  --+ (-0.7, 0) pic {blackdot};
	\end{scope}

  \end{tikzpicture}
\end{center}

One can visualize $ \DD $ as
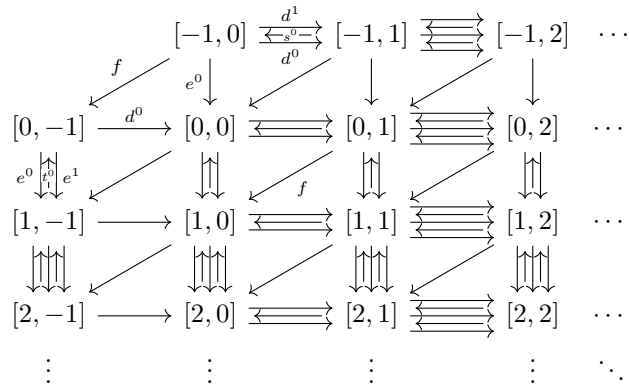
\begin{center}
  \begin{tikzcd}
    & {[-1,0]} & {[-1,1]} & {[-1,2]} & {} \ar[l, phantom, "\dots"]\\
    {[0,-1]} & {[0,0]} & {[0,1]} & {[0,2]} & {} \\
    {[1,-1]} & {[1,0]} & {[1,1]} & {[1,2]} & {} \\
    {[2,-1]} & {[2,0]} & {[2,1]} & {[2,2]} & {} \\
    {} & {} & {} & {} & {}
    \arrow["{\scalebox{0.9}{$d^0$}}"', shift right=1.2, from=1-2, to=1-3]
    \arrow["{\scalebox{0.9}{$d^1$}}", shift left=1.2, from=1-2, to=1-3]
    \arrow["f"', from=1-2, to=2-1]
    \arrow["{\scalebox{0.9}{$e^0$}}"', from=1-2, to=2-2]
    \arrow["{\scalebox{0.7}{$s^0$}}"{description, inner sep=.2}, shorten <=4pt, shorten >=2pt, from=1-3, to=1-2]
    \arrow[shift right=2.4, from=1-3, to=1-4]
    \arrow[shift left=2.4, from=1-3, to=1-4]
    \arrow[from=1-3, to=1-4]
    \arrow[from=1-3, to=2-2]
    \arrow[from=1-3, to=2-3]
    \arrow[shift left=1.2, shorten <=4pt, shorten >=2pt, from=1-4, to=1-3]
    \arrow[shift right=1.2, shorten <=4pt, shorten >=2pt, from=1-4, to=1-3]
    \arrow[from=1-4, to=2-3]
    \arrow[from=1-4, to=2-4]
    \arrow["{\scalebox{0.9}{$d^0$}}", from=2-1, to=2-2]
    \arrow["{\scalebox{0.9}{$e^0$}}"', shift right=1.2, from=2-1, to=3-1]
    \arrow["{\scalebox{0.9}{$e^1$}}", shift left=1.2, from=2-1, to=3-1]
    \arrow[shift right=1.2, from=2-2, to=2-3]
    \arrow[shift left=1.2, from=2-2, to=2-3]
    \arrow[from=2-2, to=3-1]
    \arrow[shift right=1.2, from=2-2, to=3-2]
    \arrow[shift left=1.2, from=2-2, to=3-2]
    \arrow[shorten <=4pt, shorten >=2pt, from=2-3, to=2-2]
    \arrow[shift right=2.4, from=2-3, to=2-4]
    \arrow[shift left=2.4, from=2-3, to=2-4]
    \arrow[from=2-3, to=2-4]
    \arrow["{\scalebox{0.9}{$f$}}", from=2-3, to=3-2]
    \arrow[shift right=1.2, from=2-3, to=3-3]
    \arrow[shift left=1.2, from=2-3, to=3-3]
    \arrow[shift right=1.2, shorten <=4pt, shorten >=2pt, from=2-4, to=2-3]
    \arrow[shift left=1.2, shorten <=4pt, shorten >=2pt, from=2-4, to=2-3]
    \arrow[from=2-4, to=3-3]
    \arrow[shift right=1.2, from=2-4, to=3-4]
    \arrow[shift left=1.2, from=2-4, to=3-4]
    \arrow["\cdots"{description}, phantom, from=2-5, to=2-4]
    \arrow["{\scalebox{0.7}{$t^0$}}"{description, inner sep=.2}, shorten <=4pt, shorten >=2pt, from=3-1, to=2-1]
    \arrow[from=3-1, to=3-2]
    \arrow[shift right=2.4, from=3-1, to=4-1]
    \arrow[from=3-1, to=4-1]
    \arrow[shift left=2.4, from=3-1, to=4-1]
    \arrow[shorten <=4pt, shorten >=2pt, from=3-2, to=2-2]
    \arrow[shift right=1.2, from=3-2, to=3-3]
    \arrow[shift left=1.2, from=3-2, to=3-3]
    \arrow[from=3-2, to=4-1]
    \arrow[shift right=2.4, from=3-2, to=4-2]
    \arrow[shift left=2.4, from=3-2, to=4-2]
    \arrow[from=3-2, to=4-2]
    \arrow[shorten <=4pt, shorten >=2pt, from=3-3, to=2-3]
    \arrow[shorten <=4pt, shorten >=2pt, from=3-3, to=3-2]
    \arrow[shift right=2.4, from=3-3, to=3-4]
    \arrow[shift left=2.4, from=3-3, to=3-4]
    \arrow[from=3-3, to=3-4]
    \arrow[from=3-3, to=4-2]
    \arrow[shift right=2.4, from=3-3, to=4-3]
    \arrow[shift left=2.4, from=3-3, to=4-3]
    \arrow[from=3-3, to=4-3]
    \arrow[shorten <=4pt, shorten >=2pt, from=3-4, to=2-4]
    \arrow[shift left=1.2, shorten <=4pt, shorten >=2pt, from=3-4, to=3-3]
    \arrow[shift right=1.2, shorten <=4pt, shorten >=2pt, from=3-4, to=3-3]
    \arrow[from=3-4, to=4-3]
    \arrow[shift right=2.4, from=3-4, to=4-4]
    \arrow[shift left=2.4, from=3-4, to=4-4]
    \arrow[from=3-4, to=4-4]
    \arrow["\cdots"{description}, phantom, from=3-5, to=3-4]
    \arrow[shift left=1.2, shorten <=4pt, shorten >=2pt, from=4-1, to=3-1]
    \arrow[shift right=1.2, shorten <=4pt, shorten >=2pt, from=4-1, to=3-1]
    \arrow[from=4-1, to=4-2]
    \arrow[shift left=1.2, shorten <=4pt, shorten >=2pt, from=4-2, to=3-2]
    \arrow[shift right=1.2, shorten <=4pt, shorten >=2pt, from=4-2, to=3-2]
    \arrow[shift right=1.2, from=4-2, to=4-3]
    \arrow[shift left=1.2, from=4-2, to=4-3]
    \arrow[shift left=1.2, shorten <=4pt, shorten >=2pt, from=4-3, to=3-3]
    \arrow[shift right=1.2, shorten <=4pt, shorten >=2pt, from=4-3, to=3-3]
    \arrow[shorten <=4pt, shorten >=2pt, from=4-3, to=4-2]
    \arrow[shift right=2.4, from=4-3, to=4-4]
    \arrow[shift left=2.4, from=4-3, to=4-4]
    \arrow[from=4-3, to=4-4]
    \arrow[shift left=1.2, shorten <=4pt, shorten >=2pt, from=4-4, to=3-4]
    \arrow[shift right=1.2, shorten <=4pt, shorten >=2pt, from=4-4, to=3-4]
    \arrow[shift left=1.2, shorten <=4pt, shorten >=2pt, from=4-4, to=4-3]
    \arrow[shift right=1.2, shorten <=4pt, shorten >=2pt, from=4-4, to=4-3]
    \arrow["\cdots"{description}, phantom, from=4-5, to=4-4]
    \arrow["\vdots"{description}, draw=none, from=5-1, to=4-1]
    \arrow["\vdots"{description}, draw=none, from=5-2, to=4-2]
    \arrow["\vdots"{description}, draw=none, from=5-3, to=4-3]
    \arrow["\vdots"{description}, draw=none, from=5-4, to=4-4]
    \arrow["\ddots"{description}, draw=none, from=5-5, to=4-4]
  \end{tikzcd}
\end{center}

\begin{remark}
  The terminology ``abacus map'' is from Carlier's second
  paper~\cite{Carlier:1812.09915}, except that his abacus maps were slightly more
  general, and do not admit any easy combinatorial description. What we call
  abacus map corresponds to what Carlier called {\em perfect abacus map}. One of
  the constructions in his second paper was concerned with modifying a diagram
  with the more general abacus maps to one with perfect abacus maps, that is, a
  presheaf on $\DD$. Since the perfect abacus maps correspond to the combinatorial
  axioms holding in $\DD$, we prefer to use the name for these.
\end{remark}

\begin{observation}
  \label{bla:Trapezium}{\bf (Trapezium equations.)}
  Several other useful equations hold in $\DD$, including the {\em trapezium 
  equations}
  \begin{equation} \label{eq:TrapeziumEq}
    f^{\circ (m{+}n{+}1)} \circ d^m \ = \ e^{i{+}1} \circ f^{\circ (m{+}n)},
	\qquad \text{where} \qquad 
    \begin{array}{rl}
     d^m : [i{-}m,j{+}m] &\to\ \ [i{-}m,j{+}m{+}1] \\
     e^{i+1} : \,\ [i{+}n,j{-}n] &\to\ \ [i{+}n{+}1,j{-}n],
    \end{array}
  \end{equation}
  valid for all $i,j \geq -1$ (but not both equal to $-1$), and all $m\leq i+1$ and 
  $n\leq j+1$.
  They owe their name to the shape they trace out in $ \DD $
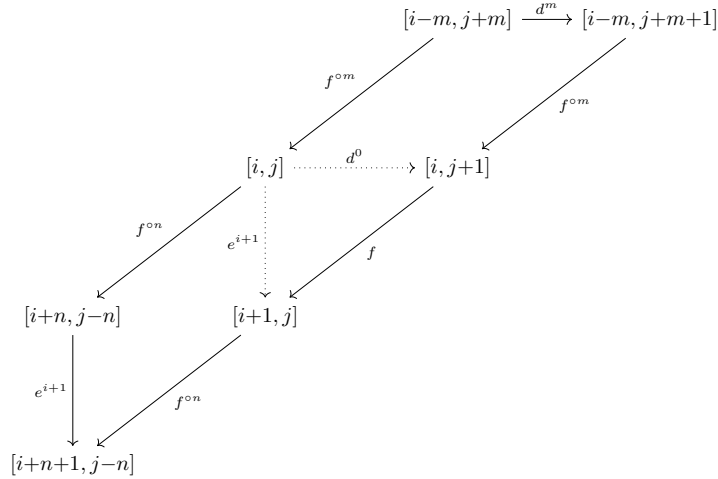
  \begin{center}
    \adjustbox{scale=0.8}{
      \begin{tikzcd}[column sep={9em,between origins}, row sep ={7em,between origins}]
        & & {[i{-}m,j{+}m]} & {[i{-}m,j{+}m{+}1]} \\
        & {[i,j]} & {[i,j{+}1]} & \\
        {[i{+}n,j{-}n]} & {[i{+}1,j]} & & \\
        {[i{+}n{+}1,j{-}n]} & &
        \arrow["{d^m}", from=1-3, to=1-4]
        \arrow["{f^{\circ m}}"', from=1-3, to=2-2]
        \arrow["{f^{\circ m}}", from=1-4, to=2-3]
        \arrow["{d^0}", dotted, from=2-2, to=2-3]
        \arrow["{f^{\circ n}}"', from=2-2, to=3-1]
        \arrow["{e^{i+1}}"', dotted, from=2-2, to=3-2]
        \arrow["f", from=2-3, to=3-2]
        \arrow["{e^{i+1}}"', from=3-1, to=4-1]
        \arrow["{f^{\circ n}}", from=3-2, to=4-1]
      \end{tikzcd}
    }
  \end{center} 
  The diagram is factored so as to contain already the proof: the 
  triangle and the two parallelograms commute thanks to 
  Lemma~\ref{lem:DDAbacus}.
\end{observation}

There is an alternative presentation of $\DD$ in terms of generators and
relations, where instead of abacus maps, we describe extra bottom codegeneracy
maps $s^\subbot$ in each row (except the augmentation row). Pictorially, these are given by
joining the last black and the first white bead to a single black bead, as in 
this example, picturing $[1,0] \stackrel{s^\subbot}\leftarrow [1,1]$:

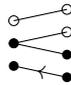
\begin{center}
  \begin{tikzpicture}
	
	\begin{scope}[shift={(0,0)},decoration={ markings, mark=at position 0.55 with {\arrow{>}}}]
	  \draw[postaction={decorate}] (0,0) pic {blackdot} --+ (-0.7, 0.15) pic {blackdot}; 
	  \draw (0,0.3) pic {blackdot} --+ (-0.7, 0.15) pic {blackdot};
	  \draw (0,0.6) pic {whitedot} --+ (-0.7, -0.15) pic {blackdot};
	  \draw (0,0.9) pic {whitedot} --+ (-0.7, -0.15) pic {whitedot};
	\end{scope}

  \end{tikzpicture}
\end{center}
It follows readily that these $s^\subbot$ are actually extra bottom codegeneracy maps in each
row in the sense that they satisfy the identities for split cosimplicial objects.

One should be aware that the extra bottom codegeneracy maps are not compatible with vertical top coface maps. They
{\em are} compatible with all the other vertical cosimplicial operators (including the
vertical top codegeneracy maps). We omit the proof of the following 
lemma which summarizes this.

\begin{lemma}\label{lem:s-generation}
  The category $\DD$ is presented in terms of generators and relations as having
  the generators and relations of $\simplexcategory_{/[1]}$ and in addition to
  that, extra bottom codegeneracy maps in all rows (except the augmentation row).
  The equations satisfied by these extra bottom codegeneracy are the split
  cosimplicial identities 
  \begin{center}
    \begin{tikzcd}
      {[i,j]} & {[i,j{+}1]} && {[i,j]} & {[i,j{+}1]} \\
      {[i,j{+}1]} & {[i,j]} && {[i,j{+}1]} & {[i,j{+}2]}
      \arrow[""{name=0, anchor=center, inner sep=0}, "{{d^k}}"', from=1-1, to=2-1]
      \arrow["\scosplit"', from=1-2, to=1-1]
      \arrow[""{name=1, anchor=center, inner sep=0}, "{{d^{(k+1)}}}", from=1-2, to=2-2]
      \arrow["\scosplit"', from=1-5, to=1-4]
      \arrow["\scosplit", from=2-2, to=2-1]
      \arrow[""{name=2, anchor=center, inner sep=0}, "{s^k}", from=2-4, to=1-4]
      \arrow[""{name=3, anchor=center, inner sep=0}, "{{s^{(k+1)}}}"', from=2-5, to=1-5]
      \arrow["\scosplit", from=2-5, to=2-4]
      \arrow["{0 \le k \le j+1}"{description}, draw=none, from=0, to=1]
      \arrow["{0 \le k \le j}"{description}, draw=none, from=2, to=3]
    \end{tikzcd}
  \end{center}
  as well as those stating that
  they are
  compatible with all vertical cosimplicial operators except the top vertical
  coface maps
  \begin{center}
    \begin{tikzcd}
      && {[i,j]} & {[i,j{+}1]} && {[i,j]} & {[i,j{+}1]} && {\phantom{}} \\
      && {[i{+}1,j]} & {[i{+}1,j{+}1]} && {[i{+}1,j]} & {[i{+}1,j{+}1]}
      \arrow[""{name=0, anchor=center, inner sep=0}, "{e^k}"', from=1-3, to=2-3]
      \arrow["\scosplit"', from=1-4, to=1-3]
      \arrow[""{name=1, anchor=center, inner sep=0}, "{e^k}", from=1-4, to=2-4]
      \arrow["\scosplit"', from=1-7, to=1-6]
      \arrow["\scosplit", from=2-4, to=2-3]
      \arrow[""{name=2, anchor=center, inner sep=0}, "{t^k}", from=2-6, to=1-6]
      \arrow[""{name=3, anchor=center, inner sep=0}, "{t^k}"', from=2-7, to=1-7]
      \arrow["\scosplit", from=2-7, to=2-6]
      \arrow["{0 \le k \le i+1}"{description}, draw=none, from=0, to=1]
      \arrow["{0 \le k \le i}"{description}, draw=none, from=2, to=3]
    \end{tikzcd}
  \end{center}
  where in all diagrams $ i\ge 0, j \ge -1 $.
\end{lemma}

One can pass back and forth between the abacus viewpoint and the
extra-bottom-codegeneracy viewpoint. Given $\scosplit$, the abacus map $f$ is defined as
$$
f := \scosplit \circ e^\top 
$$
\begin{center}
  \begin{tikzpicture}
	
	\begin{scope}[shift={(0,0)},decoration={ markings, mark=at position 0.6 with {\arrow{>}}}]
	  \draw[postaction={decorate}] (0,0) pic {blackdot} -- +(-0.7, 0) pic {blackdot};
	  \draw (0,0.3) pic {blackdot} --+ (-0.7, 0) pic {blackdot};
	  \draw (0,0.6) pic {whitedot} --+ (-0.7, 0) pic {blackdot};
	  \draw (0,0.9) pic {whitedot} --+ (-0.7, 0) pic {whitedot};
	\end{scope}
	  
	\node at (0.8,0.4) {$:=$};

	\begin{scope}[shift={(3,0)},decoration={ markings, mark=at position 0.55 with {\arrow{>}}}]
	  \draw[postaction={decorate}] (0,0) pic {blackdot} --+ (-0.7, -0.15) pic {blackdot}; 
	  \draw[postaction={decorate}] (-0.7, -0.15) -- +(-0.7, 0.15) pic {blackdot};
	  \draw (0,0.3) pic {blackdot} --+ (-0.7, -0.15) pic {blackdot} -- +(-1.4, 0) pic {blackdot};
	  \draw (0,0.6) pic {whitedot} --+ (-0.7, 0.15) pic {whitedot} -- +(-1.4, 0) pic {blackdot};
	  \draw (0,0.9) pic {whitedot} --+ (-0.7, 0.15) pic {whitedot} -- +(-1.4, 0) pic {whitedot};
	  \draw (-0.7,0.45) pic {blackdot} --+ (-0.7, 0.15);
	\end{scope}
	
  \end{tikzpicture}
\end{center}
Conversely, given $f$, the extra bottom codegeneracy map $\scosplit$ is defined as
$$
\scosplit := t^\top \circ f 
$$
\begin{center}
  \begin{tikzpicture}[decoration={ markings, mark=at position 0.6 with {\arrow{>}}}]
	
	\begin{scope}[shift={(0,0)}]
	  \draw[postaction={decorate}] (0,0) pic {blackdot} -- +(-0.7, 0.15) pic {blackdot};
	  \draw (0,0.3) pic {blackdot} -- +(-0.7, 0.15);
	  \draw (0,0.6) pic {whitedot} -- +(-0.7, -0.15) pic {blackdot};
	  \draw (0,0.9) pic {whitedot} -- +(-0.7, -0.15) pic {whitedot};
	\end{scope}

	\node at (0.8,0.4) {$:=$};

	\begin{scope}[shift={(3,0)}]
	  \draw[postaction={decorate}] (0,0) pic {blackdot} -- +(-0.7, 0) pic {blackdot};
	  \draw[postaction={decorate}] (-0.7,0) -- +(-0.7, 0.15) pic {blackdot};
	  \draw (0,0.3) pic {blackdot} --+ (-0.7, 0) pic {blackdot} -- +(-1.4, 0.15) pic {blackdot};
	  \draw (0,0.6) pic {whitedot} --+ (-0.7, 0) pic {blackdot} -- +(-1.4, -0.15);
	  \draw (0,0.9) pic {whitedot} --+ (-0.7, 0) pic {whitedot} -- +(-1.4, -0.15) pic {whitedot};
	\end{scope}
	
  \end{tikzpicture}
\end{center}

\subsection{\texorpdfstring{$\DD$}{D}-presheaves}

  Let $ B $ be a $ \DD $-presheaf. The relations holding in $\DD$
  (Lemma~\ref{lem:DDAbacus})
  translate into equations for the abacus maps $f$ at the presheaf level, which
  can be interpreted either in terms of rows or columns of $B$.
Precisely, the abacus maps form simplicial maps
between rows \begin{equation} \label{eq:AbacusRowsMap}
  f: B_{i+1,\bullet} \to \Decbot{(B_{i,\bullet})}
\end{equation}
and simplicial maps between columns
$$
f : \Dectop{(B_{\bullet,j})} \to B_{\bullet,j+1}.
$$
In these formulae, we allow $i=-1$, so that $f$ goes from the zeroth row to the
decalage of the augmentation row, and we also allow $j=-1$, so that $f$ goes
from the decalage of the augmentation column to the zeroth column.

\begin{remark}\label{rmk:comb}
  Purely combinatorially, from the relations holding in $\DD$ (again
  Lemma~\ref{lem:DDAbacus}), we have the important equation holding inside
  a $\DD$-presheaf:
$$
f = e_\top \circ \ssplit : B_{i+1,j} \to B_{i,j+1}.
$$
This means that each abacus map, viewed as a simplicial map
between rows, $f: B_{i+1,\bullet} \to \decbot{(B_{i,\bullet})}$,
can be described as the composite
\begin{equation}\label{eq:f=es}
B_{i+1,\bullet} \stackrel{\ssplit} \longrightarrow 
\decbot{(B_{i+1,\bullet})}
\stackrel{\decbot(e_\top)}\longrightarrow \decbot{(B_{i,\bullet})}  .
\end{equation}
Conversely, from the equation $\ssplit = f \circ t_\top : B_{i,j} \to B_{i,j+1}$ holding in any 
$\DD$-presheaf (which again
follows from Lemma~\ref{lem:DDAbacus}), we see that we can write the simplicial 
map $\ssplit: B_{i,\bullet} \to \decbot(B_{i,\bullet})$ (between rows) as the composite
\begin{equation}\label{eq:s=ft}
B_{i,\bullet}
\stackrel{t_\top} \longrightarrow
B_{i+1,\bullet}
\stackrel{f} \longrightarrow
\decbot(B_{i,\bullet}).
\end{equation}
 
From the viewpoint of columns we can write instead the simplicial map
$f : \Dectop{(B_{\bullet,j})} \to B_{\bullet,j+1}$ (between columns)
as the composite 
\begin{equation}\label{eq:f=s.epsilon}
\Dectop{(B_{\bullet,j})}
\stackrel{\ssplit} \longrightarrow 
\dectop (B_{\bullet,j+1})
\stackrel{\varepsilon}{\longrightarrow}
B_{\bullet,j+1}  .
\end{equation}

Finally, still from the equations holding in $\DD$ (again
Lemma~\ref{lem:DDAbacus}), we have the following useful equation holding in
any $\DD$-presheaf:
$$
e_\top = d_\bot \circ f : B_{i+1,j} \to B_{i,j}.
$$
In particular, we can describe the vertical augmentation map $e_0 : B_{0,\bullet} 
\to Y$ as the composite
$$
B_{0,\bullet} \stackrel f
\to \decbot{Y}\stackrel{\varepsilon}\to Y.
$$
\end{remark}

Among the $\DD$-presheaves, we are interested  particularly in those 
that are also bimodule configurations:
\begin{definition}
  A $ \DD $-presheaf $ B $ is said to be an {\em abacus bimodule
  configuration} if its underlying $\simplexcategory_{/[1]}$-presheaf
  (augmented bisimplicial space) is a bimodule configuration (as in
  Definition~\ref{bimod-conf}). We write $\Abcs$ for the full
  subcategory of $ \PrSh(\DD) $ spanned by the \underline{a}bacus
  \underline{b}icomodule \underline{c}onfigurations in the sense of
  Carlier.
\end{definition}

\subsection{How stability affects abacus structure}

\begin{definition}
  A $\DD$-presheaf is called {\em upper stable} or {\em lower stable} if the
  underlying bisimplicial space is so (cf.~\ref{subsec:stab}).
  Thus upper stable means that
  all simplicial maps $e_\bot$ between bulk rows are right fibrations,
  and all simplicial maps $d_\bot$ between bulk columns are right fibrations. Similarly lower stable means that
  all simplicial maps $e_\top$ between bulk rows are left fibrations,
  and all simplicial maps $d_\top$ between bulk columns are left fibrations.
  Note that nothing is said about the
  augmentation maps $e_\bot : B_{0,\bullet} \to B_{-1,\bullet}$
  or $d_\top : B_{\bullet,0} \to B_{\bullet,-1}$. 
\end{definition}

\begin{lemma}\label{Segal=>s-1cart}
  In a $\DD$-presheaf $B$, if all bulk rows are $1$-Segal, then for all $i\geq 0$, 
  the simplicial map between rows
  $\ssplit : B_{i,\bullet} 
  \to \decbot(B_{i,\bullet})$ is cartesian.
\end{lemma}
\begin{proof}
  This is the statement that for $1$-Segal spaces, $\decbot$-coalgebras are 
  always rigid (Lemma~\ref{lem:1-Segal-is-rigid}).
\end{proof}

\begin{lemma}
  If a $\DD$-presheaf $B$ is lower stable and all its bulk rows are $1$-Segal, then
  for all $i\geq 0$, the abacus map $f : B_{i+1,\bullet} \to
  \decbot(B_{i,\bullet})$ is cartesian.
\end{lemma}
\begin{proof}
  By Equation~\eqref{eq:f=es} we can write
  $f : B_{i+1,\bullet} \to \decbot(B_{i,\bullet}) $  as the 
  composite
    \begin{equation*}
      B_{i+1,\bullet} \overset{\ssplit}{\longrightarrow} 
	  \decbot(B_{i+1,\bullet}) \overset{\decbot (e_\top)}{\longrightarrow} 
	  \decbot (B_{i,\bullet})  .
    \end{equation*}
  But $\ssplit$ is cartesian as a consequence of Lemma~\ref{Segal=>s-1cart} (since
  the bulk rows are $1$-Segal), and $\decbot (e_\top)$ is cartesian since $e_\top$
  is a left fibration by the lower stability assumption (here we use $i\geq 0$)
  and since the lower decalage of a left fibration is cartesian
  (\ref{Decbot(lfib)}). 
\end{proof}
\begin{remark}
  The conclusion of the lemma is almost Condition~($\star$) from~\ref{condition-star} below, but without saying anything about the augmentation
  row. Note that we cannot use the same argument in the case $i=-1$.
  The argument would involve the
  augmentation map $e_0 : B_{0,\bullet} \to Y$ which cannot be assumed to be a
  left fibration. In fact the augmentation map is not even a left fibration for
  $B=\Tot(Y)$ except when $Y$ is $1$-Segal. (Nevertheless, for $B=\Tot(Y)$, we do
  have that $f$ is cartesian (it is even invertible).)
\end{remark}

The previous results concern the bottom splittings $\ssplit$ as simplicial maps 
between rows, as well as the interpretation of the abacus maps as simplicial 
maps between rows. We now turn to the interpretation of $\ssplit$ and $f$ as
going between columns. This is trickier, since $\ssplit$ is {\em not}
a simplicial map between columns (cf.~\ref{lem:s-generation}): it fails to be compatible with the top face maps 
$e_\top$. But after taking upper decalage, we do get a simplicial map
$\ssplit : \dectop(B_{\bullet},j) \to \dectop(B_{\bullet, j+1})$ for each $j\geq -1$.

\begin{lemma}\label{lem:s-1 rfib}
  In an upper-stable $\DD$-presheaf $B$, the $\ssplit$ constitute cartesian 
  simplicial maps
  between upper-decs of columns
  $$
  \dectop(B_{\bullet,j}) \stackrel{\ssplit}\longrightarrow 
  \dectop(B_{\bullet,j+1}) .
  $$
  This is for $j\geq -1$, so as to include the case of the augmentation column
  $\dectop{X} \stackrel{\ssplit}\longrightarrow \dectop{(B_{\bullet,0})}$.
\end{lemma}
\begin{proof}
The simplicial maps $e_\bot$ are right fibrations between (bulk) rows by virtue of
upper stability. Lemma~\ref{lem:rfib-on-s-1} tells us that each $e_\bot$ is also
cartesian on all bottom splittings (including $X_{i}
\stackrel{\ssplit}\longrightarrow B_{i,0}$). Concretely this means that all the
squares
\[
\begin{tikzcd}
B_{i,j} \ar[r, "\ssplit"] & B_{i,j+1}  \\
B_{i+1,j} \ar[r, "\ssplit"'] \ar[u, "e_k"] & B_{i+1,j+1} \ar[u, "e_k"']
\end{tikzcd}
\]
are pullbacks for all $i \geq 0$, $j \ge -1$, and for $k=0$ (that's $e_\bot$). 
A standard pullback argument shows that
we then get pullback squares also for $e_k$ for all $0\leq k \leq i$, but for
$k=i+1$ the square does not even commute. But after taking upper decalage,
the $\ssplit$ do form simplicial maps $\ssplit : 
\dectop(B_{\bullet,j}) \stackrel{\ssplit}\longrightarrow 
  \dectop(B_{\bullet,j+1})$, whose component on face maps are the pullback 
  squares above. Therefore these simplicial maps are cartesian.
\end{proof}

\begin{prop}\label{prop:fcol rfib}
  Let $B$ be a $\DD_{i\geq 0}$-presheaf. If $B$ is upper stable and if the bulk columns 
  are $1$-Segal, then the abacus maps regarded as simplicial maps between columns
  $$
  f: \dectop(B_{\bullet,j}) \to B_{\bullet, j+1}
  $$
  are right fibrations.
  This is for $j\geq -1$, so as to include the case of 
  $\dectop(X) \stackrel{f}\longrightarrow B_{\bullet,0}$, but it 
  does {\em not}
  involve the abacus maps to the augmentation row.
\end{prop}

\begin{proof}
  By \eqref{eq:f=s.epsilon} in Remark~\ref{rmk:comb}, $f$ factors as
$$
\dectop(B_{\bullet,j}) \stackrel{\ssplit}\longrightarrow  
\dectop(B_{\bullet,j+1}) \stackrel{\varepsilon}\to B_{\bullet, j+1}  .
$$
The first map is cartesian by Lemma~\ref{lem:s-1 rfib}
(since $B$ is upper stable),
  and the second map is a right fibration since we assume
  that the bulk columns $B_{\bullet, j+1}$ are
  $1$-Segal (\ref{cheat:counits}).
\end{proof}

\section{Simplicial maps vs. abacus bicomodule configurations}

\label{sec:carlier-stuff}

Carlier~\cite{Carlier:1801.07504} constructed from any simplicial map a
$\DD$-presheaf. The main results in this section improve upon his
construction by establishing three equivalences: the first equivalence
identifies those $\DD$-presheaves that arise from simplicial maps (the
condition ($\star$) of Theorem~\ref{thm:star-equiv}), while the second
equivalence (Theorem~\ref{thm:ABC*=Up2Seg}) characterizes those
simplicial maps that correspond to the abacus bicomodule configurations
that satisfy condition ($\star$). The third equivalence concerns
identity simplicial maps, and is interpreted as an equivalence between
$2$-Segal spaces and bicomodule configurations with invertible abacus
maps (Theorem~\ref{thm:ABCInv=2Seg}). This is an analogue of the BOORS
equivalence, but with Carlier augmentations instead of BOORS pointing.

\subsection{Right Kan extension along \texorpdfstring{$\qq$}{q}}

\label{subsec:q}

Denote by $[1]$ the category $0\stackrel a \to 1$.
Consider the functor $\qq: \simplexcategory \times [1] \to \DD$
that includes $\simplexcategory \times \{0\}$ as the 
augmentation row of $\DD$ and includes $\simplexcategory\times 
\{1\}$ as the augmentation column of $\DD$, and sends all maps 
$[i] \times a$ to the long composite abacus maps.

The right Kan extension has an explicit formula: given a simplicial map 
$F \in \PrSh(\simplexcategory\times [1])$, the right Kan extension
$ B:= \qq\lowerstar (F) $ is given by
\begin{align*}
  B_{i,j} \simeq \Map_{\PrSh(\simplexcategory\times [1])}(\qq\upperstar  
  (\yo[i,j]), F)  ,
\end{align*}
where $\yo: \DD \to \PrSh(\DD)$ is the Yoneda embedding.
In fact, a direct computation gives $ \qq\upperstar (\yo[i,j]) \simeq (\Delta^i \stackrel{(d^\top)^{\circ (1+j)}}{\longrightarrow}
\Delta^{i+1+j}) $ where by convention we take $ \Delta^{-1} \simeq \emptyset $.
(Note that since we dealing with simplicial spaces, $\Delta^{-1}$ is 
not a representable functor, but it is convenient notation.)
With this the right Kan extension becomes
\begin{equation} \label{eq:qstarFormula} 
  B_{i,j} \simeq 
  \Map_{\PrSh(\simplexcategory\times [1])}(\Delta^i {\to} \Delta^{i+1+j}, F).
\end{equation}
This formula is essentially Carlier's cocartesian nerve~\cite{Carlier:1801.07504}.

An object of $ B_{i,j} $ is thus a commutative square
\begin{center}
  \begin{tikzcd}
    {\Delta^i} & X \\
    {\Delta^{i+1+j}} & Y  .
    \arrow[from=1-1, to=1-2]
    \arrow["{(d^\top)^{\circ(1+j)}}"', from=1-1, to=2-1]
    \arrow["F", from=1-2, to=2-2]
    \arrow[from=2-1, to=2-2]
  \end{tikzcd}
\end{center}
For $i,j\geq 0$ we equivalently have $B_{ij} = X_i \times_{Y_i} Y_{i+1+j} $ by
Yoneda. More precisely, and exhibiting more of the structure, we have the following
proposition.

\begin{prop} \label{prop:qstarDescrip}
  For any simplicial map $F:X \to Y$ considered as an object in 
  $\PrSh(\simplexcategory\times[1])$, put
  $B := \qq\lowerstar(F)$
  Then we have:
  
  \begin{enumerate}
  
  \item \label{item:q*Descrip1}
  The $i$th row of $B$ is $B_{i,\bullet}=X_i \times_{Y_i} \Dec_\bot^{i+1}(Y)$, 
  which is more 
  precisely given by
  \begin{center}  
    \begin{tikzcd}
      {\bar Y_i} & {\Dec_\bot^{i+1}Y} \\
      {\bar X_i} & {\ulpullback B_{i,\bullet}}  .
      \arrow["{{\augmap}}"', from=1-2, to=1-1]
      \arrow["{{{\bar F_i}}}", from=2-1, to=1-1]
      \arrow["{{{f^{\circ(i+1)}}}}"', from=2-2, to=1-2]
      \arrow["{{{\augmap}}}", from=2-2, to=2-1]
    \end{tikzcd}
  \end{center}
  
  (Recall that overline means constant simplicial space and that the zeroth
  component of $ \augmap $ is $ d_\top $.)
  
  \item \label{item:q*Desctip2}
  The $j$th column of $B$ is $B_{\bullet,j} = X \times_Y \dectop^{1+j}(Y)$, 
  which is more precisely
  \begin{center}
    \begin{tikzcd}
      Y & {\dectop^{1+j}Y} \\
      X & {\ulpullback B_{\bullet,j}}  .
      \arrow["{{{d_\top^{\circ(1+j)}}}}"', from=1-2, to=1-1]
      \arrow["F", from=2-1, to=1-1]
      \arrow["{{{f^{\circ(\bullet+1)}}}}"', from=2-2, to=1-2]
      \arrow["{{{d_\top^{\circ(1+j)}}}}", from=2-2, to=2-1]
    \end{tikzcd}
  \end{center}
  This includes the case $j=-1$ which shows that $X$ itself appears as the 
  augmentation column of $\qq\lowerstar(F)$; the case $j=0$ exhibits 
  the horizontal augmentation map 
$X \stackrel{d_0}\longleftarrow B_{\bullet,0}$
as the projection seen in the pullback diagram.

  \item 
  The abacus map (Equation~\eqref{eq:AbacusRowsMap}) $ f : B_{i+1,\bullet} \to
  \decbot (B_{i,\bullet}) $ is
    \begin{center}
      \begin{tikzcd}[column sep=large]
        {\bar X_{i+1} \underset{\bar Y_{i+1}}{\times}     
		\Dec_\bot^{(i+1)+1} (Y)} & {\bar X_i \underset{\bar Y_i}{\times} 
		\Dec_\bot^{(i+1)+1} (Y)  .}
        \arrow["{\bar e_\top \underset{\bar d_\top}{\times} \id}", from=1-1, to=1-2]
      \end{tikzcd}
    \end{center}

  \end{enumerate}
\end{prop}

\begin{proof}
  The first two claims follow directly from Equation~\eqref{eq:qstarFormula}. 
  As for the third claim, consider the commutative diagram
  \begin{center}
    \begin{tikzcd}[column sep={6em,between origins}, row sep ={4em,between origins}]
      & {\bar X_{i+1}} && {B_{i+1,\bullet}} \\
      {\bar X_i} && {\decbot(B_{i,\bullet})} \dlpullback \\
      & {\bar Y_{i+1}} && {\Dec_\bot^{(i+1)+1}(Y)} \\
      {\bar Y_i} && {\decbot^{1+i+1}(Y)}
      \arrow["{{\bar e_\top}}"', from=1-2, to=2-1]
      \arrow["{{\bar F_{i+1}}}"'{pos=0.8}, from=1-2, to=3-2]
      \arrow[from=1-4, to=1-2]
      \arrow["f"', from=1-4, to=2-3]
      \arrow["{f^{\circ((i+1)+1)}}", from=1-4, to=3-4]
      \arrow["{{\bar F_i}}"', from=2-1, to=4-1]
      \arrow[from=2-3, to=2-1]
      \arrow["{f^{\circ(i+1)}}"{pos=0.2}, from=2-3, to=4-3]
      \arrow["{{\bar d_\top}}"', from=3-2, to=4-1]
      \arrow[from=3-4, to=3-2]
      \arrow[Rightarrow, no head, from=3-4, to=4-3]
      \arrow[from=4-3, to=4-1]
    \end{tikzcd}
  \end{center}
  Here the front and back faces are the pullback squares of the first claim, where in the case of the front face we first applied
  $ \decbot $ to the whole square. The result follows from the uniqueness
  of the induced map into the front pullback.
\end{proof}

\begin{remark}
  The simplicial map $B_{\bullet,k} \to \dectop^{1+k} Y$ is given in
  degree $p$ by $f^{\circ (p+1)}$. The fact that these components form a simplicial map
  is an expression of the trapezium equations of~\ref{bla:Trapezium}.
  Specifically, evaluated at the $m$th coface map $[p] \to [p+1]$ we get the
  face maps $e_m : B_{p+1,k} \to B_{p,k}$ and $d_m : Y_{p+k+2} \to Y_{p+k+1}$;
  the simplicial-map equation $d_m \circ f^{\circ (p+2)} = f^{\circ (p+1)}
  \circ e_m$ now appears as an instance of the trapezium equations (by taking
  $m=m$, $n=p-m+1$, $i=m-1$, $j=k+p-m+1$, in the notation of~\ref{bla:Trapezium}).
\end{remark}

Since $ \qq $ is fully faithful, so is $ \qq\lowerstar $. We now characterize the
image by analyzing the unit for the $\qq\upperstar \isleftadjointto \qq\lowerstar$ 
adjunction. For $ B \in \PrSh(\DD) $, let $ \eta : B \to \qq\lowerstar \qq\upperstar(B) $ be the unit. Evaluated on $[i,-1]$
(for $i\ge 0$) the unit reduces to the identity $ \id_{X_i} : X_i \to X_i $.
Similarly, when evaluated on $[-1,j]$ it reduces to the identity $\id_{Y_j} :
Y_j \to Y_j $. For the remaining values $ i,j \ge 0 $, the unit is given by the pullback induced map
\begin{center}
  \begin{tikzcd}
    \eta_{i,j} : B_{i,j} & X_{i} \times_{Y_i} Y_{i+1+j},
    \arrow[from=1-1, to=1-2]
  \end{tikzcd}
\end{center}
which, in more detail, is constructed as in the diagram
\begin{center}
  \begin{tikzcd}
    {B_{i,j}} \\
    & {\drpullback P_{i,j}} & {X_i} \\
    & {Y_{i+1+j}} & {Y_i}  ,
    \arrow["{\eta_{i,j}}", dashed, from=1-1, to=2-2]
    \arrow["{d_\top^{\circ(1+j)}}", bend left=20, from=1-1, to=2-3]
    \arrow["{f^{\circ(i+1)}}"', bend right=25, from=1-1, to=3-2]
    \arrow[from=2-2, to=2-3]
    \arrow[from=2-2, to=3-2]
    \arrow["{F_i}", from=2-3, to=3-3]
    \arrow["{d_\top^{\circ(1+j)}}"', from=3-2, to=3-3]
  \end{tikzcd}
\end{center}
where $P_{i,j}$ is the pullback characterizing 
$(\qq\lowerstar \qq\upperstar(B))_{ij}$.

\begin{definition}
  \label{condition-star}
  We say a $\DD$-presheaf $B$ satisfies {\em Condition ($\star$)} if
  the abacus maps are cartesian, regarded as a simplicial map between rows:
  \begin{center}
    ($\star$)  \hfill
    $f: B_{i+1,\bullet} \to \Decbot{(B_{i,\bullet})}$
    \quad is cartesian for all $i\geq -1$. \hfill \phantom{($\star$)}
  \end{center}
\end{definition}

\begin{lemma} \label{lem:UnitInv<=>star}
  The unit $B \to \qq\lowerstar \qq\upperstar B$ is invertible if and
  only if $B$ satisfies Condition ($\star$) of~\ref{condition-star}.
\end{lemma}

\begin{proof}
  Assume that $f : B_{i+1,\bullet} \to \decbot (B_{i,\bullet})$ is cartesian. The unit is invertible if the square
  \begin{equation} \label{diag:UnitPllbk}
    \begin{tikzcd}[column sep={7em,between origins}, row sep ={5em,between origins}]
      {Y_{i}} & {Y_{i+1+j}} \\
      {X_{i}} & {B_{i,j}}
      \arrow["{d_\top^{\circ(1+j)}}"', from=1-2, to=1-1]
      \arrow["{F_i}", from=2-1, to=1-1]
      \arrow["{f^{\circ(i+1)}}"', from=2-2, to=1-2]
      \arrow["{d_\top^{\circ(1+j)}}", from=2-2, to=2-1]
    \end{tikzcd}
  \end{equation}
  is a pullback for all $ i,j \ge 0 $. But this square can be broken down into
  squares with abacus maps against top face maps, which are pullbacks 
  because
  $f : B_{i+1,\bullet} \to \decbot (B_{i,\bullet})$ 
  is cartesian. Thus the
  square \eqref{diag:UnitPllbk} is a pullback by the pullback prism lemma.

  For the other direction, assume that the unit is invertible, i.e.~the square
  \eqref{diag:UnitPllbk} is a pullback for all $i,j \ge 0$. To show that $f :
  B_{i+1,\bullet} \to \decbot (B_{i,\bullet})$ is cartesian, it suffices to show
  that the abacus maps form pullbacks against all face maps, so we must show that
  the square
  \begin{center}
    \begin{tikzcd}[column sep={7em,between origins}, row sep ={5em,between origins}]
      {B_{i-1,j}} & {B_{i-1,j+1}} \\
      {B_{i,j-1}} & {B_{i,j}}
      \arrow["{d_{k+1}}"', from=1-2, to=1-1]
      \arrow["f", from=2-1, to=1-1]
      \arrow["f"', from=2-2, to=1-2]
      \arrow["{d_{k}}", from=2-2, to=2-1]
    \end{tikzcd}
  \end{center}
  is a pullback for all $i,j \ge 0$ and all $0\le k \le j$. For the case of top
  face maps, consider the diagram
  \begin{center}
    \adjustbox{scale=0.9}{
    \begin{tikzcd}
      & {Y_{i-1}} & {Y_{i}} && {Y_{i+j}} & {Y_{-i+1+j}} \\
      {X_{i-1}} &&& {B_{i-1,j}} & {B_{i-1,j+1}} \\
      {X_{i}} && {B_{i,j-1}} & {B_{i,j}}
      \arrow[dotted, from=1-3, to=1-2]
      \arrow[dotted, from=1-5, to=1-3]
      \arrow[dotted, from=1-6, to=1-5]
      \arrow[dotted, from=2-1, to=1-2]
      \arrow[dotted, from=2-4, to=1-5]
      \arrow[dotted, from=2-4, to=2-1]
      \arrow[dotted, from=2-5, to=1-6]
      \arrow["{d_\top}"', from=2-5, to=2-4]
      \arrow[dotted, from=3-1, to=1-3]
      \arrow["f", from=3-3, to=2-4]
      \arrow[dotted, from=3-3, to=3-1]
      \arrow["f"', from=3-4, to=2-5]
      \arrow["{d_\top}", from=3-4, to=3-3]
    \end{tikzcd}
    }
  \end{center}
  Here all the horizontal dotted arrows are composites of $ d_\top $ maps, and
  the diagonal dotted maps are composites of abacus maps. All the big
  parallelograms that end at the augmentation row and column are pullbacks by
  assumption. Chopping up all the parallelograms appropriately and using the
  pullback prism lemma repeatedly, we see that the bottom right solid
  parallelogram is a pullback. For the squares with the remaining (non-top) face
  maps consider the commutative diagram
  \begin{center}
    \begin{tikzcd}
      & {B_{i-1,0}} && {B_{i-1,j}} & {B_{i-1,j+1}} \\
      {X_{i}} && {B_{i,j-1}} & {B_{i,j}}
      \arrow["{d_\top^{\circ j}}", dotted, from=1-4, to=1-2]
      \arrow["{d_\top^{\circ(j+1)}}"', bend right=15, dotted, from=1-5, to=1-2]
      \arrow["{d_{k+1}}", from=1-5, to=1-4]
      \arrow["f", dotted, from=2-1, to=1-2]
      \arrow["f", from=2-3, to=1-4]
      \arrow["{d_\top^{\circ j}}"', dotted, from=2-3, to=2-1]
      \arrow["f"', from=2-4, to=1-5]
      \arrow["{d_\top^{\circ(j+1)}}", bend left=15, dotted, from=2-4, to=2-1]
      \arrow["{d_k}"', from=2-4, to=2-3]
    \end{tikzcd}
  \end{center}
  The outer and inner left parallelogram are pullbacks by arguments similar to
  that above. By the pullback prism lemma the solid parallelogram is a
  pullback. Thus, in total, $f : B_{i+1,\bullet} \to \decbot (B_{i,\bullet})$ is
  cartesian for all $i \ge -1$.
\end{proof}

\begin{cor} \label{cor:Imq*star}
  The image of $ \qq\lowerstar $ is characterized by Condition ($\star$) (of~\ref{condition-star}), so as to consist of those presheaves $ B \in
  \PrSh(\DD) $ for which
  $$
    f: B_{i+1,\bullet} \to \decbot{(B_{i,\bullet})} 
  $$
  is cartesian (for all $ i \ge -1 $).
\end{cor}

Let $\Pr\upperfivestar(\DD)$ denote the full subcategory $\Pr(\DD)$ 
consisting of the $\DD$-presheaves that satisfy Condition ($\star$).
We summarize the results in the following theorem.
\begin{theorem}\label{thm:star-equiv}
  The $\qq\upperstar \isleftadjointto \qq\lowerstar$ adjunction 
  restricts to an equivalence
  $$
  \PrSh(\simplexcategory \times [1]) \simeq \Pr\upperfivestar(\DD) .
  $$
\end{theorem}

\subsection{Dictionary between conditions on \texorpdfstring{$F$}{F} and on \texorpdfstring{$B:=\qq\lowerstar(F)$}{B=q*(F)}}

In this subsection, we consider a simplicial map $F: X \to Y$ and the
corresponding $\DD$-presheaf $B := \qq\lowerstar(F)$, and explain how various
conditions match up.

\begin{lemma} \label{lem:YSegal-Bstable}
  Given a simplicial map $F : X \to Y$, if $Y$ is lower $2$-Segal, then
  $B:=\qq\lowerstar(F)$ is lower stable. If $Y$ is upper $2$-Segal, then
  $\qq\lowerstar(F)$ is upper stable.
\end{lemma}
\begin{proof}
  For the first claim, it
  is a question of checking each individual square in the condition of 
  being lower stable. Consider the square
  \[
  \begin{tikzcd}
  B_{0,0}  & B_{0,1} \ar[l, "d_1"']  \\
  B_{1,0} \ar[u, "e_1"] & B_{1,1} \ar[u, "e_1"'] \ar[l, "d_1"]
  \end{tikzcd}
  \]
  which in terms of the data of $F$ is given by the outer square of the diagram
  \begin{center}
    \begin{tikzcd}
      {X_0 \times_{Y_0} Y_1} &&& {X_0 \times_{Y_0} Y_2} \\
      & {Y_1} & {Y_2} \\
      & {Y_2} & {Y_3} \\
      {X_1\times_{Y_1} Y_2} &&& {X_1 \times_{Y_1} Y_3 }  .
      \arrow[from=1-1, to=2-2]
      \arrow["{\id\times_{\id} d_2}"', from=1-4, to=1-1]
      \arrow[from=1-4, to=2-3]
      \arrow["{d_2}"', from=2-3, to=2-2]
      \arrow["{d_1}", from=3-2, to=2-2]
      \arrow["{d_1}"', from=3-3, to=2-3]
      \arrow["{d_3}", from=3-3, to=3-2]
      \arrow["{e_1\times_{d_1} d_1}", from=4-1, to=1-1]
      \arrow[from=4-1, to=3-2]
      \arrow["{e_1 \times_{d_1} d_1}"', from=4-4, to=1-4]
      \arrow[from=4-4, to=3-3]
      \arrow["{\id\times_{\id} d_3}", from=4-4, to=4-1]
    \end{tikzcd}
  \end{center}
  Here all diagonals are projections and the top and bottom inner trapezoids are
  pullbacks. Now the central square is a pullback since $Y$ is lower 2-Segal.
  Thus, by the pullback prism lemma, the outer square is again a pullback. The
  same argument works for all other $e_\top$-against-$d_\top$ squares.

  The proof of the second statement is similar.
\end{proof}

\begin{lemma}\label{lem:bulk-rows-1-Segal} \label{lem:YlowSegalS-1Cart}
  If $Y$ is lower $2$-Segal, then in $B:=\qq\lowerstar(F)$ all bulk rows are
  $1$-Segal, and as a result $B$ has cartesian $\ssplit : \dectop(B_{\bullet},j) \to \dectop(B_{\bullet, j+1})$ in all bulk rows.
\end{lemma}

\begin{proof}
  By Lemma~\ref{prop:qstarDescrip}\eqref{item:q*Descrip1} we have a map $
  f^{\circ(i+1)} : B_{i,\bullet} \to \Dec_\bot^{i+1} Y $, which by
  Corollary~\ref{cor:Imq*star} is cartesian. Since $ Y $ is 2-Segal, $ \decbot Y
  $ is $1$-Segal, and since $B_{i,\bullet}$ is cartesian over $\decbot Y$, it is
  $1$-Segal too (by~\ref{fibover1Segal}). The second statement follows from the fact that every 1-Segal
  space with bottom splittings is automatically rigid by
  Lemma~\ref{lem:1-Segal-is-rigid}.
\end{proof}

  One way to interpret the upper-$2$-Segal condition (see~\ref{upper-lower-2Segal}) on a simplicial space $X$ is that one cannot quite
  compose arrows in $X$, so as to get a $2$-simplex from a pair of composable 
  arrows $\cdot \to \cdot \to \cdot$, but if everything has a further arrow down to a
  common point $z \in X_0$, then it {\em is} possible to compose. In other
  words, one can compose in the slice over $z$. Indeed, one formulation of the
  upper-$2$-Segal condition is that every slice is a $1$-Segal space. 

  Now consider $F:X\to Y$. Rather than asking that $X_{/x}$ be $1$-Segal for
  every point $x\in X$, we may ask instead that the slice $X_{/y}$ --- or more
  precisely comma category $F \comma y$ --- be $1$-Segal for every $y\in Y$.
  That is, we demand that the $X$-arrows can be composed in $Y$ if just they are
  over some point $y\in Y$. This discussion motivates the following relative
  notion.

\begin{definition}
  A simplicial map $F : X\to Y$ is called {\em relatively upper $2$-Segal} when 
  the pullback
  $X\times_Y \Dectop Y$ 
  (of $\varepsilon$ along $F$)
  is $1$-Segal.
\end{definition}

\begin{example}
  If $Y$ is upper $2$-Segal and $F: X\to Y$ is a left fibration or a 
  right fibration, then $F$ is relatively upper $2$-Segal. Indeed, $\dectop(Y)$
  is then $1$-Segal, and $X\times_Y \Dectop Y$ will then be a left or right 
  fibration over $\dectop(Y)$ and hence $1$-Segal (by~\ref{fibover1Segal}).
  (The same arguments work in the situation where $F$ is ikeo, or just 
  semi-ikeo; see~\cite{GKT:Crapo} for these notions.)
\end{example}

\begin{lemma}\label{higher-upper}
  If $Y$ is upper $2$-Segal and if $F:X
  \to Y$ is relatively upper $2$-Segal (meaning that $X\times_Y 
  \Dectop Y$ is $1$-Segal), then also $X\times_Y \dectop^{1+j} Y$
  is $1$-Segal (for all $j \geq 0$).
\end{lemma}
\begin{proof}
  The map $\varepsilon_{\Dectop Y} : \Dectop Y \leftarrow \dectop\Dectop Y$ is a right fibration since $\Dectop Y$ is $1$-Segal
  (\ref{cheat:counits}). 
  Now take $X\times_Y \_$ on that fibration to obtain that
  $X\times_Y \dectop^{1+1} Y$ is a fibration over $X\times_Y \dectop^{1} Y$, and
  so on. As a result, all the higher versions are $1$-Segal too.
\end{proof}

\begin{cor} \label{cor:RelUp=>1SegBulkCol}
  If $Y$ is upper 2-Segal and if $F : X \to Y$ is relatively upper 2-Segal, then
  $B=\qq\lowerstar(F)$ has 1-Segal bulk columns.
\end{cor}

\begin{proof}
  By Proposition~\ref{prop:qstarDescrip}\eqref{item:q*Desctip2} we have $B_{\bullet,j}
  \simeq X\times_Y \Dec_\bot^{j+1} Y$ which is 1-Segal for $j\geq 0$ 
  by Lemma~\ref{higher-upper}.
\end{proof}

We can now state the main result of this subsection:

\begin{theorem} \label{thm:ABC*=Up2Seg} 
  Let $B = \qq\lowerstar (F)$ be the $\DD$-presheaf corresponding to a
  simplicial map $F:X\to Y$ as in~\ref{thm:star-equiv}. Then $X$ and $Y$
  are $2$-Segal and $F:X\to Y$ is relatively upper $2$-Segal if and only
  if $B$ is a bicomodule configuration. In particular, the functor $
  \qq\lowerstar $ restricts to an equivalence
  $$
  {\PrSh^{\operatorname{up\, 2-Seg}}(\simplexcategory\times [1])} 
  \stackrel{\simeq}{\longrightarrow} \Abcs\upperfivestar.
  $$
\end{theorem}

\begin{proof}
  Suppose $X$ and $Y$ are $2$-Segal and that $F$ is relatively 
  upper $2$-Segal. Now we start checking the axioms for $B$ 
  being a comodule configuration. 

  {\em The augmentation row is 
  $2$-Segal with culf augmentation}. Indeed, the augmentation 
  row is $Y$ itself, and the augmentation map $B_{0,\bullet}\to 
  Y$ is given by $B_{0,\bullet}\stackrel{f}\to \Decbot Y 
  \stackrel{\varepsilon}\to Y$, where $f$ is cartesian by 
  Condition ($\star$) (which holds by~\ref{cor:Imq*star} since 
  we are in the image of $\qq\lowerstar$)
  and $\varepsilon$ is culf because $Y$ is 
  $2$-Segal (see~\ref{cheat:counitculf}).

  {\em The augmentation column is $2$-Segal and the augmentation map is culf}.
  Indeed, the augmentation column is $X$ itself, and the augmentation map $X
  \leftarrow B_{\bullet,0}$ is the pullback of $Y
  \stackrel{\varepsilon}\leftarrow \Dectop Y$ (see
  Lemma~\ref{prop:qstarDescrip}\eqref{item:q*Desctip2}), which is culf since $Y$
  is $2$-Segal (by~\ref{cheat:counitculf}).

  {\em All bulk rows are $1$-Segal.} This follows from Lemma~\ref{lem:bulk-rows-1-Segal}.
  
  {\em All bulk columns are $1$-Segal.} This follows from
  Corollary~\ref{cor:RelUp=>1SegBulkCol}.

  {\em Stability.} This follows from Lemma~\ref{lem:YSegal-Bstable}.

  For the converse implication, assume $B:=\qq\lowerstar(F)$ is an abacus bicomodule
  configuration. Then $X$ and $Y$ are already $2$-Segal as part of the
  assumptions, and the relative upper-$2$-Segal-ness of $F$ is precisely the
  assumption that the zeroth column is $1$-Segal (see
  Proposition~\ref{prop:qstarDescrip}\eqref{item:q*Desctip2}).
\end{proof}

\subsection{\texorpdfstring{$2$}{2}-Segal cocartesian correspondences} \label{ssec:CoCartCorr}

The relatively upper $2$-Segal condition featured in the previous subsection
may appear a bit mysterious, but it can
be explained in connection with Carlier's viewpoint of cocartesian fibrations over
$\Delta^1$, which is the purpose of this subsection. We will see (in
Proposition~\ref{prop:2up=2}) that $F:X\to Y$ is relatively upper $2$-Segal (with
both $X$ and $Y$ $2$-Segal) if and only if the associated cocartesian fibration $M
\to\Delta^1$ has $M$ $2$-Segal. We will be slightly sketchy in this subsection, so
as not to deviate too much from the main thread.

Given any simplicial map $F:X\to Y$, there is associated
a cocartesian fibration $M \to \Delta^1$, in the sense of simplicial spaces (see 
Carlier~\cite[\S3.2]{Carlier:1801.07504}). This should be a kind of
Grothendieck construction of a functor $[1] \to \sS$, but we are
not aware of the general theory of such a thing for $\sS$-valued functors,
so we just spell out the construction by hand in this very simple case where the
base is just $\Delta^1$. The $0$-simplices of $M$ are given by
$$
M_0 := X_0 + Y_0 ,
$$
with $X_0$ mapping to $0\in \Delta^1$ and $Y_0$ mapping to $1\in \Delta^1$.
The $1$-simplices of $M$ are
$$
M_1 := X_1 + [X_0,Y_1] + Y_1  ,
$$
where $[X_0,Y_1]$ is temporary notation for the space of $1$-simplices in $Y$
that start in a $0$-simplex of $X$. Formally it is the pullback
$$
\begin{tikzcd}
{}[X_0,Y_1] \drpullback \ar[r] \ar[d] & Y_1 \ar[d, "d_1"]  \\
X_0 \ar[r, "F"'] & Y_0  .
\end{tikzcd}
$$
The map to $\Delta^1$ is given as follows: the $1$-simplices in $X_1$ map to
$\id_0 \in \Delta^1$, the $1$-simplices in $Y_1$ map to $\id_1\in \Delta^1$, and
the $1$-simplices in $[X_0,Y_1]$ map to the nontrivial edge $a\in \Delta^1$.

More generally, an $n$-simplex of $M$ is either an $n$-simplex of $X$, an 
$n$-simplex of $Y$, or an $n$-simplex of $Y$ together
with a specification of how an initial-segment $i$-simplex comes from $X$.
Formally this component of $M_n$ is the pullback
$$
\begin{tikzcd}
{}[X_i,Y_n] \drpullback \ar[r] \ar[d] & Y_n \ar[d, "(d_\top)^{\circ (n-i)}"]  \\
X_i \ar[r, "F"'] & Y_i  .
\end{tikzcd}
$$
We recognize that this is precisely
\begin{equation}\label{Mn=Bij}
M_n = \sum_{i+1+j=n} B_{i,j}  ,
\end{equation}
with reference to $B:= \qq\lowerstar(F)$ as in Subsection~\ref{subsec:q}. 

The simplicial map $M\to
\Delta^1$ is given as follows. An $n$-simplex in $B_{i,j}$ (that is, $i+1+j=n$)
maps to $s_0{}^{\circ i} s_1{}^{\circ j}
(a) \in (\Delta^1)_n$, meaning that out of the $i+1+j$ edges the first $i$ edges
contract to $0$ and the last $j$ edges contract to $1$, whereas the remaining
middle edge maps to the nontrivial edge $a$. In the special case where $i=-1$,
this means that all edges map to $1\in \Delta^1$ and in the special case where
$j=-1$, it means that all edges map to $0\in \Delta^1$.

The simplicial structure of $M$ is also readily described in terms of $B$:
the $n+1$ face maps 
$$
M_{n-1} \stackrel{\bar d_k} \longleftarrow M_n , \qquad\qquad 0\leq k\leq n ,
$$
are given on the $B_{i,j}$-component as the $n+1$ face maps going out of 
$B_{i,j}$:
$$
\begin{tikzcd}[sep={60pt,between origins}]
 & B_{i-1,j}  \\
B_{i,j-1} & B_{i,j}  .
\ar[u, shift left=3.5, "0"] 
\ar[u, phantom, "\mbox{\footnotesize $\cdots$}" description] 
\ar[u, shift right=3.5, "i"']
\ar[l, shift left=3.5, "i+1"] 
\ar[l, shift right=1, phantom, "\mbox{\footnotesize $\vdots$}" description] 
\ar[l, shift right=3, "i+1+j"']
\end{tikzcd}
$$
More precisely,
\begin{equation}\label{M-faces}
\bar d_k := \begin{cases} 
e_k & \text{ for } k=0,\ldots,i \\
d_{k-i-1} & \text{ for } k=i+1,\ldots,n  .
\end{cases}
\end{equation}
In particular, in the case $i=-1$, this reduces to
$Y_{n-1} \stackrel{d_k}\leftarrow Y_n$, and in the case 
$j=-1$ this reduces to $X_{n-1} \stackrel{e_k}\leftarrow X_n$.
Similarly, the degeneracy maps $\bar s_k : M_n \to M_{n+1}$ are given by
$$
\bar s_k := \begin{cases} 
t_k & \text{ for } k=0,\ldots,i \\
s_{k-i-1} & \text{ for } k=i+1,\ldots,n  .
\end{cases}
$$

This construction is in fact nothing but the canonical equivalence between 
presheaves on a slice and slice of the presheaf category
$$
\Pr( \simplexcategory_{/[1]}) \isopil \Pr(\simplexcategory)_{/\Delta^1} .
$$
In the direction indicated, this is precisely the assignment $B \mapsto M$
explained above. In the other direction, given $M \to \Delta^1$, 
one can extract each individual $B_{i,j}$ as the pullback
\[
\begin{tikzcd}
B_{i,j} \drpullback \ar[r]\ar[d] & \Map(\Delta^{i+1+j},M) \ar[d] 
\\
1\ar[r, "\name{\id}"'] & \Map(\Delta^1,\Delta^1)  .
\end{tikzcd}
\]
The vertical map is postcomposition with $M\to\Delta^1$ and precomposition with
$(d^\bot)^{\circ i} (d^\top)^{\circ j}$, so it amounts to picking out the edge of the
$(i+1+j)$-simplex that lies over the nontrivial arrow $a\in \Delta^1$.

\begin{lemma}
  Starting with $F:X \to Y$, the associated simplicial map $M\to\Delta^1$ is a
  cocartesian fibration.
\end{lemma}
\begin{proof}
  The cocartesian edges of $M$ (see~\cite[\S3.2]{Carlier:1801.07504}) are the
  elements in $[X_0,Y_1]$ corresponding to invertible maps $Fx\isopil y$. The
  canonical cocartesian lift of $a:0\to 1$ in $\Delta^1$ to an object $x\in X$ is
  simply the identity map $Fx\stackrel=\to Fx$.
\end{proof}

\begin{prop}\label{prop:2up=2}
  A simplicial map $F:X\to Y$ is relatively upper $2$-Segal between $2$-Segal
  spaces if and only if the associated cocartesian fibration
  $M \to \Delta^1$ has $2$-Segal total space.
\end{prop}

\begin{remark}
  It is very likely that this result can be improved to an equivalence of
  $\infty$-categories between a category of $2$-Segal cocartesian fibrations over
  $\Delta^1$ and a category of abacus bicomodule configurations satisfying
  Condition~($\star$), but the arguments employed in the present proof are not quite enough, since
  these categories are not full within the categories of the equivalence $\Pr(
  \simplexcategory_{/[1]}) \isopil \Pr(\simplexcategory)_{/\Delta^1}$.
\end{remark}

\begin{proof}[Proof of Proposition \ref{prop:2up=2}]
  Theorem~\ref{thm:ABC*=Up2Seg} says that the condition ``relatively upper 
  $2$-Segal between $2$-Segal spaces'' means that $B$ is a bicomodule 
  configuration. We claim that this condition in turn matches up with the $2$-Segal condition 
  on $M$. As an illustration of 
  the arguments, let us check the square
  \begin{equation}\label{M-sq}
  \begin{tikzcd}
  M_3 \ar[r, "d_2"] \ar[d, "d_0"'] & M_2 \ar[d, "d_0"]  \\
  M_2 \ar[r, "d_1"'] & M_1  .
  \end{tikzcd}
  \end{equation}
  In terms of $B_{i,j}$, this expands to
  $$
  \begin{tikzcd}
  X_3{+}B_{2,0}{+}B_{1,1}{+}B_{0,2}{+}Y_3 \ar[r] \ar[d] & X_2{+}B_{1,0}{+}B_{0,1}{+}Y_2 \ar[d]  \\
  X_2{+}B_{1,0}{+}B_{0,1}{+}Y_2 \ar[r] & X_1{+}B_{0,0}{+}Y_1 ,
  \end{tikzcd}
  $$
  and by following through the definition of the face maps in $M$ 
  (Equation~\eqref{M-faces}), this square is 
  the sum of three squares:
  $$
  \begin{tikzcd}
  X_3 \ar[r, "e_2"] \ar[d, "e_0"'] & X_2 \ar[d, "e_0"]  \\
  X_2 \ar[r, "e_1"'] & X_1
  \end{tikzcd}
  \quad + \quad
  \begin{tikzcd}
  B_{2,0}{+}B_{1,1} \ar[r, "e_2|d_0"] \ar[d, "e_0{+}e_0"'] & B_{1,0} \ar[d, "e_0"]  \\
  B_{1,0}{+}B_{0,1} \ar[r, "e_1|d_0"'] & B_{0,0}
  \end{tikzcd}
  \quad + \quad
  \begin{tikzcd}
  B_{0,2}{+}Y_3 \ar[r, "d_1+d_2"] \ar[d, "e_0|d_0"'] & B_{0,1}{+}Y_2 \ar[d, 
  "e_0|d_0"]  \\
  Y_2 \ar[r, "d_1"'] & Y_1  .
  \end{tikzcd}
$$
The first square is a pullback since $X$ is $2$-Segal. The second square is a 
combination of two pullbacks: 
one is a Segal square for $B_{\bullet,0}$, the other is an upper stability square.
The third is a combination of two pullbacks:
one is a pullback since the augmentation map to $Y$ is culf, the other because 
$Y$ is $2$-Segal. So if $B$ is a 
bicomodule configuration, then \eqref{M-sq} is a pullback.
Similarly, the
mirror image of \eqref{M-sq} is a pullback because of other bicomodule axioms ---
and all the bicomodule axioms are actually used.
The fact that conversely $2$-Segalness of $M$ implies the bicomodule axioms was
proved already by Carlier~\cite[Proposition~3.1.1]{Carlier:1801.07504}.
\end{proof}

\subsection{Invertible abacus maps and \texorpdfstring{$2$}{2}-Segal spaces}

So far we restricted the adjunction $ \qq\upperstar \isleftadjointto \qq\lowerstar$ to an
equivalence between 2-Segal relatively upper-2-Segal simplicial maps and abacus
bicomodule configurations satisfying Condition ($\star$), as on the top row of the
diagram
\begin{center}
    \begin{tikzcd}
      {\PrSh^{\operatorname{up \, 2-Seg}}(\simplexcategory\times [1])} &  {\Abcs\upperfivestar}
	  &  \\
      \twoSeg &  {\Abcs^\simeq} 
        \arrow["{{\ref{thm:ABC*=Up2Seg}}}"', "\simeq", from=1-1, to=1-2]
        \arrow[hook, from=2-1, to=1-1]
        \arrow["\simeq", from=2-1, to=2-2]
        \arrow[hook, from=2-2, to=1-2]
    \end{tikzcd}
\end{center}
We now show how this restricts further to an equivalence between 2-Segal spaces
and bicomodule configurations with invertible abacus maps, as indicated above in
the second row.  (In reality, the category $\twoSeg$ appears as
the full subcategory of $\PrSh(\simplexcategory\times [1])$ consisting of
invertible simplicial maps between $2$-Segal spaces.)

\begin{lemma} \label{lem:FInv<=>AbcsInv}
  Let $ F : X \to Y $ be a map of simplicial spaces. Then $ \qq\lowerstar (F) $
  has invertible abacus maps if and only if $ F $ is an equivalence.
\end{lemma}

\begin{proof}
  First note that the each component $F_i : X_i \to Y_i$ of $F$ is identified
  with the composite of abacus maps $ f^{\circ(i+1)} : X_i \to Y_i $ in $B := \qq\lowerstar
  (F)$. If all abacus maps are invertible, then it follows immediately that $F$
  is invertible.
    
  For the other direction, assume that $F$ is invertible and consider some abacus map
  $f$. If $f$ ends in the augmentation row, i.e.~it is of the form $ f : B_{0,j}
  \to Y_{j+1} $, then by
  Proposition~\ref{prop:qstarDescrip}\eqref{item:q*Descrip1}, $f$ is a pullback
  of $F_0$ and is therefore invertible. If $f$ is in the bulk, i.e.~it is of the form
  $ f : B_{i+1,j} \to B_{i,j+1} $, then we first postcompose with abacus maps
  until we reach the augmentation row, giving the map $ B_{i+1,j}
  \overset{f}{\to} B_{i,j+1} \overset{f^{\circ(i+1)}}{\to} Y_{i+j+2} $.
  Again by the same
  proposition and the same argument, we find that $ f^{\circ (i+1)} \circ f : B_{i+1,j}
  \to Y_{i+j+2}$ as well as $ f^{\circ(i+1)} : B_{i,j+1} \to Y_{i+j+2} $ are 
  invertible. By the 2-out-of-3 property, it follows that $ f : B_{i+1,j} \to
  B_{i,j+1} $ is invertible. A similar argument works in the last case in which
  $f$ starts in the augmentation column.
\end{proof}

\begin{theorem} \label{thm:ABCInv=2Seg}
   The adjunction $ \qq\upperstar \isleftadjointto \qq\lowerstar $ restricts to an equivalence
   between $2$-Segal
   spaces and bicomodule configurations with invertible abacus maps:
   $$
   \twoSeg \isopil \Abcs^\simeq .
   $$
\end{theorem}

\begin{proof}
	This follows by putting together Theorem~\ref{thm:ABC*=Up2Seg} and
	Lemma~\ref{lem:FInv<=>AbcsInv}, noting that invertible simplicial maps of
	2-Segal spaces are trivially relatively upper 2-Segal.
\end{proof}

\begin{remark} \label{rem:qstar=Tot} 
  From the formula for $\qq\lowerstar$ in
  Equation~\eqref{eq:qstarFormula} we see that the equivalence is
  actually the total decalage functor $\Tot = \rr\upperstar$ induced by
  the functor $\rr : \DD \to \simplexcategory$ that sends $[i,j]$ to
  $[i{+}1{+}j]$.
\end{remark}

\section{\texorpdfstring{$\Sigma$}{Sigma}-presheaves vs.~\texorpdfstring{$\DD$}{D}-presheaves, and the BOORS equivalence}
\label{sec:Boors-comp}

The goal of this section is to relate the disparate notions of augmentation of
BOORS and Carlier, and as a result derive the BOORS equivalence, using the results
already established in the previous sections.

We first show that for a $\Sigma$-presheaf satisfying both the
horizontal and the vertical pointing axiom, the associated abacus maps
become invertible (Corollary~\ref{cor:BOORS=>InvAbcs}). We use this result to derive an equivalence between
$\Sigma$-presheaves satisfying the BOORS axioms and bicomodule
configurations with invertible abacus maps (Theorem~\ref{thm:j-starEquivalence}). Composing this equivalence
with that between bicomodule configurations with invertible abacus and
$2$-Segal spaces (Theorem~\ref{thm:ABCInv=2Seg}) establishes the
original BOORS equivalence.

We end this section with a finer analysis of the relation between
$\Sigma$-presheaves and $\DD$-presheaves, establishing an equivalence (Theorem~\ref{thm:Di=Sigma})
between $\Sigma$-presheaves that satisfy only half of the BOORS axioms
(horizontal pointing, upper stable, Segal rows) and $\DD_{i \geq
0}$-presheaves that are upper stable and have Segal rows (the $i\geq 0$
decoration means that the augmentation row is missing).

Before we embark on this deeper comparison between the two notions of
augmentation, let us note how the total decalage exemplifies both.

\subsection{Augmentations of the total decalage} \label{bla:Tot}
  For $X$ a
  simplicial space, the {\em total decalage} $\Tot(X)$ is the bisimplicial space
  obtained by pulling back $X$ along the ordinal sum functor
  $\simplexcategory\times \simplexcategory \to \simplexcategory $ (see for
  example~\cite{Stevenson:TAC26}). It has the upper decalage $\dectop(X)$ as its
  zeroth column and the lower decalage $\decbot(X)$ as its zeroth row, so a
  picture of $\Tot(X)$ starts like this (suppressing degeneracy maps):
  \[
  \begin{tikzcd}[column sep={50pt,between origins}, row sep={36pt,between origins}]
  X_1 & \ar[l, shift left, "d_1"] \ar[l, shift right, "d_2"'] X_2 \ar[r, "\cdots"{description}, phantom] & {} \\
  X_2 \ar[u, shift left, "d_0"] \ar[u, shift right, "d_1"'] \ar[d, "\vdots"{description}, phantom] & \ar[l, shift left, 
  "d_2"] \ar[l, shift right, "d_3"'] X_3  \ar[u, shift left, "d_0"] \ar[u, shift 
  right, "d_1"'] \ar[r, "\cdots"{description}, phantom] \ar[d, "\vdots"{description}, phantom] & {} \\[-9pt]
  {} & {} &
  \end{tikzcd}
  \]
  (It also contains the edgewise subdivision $\sd(X)$
  as its diagonal. This has recently turned out to be of relevance
  for $2$-Segal spaces: Bergner et al.~\cite{BOORS:Edgewise} show that $X$ is
  $2$-Segal if and only if $\sd(X)$ is $1$-Segal, and
  Hackney--Kock~\cite{Hackney-Kock:2210.11191} show that a simplicial map $F: X\to
  Y$ is culf if and only if $\sd(F)$ is a right fibration.)
  
  The $\Tot$ construction can be refined to provide either a BOORS augmentation 
  (i.e.~pointing)
  or a row-and-column augmentation \`a la Carlier. In the BOORS case this is
  just to take the degeneracy map $s_0: X_0 \to X_1$ as pointing. The
  row-and-column-augmented $\Tot$ instead simply puts the original simplicial space $X$ both in
  the augmentation column and in the augmentation row, where they fit in by
  simplicial operators. 
  
  Let $ \Sigma $ be the category obtained as 
  the
  cocone on $ \simplexcategory \times \simplexcategory $. In terms of generators
  and relations this category is described by the generators and relations of $
  \simplexcategory\times \simplexcategory $ together with an object $[-1]$ and a
  morphism $[0,0] \to [-1]$. The BOORS version of the total decalage
  is $\Tot = \pp\upperstar$, where $\pp: \Sigma\to\simplexcategory$ extends the
  ordinal-sum functor by mapping $[-1]$ to $[0]$. 
  The BOORS equivalence is established
  in~\cite{Bergner-Osorno-Ozornova-Rovelli-Scheimbauer:1809.10924} as a
  restriction of the adjunction $\pp\upperstar\isleftadjointto\pp\lowerstar$,
  with $\pp\lowerstar$ being interpreted as a generalized Waldhausen
  construction. Pulling back along the similarly defined functor
  $\simplexcategory_{/[1]}\to\simplexcategory$ gives the
  row-and-column-augmented $\Tot$, which plays a role in Carlier's theory: it is
  the result of applying his constructions to the identity correspondence or the
  identity functor of an $ \infty$-category. The functor
  $\simplexcategory_{/[1]} \to \simplexcategory$ factors 
  as the inclusion $\simplexcategory_{/[1]} \to \DD$ followed by the functor $\rr : \DD \to \simplexcategory$ (cf.~Remark~\ref{rem:qstar=Tot}).
  Altogether the shapes
  mentioned so far fit into the diagram (of ordinary categories)
  
    \begin{center}
    \begin{tikzcd}[column sep={24pt,between origins}, row sep={24pt,between origins}]
    && \simplexcategory_{/[1]} \ar[rr] && \DD \ar[rrd, "\rr"] &&
	\\
    {\simplexcategory\times \simplexcategory}  \ar[rrrrrr] \ar[rru, hook] \ar[rrrd, hook] 
	  &&&&&& \simplexcategory 
	\\
    &&& \Sigma \ar[rrru, "\pp"'] \ar[ruu, dotted, pos=0.63, "\jj"] &&&
    \end{tikzcd}
  \end{center}
  The functor $\jj$ is the topic of Subsection~\ref{subsec:Abcs=BOORS}.

\subsection{From the pointing axioms to invertible abacus maps} \label{subsec:Pointing->InvAbcs}

We first recall the BOORS equivalence.

\begin{definition} \label{def:PointingAxiom}
  A presheaf $B \in \PrSh(\Sigma)$ is said to satisfy the {\em horizontal 
  pointing axiom} if 
  the pointing $B_{-1}
  \to B_{0,0}$ constitutes a
  local-initial-objects structure on the zeroth row.
  Dually, 
  $B \in \PrSh(\Sigma)$ is said to satisfy the {\em vertical 
  pointing axiom} if 
  the pointing $B_{-1}
  \to B_{0,0}$ constitutes a local-terminal-objects structure on the zeroth 
  column.
\end{definition}

\begin{remark}
  Bergner et
  al.~\cite{Bergner-Osorno-Ozornova-Rovelli-Scheimbauer:1609.02853,Bergner-Osorno-Ozornova-Rovelli-Scheimbauer:1809.10924}
  use the terminology {\em preaugmented bisimplicial spaces} for general presheaves on
  $\Sigma$, and say {\em augmented bisimplicial spaces} for presheaves that
  satisfy both the horizontal and vertical pointing axioms of
  Definition~\ref{def:PointingAxiom}.
\end{remark}

Let $ B $ be an upper-stable $\Sigma$-presheaf satisfying the horizontal pointing
axiom. By Proposition~\ref{prop:RigCoalgIsRigPoint}, which identifies 
local-initial-objects structure with rigid $\decbot$-coalgebra structure,
the zeroth row of $B$ is
endowed with a bottom-split structure. This is the first step towards relating
$\Sigma$-presheaves and $\DD$-presheaves. Upper stability, which says that each
simplicial map
$e_\bot : B_{i+1,\bullet} \to B_{i,\bullet}$ is a right fibration between rows, induces
the remaining extra bottom sections in the bulk: applying
Lemma~\ref{lem:RFibLiftDecCoal} produces first a $\decbot$-coalgebra structure
on $B_{1,\bullet} $ which is compatible with $ e_\bot $, and with the help of
an inductive argument we can propagate down the bottom-split structure to all
rows.

A priori, the aforementioned construction of the bottom sections $\ssplit$ only
guarantees that they commute with the bottom face maps $ e_\bot $. It remains to
show that the bottom sections are compatible also with the active part of the
columns, so that altogether the bottom sections $\ssplit$ form simplicial maps between
columns after applying $ \dectop $ to the columns. To see that the bottom
sections are compatible with the active part of each column we provide a construction that produces all the bottom sections uniformly. In
this construction we invoke Lemma~\ref{lem:RFibLiftDecCoal} again, but this time
take $ \mathcal E $ to be $ \sS $ (simplicial spaces). The simplicial object $ C $ 
of the lemma is taken to be
the composite
\begin{center}
  \begin{tikzcd}[column sep=large]
    \simplexcategory^\text{op} \arrow[r, "{B_{0,\bullet}}"] & \spaces \arrow[r, "{\text{const}}"] & \sS
  \end{tikzcd}
\end{center}
which is a $\decbot$-coalgebra in $\mathcal E$.
For $ C' $ we take the transpose of
\begin{center}
  \begin{tikzcd}[column sep=large]
    \simplexcategory\op\times\simplexcategory\op \arrow[r,"\dectop\times \id"] 
	& \simplexcategory\op\times \simplexcategory\op \arrow[r,"B_{\bullet,\bullet}"] &\spaces
  \end{tikzcd}
\end{center}
in the first coordinate, which gives $ C' :
\simplexcategory^\text{op} \to \sS; \ j \mapsto \dectop^{\text{vert}}
(B_{\bullet,j}) $.  The simplicial map $C' \to C$ is given by the vertical 
augmentation map for $\dectop$ (which in degree $0$ is the map $e_\bot$).
So far,  the
argument leaves out the compatibility of the horizontal bottom sections $ \ssplit
$ with the vertical top degeneracies $ t_\top $ (as these are discarded by the 
vertical $\dectop$); however, this compatibility will be automatic thanks to the top
degeneracy $ t_\top $ being a section of a face map $ e_{\top-1} $, which is
contained in $\dectop (B_{\bullet,0})$.

The above discussion can be distilled into a proof of the following lemma.

\begin{lemma} \label{lem:PropagateS-1}
  Let $B$ be a $\Sigma$-presheaf which is upper stable and satisfies the
  horizontal pointing axiom. Then $B$ has induced extra bottom sections
  in every row.
\end{lemma}

\begin{remark} \label{rem:PropagateS-1}
  As a matter of fact, the argument above Lemma~\ref{lem:PropagateS-1} applies
  to any bisimplicial space $B$ with extra bottom sections in the first row. In
  other words, any upper-stable bisimplicial space with a $\decbot$-coalgebra
  structure on its zeroth row has a canonical extension to a
  $\DD_{\geq0}$-presheaf, where $\DD_{\geq0}$ is the category $\DD$ with the
  augmentations removed.
\end{remark}

The equivalence of Bergner et
al.~\cite{Bergner-Osorno-Ozornova-Rovelli-Scheimbauer:1609.02853,Bergner-Osorno-Ozornova-Rovelli-Scheimbauer:1809.10924}
says that if a $ \Sigma$-presheaf is stable, double Segal and satisfies both
horizontal and vertical pointing axioms, then it is the total decalage of a
2-Segal space. In the total decalage of a simplicial space the abacus maps are
the identities, and  in particular invertible. We show directly how for
a $\Sigma$-presheaf with all the above properties the induced abacus maps are
invertible.

Let $ B $ be a $ \Sigma$-presheaf which is stable, double Segal and satisfies
both horizontal and vertical pointing axioms. By Lemma~\ref{lem:PropagateS-1}
and its dual, $B$ is endowed both with extra bottom sections $\ssplit$ in all bulk
rows and with extra top sections $\tsplit$ in all bulk columns. Taking colimits
row-wise gives an augmentation column, where the augmentation maps inherit a
bottom section from the bottom-split structure of the rows. Similarly, taking
colimits column-wise gives an augmentation row, where the augmentation map is
equipped with extra top sections. 
The following lemma expresses a key compatibility between the
horizontal bottom splittings and the vertical top splittings.

\begin{lemma} \label{lem:ts=ts}
  Let $ B $ be a $ \Sigma$-presheaf which is stable, double Segal and satisfies
  both horizontal and vertical pointing axioms. Then for the induced splittings 
  we have
  \begin{equation} \label{eq:ts=ts}
    t_\top \ssplit = \tsplit \, \ssplit   \,.
  \end{equation}
  This holds also in the augmentation column which is obtained by taking colimits.
\end{lemma}

\begin{proof}
  The proof relies on an inductive argument. First of all, consider the diagram
  \begin{center}
    \begin{tikzcd}
      & {B_{-1}} \\
      {B_{-1}} & {B_{0,0}} \\
      & {B_{1,0}}
      \arrow["{\tsplit}", from=1-2, to=2-2]
      \arrow[Rightarrow, no head, from=2-1, to=1-2]
      \arrow["{\ssplit}"', from=2-1, to=2-2]
      \arrow["{\tsplit}", shift left, from=2-2, to=3-2]
      \arrow["{t_\top}"', shift right, from=2-2, to=3-2]
    \end{tikzcd}
  \end{center}
  The top triangle commutes by definition, where both $ \ssplit $ and $ \tsplit
  $ are given by the pointing in $ \Sigma $. Since $ \tsplit : B_{-1} \to
  B_{0,0} $ equalizes the two maps $ t_\top, \tsplit : B_{0,0} \to B_{1,0} $,
  (this is a simplicial identity), it follows that also $ \ssplit$ equalizes
  $t_\top$ and $\tsplit $. This gives us the first instance of 
  Equation~\eqref{eq:ts=ts} and forms the base case for the inductive proof.

  To allow for a uniform argument, let us denote the objects of the augmentation
  column by $ B_{i,-1} $. In particular, $ B_{0,-1} = B_{-1} $. For the
  inductive step, assume that
  \begin{center}
    \begin{tikzcd}
      B_{i,j} \arrow[r,"\ssplit"] & B_{i,j+1} \arrow[r, "t_\top"', shift right] \arrow[r, "\tsplit", shift left] & B_{i+1,j+1}
    \end{tikzcd} 
  \end{center}
  commutes for some $ i \ge 0 $ and some $ j \ge -1 $. We show that this
  continues to commute in the next row, i.e.~with $i$ replaced by $i+1$. For
  this consider the diagram
  \begin{center}
    \begin{tikzcd}
      {B_{i,j}} & {B_{i,j+1}} & {B_{i+1,j+1}} & {B_{i,j+1}} \\
      {B_{i+1,j}} & {B_{i+1,j+1}} & {\urpullback B_{i+2,j+1}} & {B_{i+1,j+1}}  .
      \arrow["{\ssplit}", from=1-1, to=1-2]
      \arrow["{\tsplit}", shift left, from=1-2, to=1-3]
      \arrow["{t_\top}"', shift right, from=1-2, to=1-3]
      \arrow["{e_\top}", from=1-3, to=1-4]
      \arrow["{e_\bot}", from=2-1, to=1-1]
      \arrow["{\ssplit}"', from=2-1, to=2-2]
      \arrow["{e_\bot}", from=2-2, to=1-2]
      \arrow["{\tsplit}", shift left, from=2-2, to=2-3]
      \arrow["{t_\top}"', shift right, from=2-2, to=2-3]
      \arrow["{e_\bot}", from=2-3, to=1-3]
      \arrow["{e_\top}"', from=2-3, to=2-4]
      \arrow["{e_\bot}"', from=2-4, to=1-4]
    \end{tikzcd}
  \end{center}
  The two outer squares commute, and so does the middle square if we consider either
  both $ t_\top $ or both $ \tsplit $ degeneracies. The right-most square is a
  pullback by the Segal condition in column $j+1$. Since both $ t_\top $ and $
  \tsplit $ are sections of $ e_\top $ and we are assuming that the equation
  holds on the first row, by the uniqueness of the pullback-induced maps it
  follows that Equation~\eqref{eq:ts=ts} holds also in the next row.

  Finally we show that it also holds in the next column, i.e.~with $j$ 
  replaced by $ j+1 $. For this we use the same argument as for the induction
  along rows, now applied to the the diagram
  \begin{center}
    \begin{tikzcd}
      {B_{i,j}} & {B_{i,j+1}} & {B_{i+1,j+1}} & {B_{i,j+1}} \\
      {B_{i,j+1}} & {B_{i,j+2}} & {\urpullback B_{i+1,j+2}} & {B_{i,j+2}}
      \arrow["{\ssplit}", from=1-1, to=1-2]
      \arrow["{\tsplit}", shift left, from=1-2, to=1-3]
      \arrow["{t_\top}"', shift right, from=1-2, to=1-3]
      \arrow["{e_\top}", from=1-3, to=1-4]
      \arrow["{d_\top}", from=2-1, to=1-1]
      \arrow["{\ssplit}"', from=2-1, to=2-2]
      \arrow["{d_\top}", from=2-2, to=1-2]
      \arrow["{\tsplit}", shift left, from=2-2, to=2-3]
      \arrow["{t_\top}"', shift right, from=2-2, to=2-3]
      \arrow["{d_\top}", from=2-3, to=1-3]
      \arrow["{e_\top}"', from=2-3, to=2-4]
      \arrow["{d_\top}"', from=2-4, to=1-4]
    \end{tikzcd}
  \end{center}
\end{proof}

\begin{cor} \label{cor:BOORS=>InvAbcs}
  Let $ B $ be a $\Sigma$-presheaf which is stable, double Segal and satisfies
  the horizontal and vertical pointing axioms. Then the induced abacus maps are
  invertible. In particular, the pair of maps
  \begin{center}
    \begin{tikzcd}
      & {B_{i,j+1}} \\
      {B_{i+1,j}}
      \arrow["{g := d_\bot \tsplit}", shift left, from=1-2, to=2-1]
      \arrow["{f:=e_\top \ssplit}", shift left, from=2-1, to=1-2]
    \end{tikzcd}
  \end{center}
  constitute a pair of inverse maps for all $ i,j \ge -1 $, where the $ B_{\bullet,-1} $ is obtained by taking colimits row-wise, and
  $ B_{-1,\bullet} $ is obtained by taking colimits column-wise.
\end{cor}
\begin{remark}
  The maps $g:=d_\bot \tsplit$ are the ``dual abacus maps'' obtained from the 
  vertical extra top degeneracy maps induced by the local-terminal-objects 
  structure corresponding to the vertical pointing axiom.
  \end{remark}
\begin{proof}[Proof of Corollary \ref{cor:BOORS=>InvAbcs}]
  We compute
  \begin{align*}
    gf &= d_\bot \tsplit \, e_\top \ssplit \\
       &= d_\bot e_{\top-1} \, \tsplit \, \ssplit \\
       &= d_\bot e_{\top-1} t_{\top} \, \ssplit, \qquad \text{by Lemma~\ref{lem:ts=ts}} \\
       &= d_\bot \ssplit \\
       &= 1.
  \end{align*}
  where every line except the first and third uses a simplicial identity. A dual
  argument shows that $ fg = 1 $.
\end{proof}

\begin{remark}
  The fact that in a BOORS-augmented stable double Segal space $B$ there is an
  equivalence $B_{1,0}\simeq B_{0,1}$ was observed in~\cite[Remark
  2.31]{Bergner-Osorno-Ozornova-Rovelli-Scheimbauer:1809.10924}: they establish a
  zig-zag of weak equivalences between the two spaces, using the 
  pointing axioms
  and the stability axiom. From the viewpoint of that zig-zag, it is perhaps
  surprising that half of the axioms are enough to get the abacus map directly, not
  as a zigzag (although of course both sides of the axioms are needed to establish
  that it is invertible, as we have just seen). On the other hand, a posteriori,
  the BOORS equivalence tells us of course that a canonical equivalence must
  exist, since in the case (which is every case) of $B = \Tot(X)$ both $B_{1,0}$
  and $B_{0,1}$ are identified with $X_2$.
\end{remark}

\subsection{The equivalence \texorpdfstring{$\BOORS \simeq \Abcs^\simeq$}{ABCinv=BOORS}} \label{subsec:Abcs=BOORS}

Consider the inclusion $ \jj : \Sigma \to \DD $, which is defined as the identity on the underlying
bisimplicial space and which maps the pointing of $ \Sigma $ onto the map $ \scosplit
: [0,0] \to [0,-1] $ in $ \DD $. Recall that $ \Abcs^\simeq
\subset \PrSh(\DD) $ is the full subcategory spanned by those presheaves that are
stable, double Segal and have invertible abacus maps. For $ B \in \Abcs^\simeq $,
the restriction $ \jj\upperstar (B) $ will be
stable and double Segal, as these conditions do not refer to the augmentations.
The zeroth row in $ B $ is equipped with a $ \decbot$-coalgebra structure. Since
$B$ is double Segal, this coalgebra structure will automatically be rigid by
Lemma~\ref{lem:1-Segal-is-rigid}. It follows from Proposition~\ref{prop:RigCoalgIsRigPoint}
that the zeroth row together with the map $ \ssplit : X_0 \to B_{0,0} $
satisfies the horizontal pointing axiom. On the other hand, the 
invertibility of the abacus maps provides extra top degeneracy maps in every 
column, and in particular a $\dectop$-coalgebra structure
$$
  B_{\bullet,j} \overset{s_0}{\longrightarrow} B_{\bullet,j+1} 
  \overset{f^{-1}}{\longrightarrow} \Dectop(B_{\bullet,j}).
$$
This is the description of $\tsplit$, dual to Equation~\eqref{eq:s=ft} in
Remark~\ref{rmk:comb}. In particular we get a $\dectop$-coalgebra 
structure on the zeroth
column, which is rigid (in the sense of top splittings) by the double Segal
condition again and by the dual of~\ref{lem:1-Segal-is-rigid}. From
Proposition~\ref{prop:RigCoalgIsRigPoint}, it follows that the zeroth column
satisfies the vertical pointing axiom. Altogether we have shown that the functor
$\jj\upperstar$ restricts to
$$
  \jj\upperstar : \Abcs^\simeq \to \BOORS.
$$

\begin{theorem} \label{thm:j-starEquivalence}
  The functor $ \jj\upperstar : \Abcs^\simeq \to \BOORS $ is an equivalence.
\end{theorem}

\begin{proof}
  {\em Essential surjectivity}: Starting with a $\Sigma$-presheaf $A' \in \BOORS$ we shall
  extend this to a $\DD$-presheaf in $ \Abcs^\simeq $ which restricts to $ A'$
  under $\jj\upperstar $. We obtain all the bottom splittings in the bulk by
  the horizontal pointing axiom and upper stability, via
  Lemma~\ref{lem:PropagateS-1}. Taking colimits row-wise allows us to build the
  complete augmentation column along with the augmentation maps. Since the rows
  are split by bottom sections, so will the augmentation maps be. Taking colimits
  column-wise constructs the augmentation row together with the augmentation maps.
  So far we have built a $\DD$-diagram $A$ which restricts to $A'$ when pulled
  back along $\jj$. Now, $A$ inherits the properties of being stable and
  double Segal from $A'$. The augmentation row and column are 2-Segal and the
  augmentation maps are culf by Proposition~\ref{prop:colim-aug}. By
  Theorem~\ref{cor:BOORS=>InvAbcs} all abacus maps are invertible.

  {\em Faithfulness:} Consider a functor $ G \in \Map_{\Abcs^\simeq}(A,B) $. Every
  object in $ \DD $ which is not in the image of $ \jj : \Sigma \hookrightarrow
  \DD $ is connected to an object in the image of $\jj$ by moving along an abacus
  map. We visualize this on the level of presheaves
  \begin{center}
    \begin{tikzcd} 
      & B_{-1,0} \ar[r, dotted, shorten <=5pt, shorten >=2pt] 
      & B_{-1,1} \ar[l, shift left=1.5, dotted]
                 \ar[l, shift right=1.5, dotted]
                 \ar[r, shift left=1.5, dotted, shorten <=5pt, shorten >=2pt] 
                 \ar[r, shift right=1.5, dotted, shorten <=5pt, shorten >=2pt]
      & B_{-1,2} \ar[l, dotted]
                 \ar[l, shift left=3, dotted] 
                 \ar[l, shift right=3, dotted] 
                 \ar[r, phantom, "\dots"]
                 & \phantom{} \\
    B_{0,-1} \ar[d, dotted, shorten <=5pt, shorten >=2pt]
             \ar[ur, dotted]
             \ar[r, bend right=15, shorten <=4pt, shorten >=3pt]
      & B_{0,0} \ar[d, shorten <=5pt, shorten >=2pt]
                \ar[l, dotted]
                \ar[u, dotted]
                \ar[r, shorten <=5pt, shorten >=2pt] 
                \ar[ur, dotted]
      & B_{0,1} \ar[d, shorten <=5pt, shorten >=2pt]
                \ar[u, dotted]
                \ar[l, shift left=1.5]
                \ar[l, shift right=1.5]
                \ar[r, shift left=1.5, shorten <=5pt, shorten >=2pt] 
                \ar[r, shift right=1.5, shorten <=5pt, shorten >=2pt]
                \ar[ur, dotted, shorten <=6pt, shorten >=2pt] 
      & B_{0,2} \ar[u, dotted]
                \ar[d, shorten <=5pt, shorten >=2pt]
                \ar[l]
                \ar[l, shift left=3] 
                \ar[l, shift right=3]
                \ar[r, phantom, "\dots"]
                 & \phantom{}\\
    B_{1,-1} \ar[u, shift left=1.5, dotted]
             \ar[u, shift right=1.5, dotted]
                \ar[ur, dotted]
                \ar[d, phantom, "\vdots"]
      & B_{1,0} \ar[u, shift left=1.5]
                \ar[u, shift right=1.5]
                \ar[l, dotted] 
                \ar[r, shorten <=5pt, shorten >=2pt] 
                \ar[ur, dotted]
                \ar[d, phantom, "\vdots"]
      & B_{1,1} \ar[u, shift left=1.5]
                \ar[u, shift right=1.5]
                \ar[l, shift left=1.5]
                \ar[l, shift right=1.5]
                \ar[r, shift left=1.5, shorten <=5pt, shorten >=2pt] 
                \ar[r, shift right=1.5, shorten <=5pt, shorten >=2pt] 
                \ar[ur, dotted, shorten <=6pt, shorten >=3pt]
                \ar[d, phantom, "\vdots"]
      & B_{1,2}   
                \ar[u, shift left=1.5]
                \ar[u, shift right=1.5]
                \ar[l]
                \ar[l, shift left=3] 
                \ar[l, shift right=3]
                \ar[r, phantom, "\dots"]
                \ar[d, phantom, "\vdots"]
                 & \phantom{}\\
     \phantom{}
       & \phantom{}\
       &\phantom{}
       & \phantom{}&
    \end{tikzcd}
  \end{center}
  where the restriction to $ \PrSh(\Sigma)$ under $\jj\upperstar$ is depicted using solid arrows. Since the
  abacus maps are assumed to be invertible, the value of any natural
  transformation on the image of $\jj\upperstar$ fixes its values outside of it,
  as can be seen for example in the square
  \begin{center}
    \begin{tikzcd}
      {A_{-1,0}} & {B_{-1,0}} \\
      {A_{0,-1}} & {B_{0,-1}}
      \arrow["G", from=1-1, to=1-2]
      \arrow["f", from=2-1, to=1-1]
      \arrow["G"', from=2-1, to=2-2]
      \arrow["f"', from=2-2, to=1-2]
    \end{tikzcd}
  \end{center}
  As a result, if two functors $ G,H \in \Map_{\Abcs^\simeq}(A,B) $ agree on $
  \Sigma$, that is $ \jj\upperstar(G) \simeq \jj\upperstar(H) $, then they also agree
  on $ \DD $, that is $ G \simeq H $, proving faithfulness.
  
  {\em Fullness:} Let $ A,B $ be two $\DD$-presheaves in $ \Abcs^\simeq $ and let $ G'
  \in \Map_{\BOORS}(\jj\upperstar(A),\jj\upperstar(B)) $ be a functor. Our goal is to
  construct a functor $ G : A \to B $ which when pulled back along $ \jj $
  recovers $ G' $. Let us begin with the diagram
  \begin{center}
    \begin{tikzcd}[row sep=small, column sep=small]
      & {B_{0,-1}} && { B_{0,0}} \\
      {A_{0,-1}} && { A_{0,0}} \\
      & {B_{0,0}} && {B_{0,1}} \\
      {A_{0,0}} && {A_{0,1}} \\
      &&& {B_{0,0}} \\
      && {A_{0,0}}
      \arrow["{\ssplit}"{pos=0.8}, from=1-2, to=3-2]
      \arrow["{d_0}"', from=1-4, to=1-2]
      \arrow["{\ssplit}", from=1-4, to=3-4]
      \arrow["{{G'}}", from=2-1, to=1-2]
      \arrow["{\ssplit}"', from=2-1, to=4-1]
      \arrow[dashed, from=2-3, to=1-4]
      \arrow["{d_0}"'{pos=0.3}, from=2-3, to=2-1]
      \arrow["{\ssplit}"{pos=0.2}, from=2-3, to=4-3]
      \arrow["{{d_1}}"{pos=0.7}, from=3-4, to=3-2]
      \arrow["{{d_0}}", from=3-4, to=5-4]
      \arrow["{{G'}}", from=4-1, to=3-2]
      \arrow["{{G'}}"', from=4-3, to=3-4]
      \arrow["{{d_1}}", from=4-3, to=4-1]
      \arrow["{{d_0}}"', from=4-3, to=6-3]
      \arrow["{{G'}}"', from=6-3, to=5-4]
    \end{tikzcd}
  \end{center}
  Here all the solid squares involving the diagonal maps $G'$ commute since
  $G'$ is a map of $\Sigma$-presheaves. Since all rows are 1-Segal and bottom-split, the $\decbot$-coalgebra structure is rigid by Lemma~\ref{lem:1-Segal-is-rigid}. It follows that the front and back face of the cube are
  pullbacks, with the back face
  pullback inducing the dashed morphism. By construction this induced morphism is
  compatible with the augmentation map and with the extra bottom degeneracy $
  [0,0] \leftarrow [0,1] $. Composing the right face of the cube with the hanging
  square $G'd_\bot \simeq d_\bot G'$, as in the diagram, shows that the dashed
  morphism is necessarily $G'$, since the two vertical composites are the
  identities. Working instead with the pullback $ \bar{A}_{0,-1} \times_{\bar
  A_{0,0}} \Decbot (A_{0,\bullet}) $ of simplicial spaces and the corresponding one
  for $ B $, a similar argument allows us to conclude that $G'$ is in fact
  compatible with all extra bottom sections of the complete zeroth row as well as with
  the augmentation map in the zeroth row.

  Next we show that $G'$ is compatible with all the extra bottom sections in
  the bulk (which are obtained by pullback using upper stability). By
  assumption, the square
  \begin{equation} \label{diag:Fs=sF}
    \begin{tikzcd}
      {A_{0,0}} & {A_{0,1}} \\
      {B_{0,0}} & {B_{0,1}}
      \arrow["{\ssplit}", from=1-1, to=1-2]
      \arrow["G'"', from=1-1, to=2-1]
      \arrow["G'", from=1-2, to=2-2]
      \arrow["{\ssplit}"', from=2-1, to=2-2]
    \end{tikzcd}
  \end{equation}
  commutes. To start with, we claim that the same square in the row below 
  commutes too. To see this consider the two commutative diagrams
  \begin{center}
    \begin{tikzcd}
      {A_{1,0}} \\
      {A_{0,0}} & {A_{1,1}} & {A_{1,0}} \\
      & {A_{0,1}} & {\drpullback B_{1,1}} & {B_{1,0}} \\
      && {B_{0,1}} & {B_{0,0}}
      \arrow["{e_\bot}"', from=1-1, to=2-1]
      \arrow["{\ssplit}", from=1-1, to=2-2]
      \arrow[bend left=20, Rightarrow, no head, from=1-1, to=2-3]
      \arrow["{\ssplit}"', from=2-1, to=3-2]
      \arrow["{d_\bot}", from=2-2, to=2-3]
      \arrow["{e_\bot}", from=2-2, to=3-2]
      \arrow["G'", from=2-2, to=3-3]
      \arrow["G'", from=2-3, to=3-4]
      \arrow["G'"', from=3-2, to=4-3]
      \arrow["{d_\bot}", from=3-3, to=3-4]
      \arrow["{e_\bot}"', from=3-3, to=4-3]
      \arrow["{e_\bot}", from=3-4, to=4-4]
      \arrow["{d_\bot}"', from=4-3, to=4-4]
    \end{tikzcd}
    \qquad\qquad
    \begin{tikzcd}
      {A_{1,0}} \\
      {A_{0,0}} & {B_{1,0}} \\
      & {B_{0,0}} & {\drpullback B_{1,1}} & {B_{1,0}} \\
      && {B_{0,1}} & {B_{0,0}}
      \arrow["{e_\bot}"', from=1-1, to=2-1]
      \arrow["G'", from=1-1, to=2-2]
      \arrow["G'"', from=2-1, to=3-2]
      \arrow["{e_\bot}"', from=2-2, to=3-2]
      \arrow["{\ssplit}", from=2-2, to=3-3]
      \arrow[bend left=20, Rightarrow, no head, from=2-2, to=3-4]
      \arrow["{\ssplit}"', from=3-2, to=4-3]
      \arrow["{d_\bot}", from=3-3, to=3-4]
      \arrow["{e_\bot}"', from=3-3, to=4-3]
      \arrow["{e_\bot}", from=3-4, to=4-4]
      \arrow["{d_\bot}"', from=4-3, to=4-4]
    \end{tikzcd}
  \end{center}
  where the pullback is an upper stability square. By the commutativity of
  diagram~\eqref{diag:Fs=sF} the two diagonal maps $ A_{1,0} \to B_{1,1} $ agree,
  when composed with either of the two pullback projections. Thus, by the
  uniqueness in the universal property of the pullback, they must agree,
  proving the claim. The same argument can be used in any row and any column.
  Thus, starting with the compatibility of $ G' $ with the extra bottom sections
  in the zeroth row, we can deduce the compatibility of $ G' $ with all the extra
  bottom sections in all other rows by applying an inductive argument. We define
  $G$ on the bulk and on $A_{0,-1}$ to be equal to $G'$.

  Next we turn our attention to the augmentation column. Since the rows of $A$ and
  $B$ are (absolute) colimits, we can define the value of $G$ on the augmentation
  by the colimit-induced map as in the diagram
  \begin{center}
    \begin{tikzcd}
      & {B_{i,-1}} & {B_{i,0}} & {B_{i,1}} & {} \\
      {A_{i,-1}} & {A_{i,0}} & {A_{i,1}} & {}
      \arrow[shift left=2, bend right=30, shorten <=3pt, shorten >=3pt, from=1-2, to=1-3]
      \arrow[from=1-3, to=1-2]
      \arrow[bend right=30, shift left=1, shorten <=3pt, shorten >=3pt, from=1-3, to=1-4]
      \arrow[shift left, from=1-4, to=1-3]
      \arrow[shift right, from=1-4, to=1-3]
      \arrow["\cdots"{description}, draw=none, from=1-5, to=1-4]
      \arrow["G"{pos=0.3}, dashed, from=2-1, to=1-2]
      \arrow[shift left=2, bend right=30, shorten <=3pt, shorten >=3pt, from=2-1, to=2-2]
      \arrow["G"{pos=0.3}, from=2-2, to=1-3]
      \arrow[from=2-2, to=2-1]
      \arrow[bend right=30, shift left=1, shorten <=4pt, shorten >=4pt, from=2-2, to=2-3]
      \arrow["G"{pos=0.3}, from=2-3, to=1-4]
      \arrow[shift left, from=2-3, to=2-2]
      \arrow[shift right, from=2-3, to=2-2]
      \arrow["\cdots"{description}, draw=none, from=2-4, to=2-3]
    \end{tikzcd}
  \end{center}
  where $ i > 0 $. This automatically makes $G$ compatible with the augmentation
  map and its extra bottom section. Thus defined, $G$ will be compatible also with
  the simplicial operators between rows, that is, all the squares
  \begin{center}
    \begin{tikzcd}
      & {B_{i,-1}} \\
      {A_{i,-1}} & {B_{i+1,-1}} \\
      {A_{i+1,-1}}
      \arrow["G"{pos=0.3}, from=2-1, to=1-2]
      \arrow["{e_k}"', from=2-2, to=1-2]
      \arrow["{e_k}", from=3-1, to=2-1]
      \arrow["G"', from=3-1, to=2-2]
    \end{tikzcd}
  \end{center}
  commute for all $ 0 \le k \le \top $ and similarly for all degeneracies. This is
  because the same equation holds in the bulk and as a result, by the uniqueness
  of induced maps on colimits, also on the augmentation.

  Finally, the augmentation row can be addressed in a similar fashion as we did
  for the augmentation column. All in all we have extended $ G'$ to a map $ G : A
  \to B $, which by construction restricts to $ G' : \jj\upperstar(A) \to
  \jj\upperstar(B) $.
\end{proof}

Recall that $\pp: \Sigma\to\simplexcategory$ sends $[i,j]$ to 
$[i{+}1{+}j]$ and $[-1]$ to $[0]$, and that we have the factorization 
$\pp = \rr \circ \jj$ (see~\ref{bla:Abacus} and \ref{rem:qstar=Tot}).
The following corollary is the BOORS equivalence.

\begin{cor} [BOORS~\cite{Bergner-Osorno-Ozornova-Rovelli-Scheimbauer:1809.10924}]
  \label{cor:BOORS}
  The functor $\pp\upperstar : \twoSeg \longrightarrow \BOORS$ is an equivalence.
\end{cor}
\begin{proof}
  Combining Theorem~\ref{thm:ABCInv=2Seg} and Theorem~\ref{thm:j-starEquivalence} gives an equivalence
  $$
  \twoSeg \stackrel{\qq\lowerstar}{\longrightarrow} \Abcs^\simeq \stackrel{\jj\upperstar}{\longrightarrow} \BOORS.
  $$
  According to Remark~\ref{rem:qstar=Tot}, the functor $\qq\lowerstar : \twoSeg
  \to \Abcs^\simeq$ is equivalent to the total decalage $\rr\upperstar$, induced by
  $\rr : \DD \to \simplexcategory$, which when composed with $\jj\upperstar$ gives
  precisely $\pp\upperstar$.
\end{proof}

\subsection{More detailed comparison between \texorpdfstring{$\Sigma$}{Sigma}-presheaves and \texorpdfstring{$\DD$}{Curly-D}-presheaves}

The category $ \Sigma $ as well as the BOORS axioms are symmetric with respect to
the diagonal. The category $\DD$ on the other hand, is
asymmetric forcing the inclusion $\jj : \Sigma \hookrightarrow \DD$ to be asymmetric
as well. We now embrace the asymmetry and study the inclusion $\jj : \Sigma
\hookrightarrow \DD $ more carefully by factorizing it into smaller steps:
\begin{center}
    \begin{tikzcd}
      \Sigma \arrow[r, "\simeq"] & \psimpcat \sqcup_{\simplexcategory} 
	  (\simplexcategory{\times}\simplexcategory) \arrow[r, "\hh 
	  \sqcup_{\id} \id"] & \simplexcategory^\bb 
	  \sqcup_{\simplexcategory} (\simplexcategory{\times}\simplexcategory) \arrow[r, "\w"] & \mathcal{D}_{i \ge 0} \arrow[r] & \mathcal{D} ,
    \end{tikzcd}
\end{center}
where $\mathcal{D}_{i \ge 0}$ is $\DD$ with the augmentation 
row removed.

It follows from Proposition~\ref{prop:RigCoalgIsRigPoint} that
the map
\begin{center}
  \begin{tikzcd}[column sep=large]
    \PrSh(\simplexcategory^\bb) \times_{\PrSh(\simplexcategory)} \PrSh(\simplexcategory\times \simplexcategory
  ) \arrow[r, "\hh\upperstar \times_{\id} \id"] & \PrSh(\psimpcat) \times_{\PrSh(\simplexcategory)} \PrSh(\simplexcategory\times \simplexcategory
  )
  \end{tikzcd}
\end{center}
restricts to an equivalence after imposing the horizontal pointing axiom on the
domain and imposing rigidity on the first row on the codomain. By
Lemma~\ref{lem:1-Segal-is-rigid} the rigidity becomes automatic if we ask for the
rows to be 1-Segal (or at least the zeroth row). We now turn our attention to
$\w\upperstar$. Let $ \PrSh^{\operatorname{up \, st}}(\DD_{i\ge 0}) $ and $
\PrSh^{\operatorname{up \,
st}}(\simplexcategory^\bb\sqcup_{\simplexcategory}(\simplexcategory\times\simplexcategory))
$ be the full subcategories of the domain and the codomain of $\w\upperstar$
respectively consisting of the presheaves that are upper stable.

\begin{lemma}
  The functor
  $$
  \PrSh^{\operatorname{up \, st}}(\mathcal{D}_{i \ge 0}) 
  \overset{\w\upperstar}{\longrightarrow} \PrSh^{\operatorname{up \, st}}(\simplexcategory_\bb \sqcup_{\simplexcategory} 
    (\simplexcategory{\times}\simplexcategory))
  $$
  is an equivalence.
\end{lemma}

\begin{proof}
  This proof relies on similar arguments as those appearing in the proof of
  Theorem~\ref{thm:j-starEquivalence}.

  {\em Essential surjectivity:} Starting with a presheaf $B$ in the codomain of
  $\w\upperstar$, Lemma~\ref{lem:PropagateS-1} together with
  Remark~\ref{rem:PropagateS-1} imply that the extra bottom sections can be propagated
  down starting from the zeroth row, thus producing a $\DD_{\ge 0}$-presheaf
  with a pointing, where $\DD_{\ge 0}$ is $\DD$ but with both augmentations
  removed. Taking colimits row-wise (which are absolute) we fill in the
  augmentation column producing the desired $\DD_{i\ge 0}$-presheaf which
  restricts to $B$ along $\w$.

  {\em Fullness and Faithfulness:} The argument is the same as that in the proof
  of Theorem~\ref{thm:j-starEquivalence}.
\end{proof}

Putting everything together gives the following theorem.

\begin{theorem} \label{thm:Di=Sigma}
  The functor $ \PrSh(\DD_{i \ge 0}) \to \PrSh(\Sigma) $
  induced by the inclusion $ \Sigma \hookrightarrow \DD_{i \ge 0} $ restricts to an equivalence on the full subcategories
  \begin{center}
    \begin{tikzcd}
      \{ B \in \PrSh(\DD_{i \ge 0}) \, | \, \text{upper stable, Segal rows} \, \} \arrow[r, "\simeq"] & \left\{ B' \in \PrSh(\Sigma) \left| \, \begin{matrix}
        \text{upper stable, Segal rows,} \\ \text{horizontal pointing axiom}
      \end{matrix} \right. \right\}  .
    \end{tikzcd}
  \end{center}
\end{theorem}

\begin{remark}
  Note that Proposition~\ref{prop:colim-aug}, where we deduced the
  $2$-Segalness of the augmentations and the culfness of the augmentation maps
  from properties of the bulk, cannot be applied in Theorem~\ref{thm:Di=Sigma}, since
  we do not have the full stability (see also Remark~\ref{rem:colim-aug}). As a
  result, the $\DD_{i\ge 0}$-presheaves which are upper stable with Segal rows,
  appearing in Theorem~\ref{thm:Di=Sigma}, cannot be ensured to have 
  $2$-Segal augmentation column or culf augmentation map.
\end{remark}

\appendix


\newpage

\section{\texorpdfstring{$2$}{2}-Segal cheat sheet}

All the following are standard facts.

\begin{aplemma}\label{cheat:counits}
  The following are equivalent conditions on a simplicial space $X$:
\begin{dashlist}
  \item $X$ is $1$-Segal

  \item the counit $\varepsilon:\Dectop X \to X$ is a right fibration (meaning cartesian on
  $d_\bot$).

  \item the counit $\varepsilon:\Decbot X \to X$ is a left fibration (meaning cartesian on
  $d_\top$).
\end{dashlist}
\end{aplemma}

\begin{aplemma}[Cf.~{\cite[Prop.4.13]{Galvez-Kock-Tonks:1512.07573}}]
  \label{cheat:dec(culf)}
   If $F:Y\to X$ is culf, then 
\begin{dashlist}
   \item $\Dectop(F)$ is a left fibration.

   \item $\Decbot(F)$ is a right fibration.
\end{dashlist}
\end{aplemma}

\begin{aplemma}\label{Decbot(lfib)}
  Let $F: Y \to X$ be a map of simplicial spaces.
  \begin{dashlist}
    \item If $F$ is a left fibration, 
    then $\Decbot(F)$ is cartesian.
    \item If $F$ is a right fibration,
    then $\Dectop(F)$ is cartesian.
  \end{dashlist}
\end{aplemma}

\begin{aplemma}\label{fibover1Segal}
  If $Y\to X$ is a left or a right fibration and $X$ is $1$-Segal, then also $Y$
  is $1$-Segal.
\end{aplemma}

\begin{apdefinition}
  \label{upper-lower-2Segal}
  Recall that a simplicial space $Y$ is called {\em upper $2$-Segal} if
  $\Dectop(Y)$ is $1$-Segal. 
  In particular, the following square is then a pullback:
  \[
  \begin{tikzcd}
  Y_2 \ar[d, "d_0"']  & Y_3 \dlpullback \ar[l, "d_2"'] \ar[d, "d_0"]  \\
  Y_1 & Y_2  .  \ar[l, "d_1"]
  \end{tikzcd}
  \]

  Similarly, a simplicial space $Y$ is called {\em lower $2$-Segal} if
  $\Decbot(Y)$ is $1$-Segal. 
  In particular, the following square is then a pullback:
  \[
  \begin{tikzcd}
  Y_2 \ar[d, "d_2"']  & Y_3 \dlpullback  \ar[l, "d_1"'] \ar[d, "d_3"]  \\
  Y_1 & Y_2 .  \ar[l, "d_1"]
  \end{tikzcd}
  \]
\end{apdefinition}

\begin{aplemma}\label{cheat:culfover2Segal}\label{cheat:culf+effepi}
  Suppose $Y\to X$ is culf.
  \begin{dashlist}
    \item If $X$ is lower $2$-Segal, then also $Y$ is lower $2$-Segal.
    \item If $X$ is upper $2$-Segal, then also $Y$ is upper $2$-Segal.
  \end{dashlist}
  In particular, if $X$ is $2$-Segal, then also $Y$ is $2$-Segal.

  If $Y\to X$ culf and also an effective epimorphism, then the 
  converse of these implications also hold.

\end{aplemma}

\begin{aplemma}\label{cheat:counitculf}
  Let $Y$ be a simplicial space.
  \begin{dashlist}
    \item  If $Y$ is upper $2$-Segal, then the counit $\varepsilon : \Decbot(Y)\to Y$ is
  culf.
  
    \item If $Y$ is lower $2$-Segal, then the counit $\varepsilon : \Dectop(Y)\to
  Y$ is culf.
  \end{dashlist}
  In particular, if $Y$ is $2$-Segal, then both the counits $\Dectop
  Y\to Y$ and $\Decbot Y\to Y$ are culf~\cite{Galvez-Kock-Tonks:1512.07573}.

\end{aplemma}

\newpage

\hyphenation{mathe-matisk}

\medskip

\noindent {\sc Joachim Kock} \texttt{<joachim.kock@uab.cat>}\\
University of Copenhagen, Universitat Aut\`onoma de 
Barcelona, and Centre de Recerca Matem\`atica

\medskip

\noindent {\sc Thomas Jan Mikhail} \texttt{<tjm@math.ku.dk>}\\
Universitat Aut\`onoma de 
Barcelona \\
Current affiliation: University of Copenhagen

\end{document}